\documentclass [11pt,oneside,a4paper,mathscr]{amsart}

\usepackage{amscd,amsmath,amssymb,euscript}
\usepackage[frame,cmtip,curve,arrow,matrix,line,graph]{xy}
\usepackage{tikz-cd}
\usepackage{bigints}
\usetikzlibrary{matrix}
\usepackage{hyperref}
\usepackage{stmaryrd}

\title[Categorical and K-theoretic DT of $\mathbb{C}^3$ (part II)]{Categorical and K-theoretic Donaldson-Thomas theory of $\mathbb{C}^3$ (part II)}
\author{Tudor P\u adurariu and Yukinobu Toda}

\newtheorem{thm}{Theorem}[section]
\newtheorem{cor}[thm]{Corollary}
\newtheorem{conj}[thm]{Conjecture}
\newtheorem{prop}[thm]{Proposition}
\newtheorem{lemma}[thm]{Lemma}

\theoremstyle{definition}
\newtheorem{defn}[thm]{Definition}

\newtheorem{thm*}[thm]{Theorem$^*$}
\newtheorem{remark}[thm]{Remark}

\newcommand{\comment}[1]{}

\renewcommand{\leq}{\leqslant}
\renewcommand{\geq}{\geqslant}

\newcommand{\F}{\mathcal{F}}

\newcommand{\OO}{\mathcal{O}}

\newcommand{\X}{\mathcal{X}}

\newcommand{\Coh}{\operatorname{Coh}}

\newcommand{\Ext}{\operatorname{Ext}}
\newcommand{\Hom}{\operatorname{Hom}}

\newcommand{\Spec}{\operatorname{Spec}}

\newcommand{\C}{\mathbb{C}}

\newcommand{\ee}{\underline{e}}
\newcommand{\dd}{\underline{d}}

\newcommand{\Tr}{\mathop{\mathrm{Tr}}\nolimits}

\newcommand{\relmiddle}[1]{\mathrel{}\middle#1\mathrel{}}

\newcommand{\ssslash}{/\!\!/}

\makeatletter

\@addtoreset{equation}{section}	
\makeatother

\usetikzlibrary{positioning,fit,calc}
\tikzstyle{block}=[draw=black, width=1cm, minimum height=2cm, align=center] 
\tikzstyle{block2}=[draw=black, text width=2cm, minimum height=1cm, align=center] 
\tikzstyle{block3}=[draw=black, text width=2cm, minimum height=1cm, align=center] 

\begin{document}
\maketitle
\begin{abstract}
Quasi-BPS categories appear as summands in semiorthogonal decompositions of DT categories for 
Hilbert schemes of points in the three dimensional affine space and in the categorical Hall algebra of the two dimensional affine space. 
In this paper, we prove several properties of quasi-BPS categories analogous
to BPS sheaves in cohomological DT theory.

    We first 
    prove a categorical analogue of Davison's support lemma, namely that complexes in the quasi-BPS categories for  coprime length and weight are supported over the small diagonal in the symmetric product 
    of the three dimensional affine space.
    The categorical support lemma is used to determine the torsion-free
    generator of the torus equivariant K-theory of the quasi-BPS category of 
    coprime length and weight. 
    
    We next construct a bialgebra structure on the torsion free equivariant K-theory of quasi-BPS categories for a
    fixed ratio of length and weight. 
  We define the K-theoretic BPS space 
  as the space of primitive elements with respect to the coproduct. 
    We show that all localized equivariant K-theoretic BPS spaces are one dimensional, which is a K-theoretic analogue of the computation of (numerical) BPS invariants of the three dimensional affine space.  
\end{abstract}

\section{Introduction}

\subsection{Quasi-BPS categories}

In \cite{PT0, P, P2}, we studied quasi-BPS categories $\mathbb{S}(d)_w$ for $d\in \mathbb{N}$
(length) 
and $w\in \mathbb{Z}$ (weight)
in relation to categorical Donaldson-Thomas (DT) theory and to categorical Hall algebras (of surfaces and of
quivers with potentials). 
They are defined to be full subcategories
of the category of matrix factorizations
\begin{align}\label{intro:Sdw}
    \mathbb{S}(d)_w :=\mathrm{MF}(\mathbb{M}(d)_w, \Tr W_d)
    \subset \mathrm{MF}(\mathcal{X}(d), \Tr W_d),
\end{align}
where $\mathbb{M}(d)_w$ is a
twisted 
noncommutative resolution first considered by \v{S}penko--Van den Bergh~\cite{SVdB}, 
see Subsection~\ref{subsection:quasibps} for more details. 
Here, the stack $\mathcal{X}(d)$ and the regular function $\Tr W_d$
are given by 
\begin{align}\label{def:Xd}
    \mathcal{X}(d):=\Hom(V, V)^{\oplus 3}/GL(V), \ 
    \Tr W_d(A, B, C)=\Tr A[B, C],
\end{align}
where $V$ is a $d$-dimensional vector space. 

There is also a 
graded quasi-BPS category $\mathbb{S}^{\rm{gr}}(d)_v$, which is equivalent (via Koszul duality)
to a subcategory
\begin{align*}
    \mathbb{T}(d)_v \subset D^b(\mathscr{C}(d)),
\end{align*}
where $\mathscr{C}(d)$ is the derived moduli stack of zero-dimensional 
sheaves on $\mathbb{C}^2$
with length $d$. We will also consider $T$-equivariant versions of these categories, where
$T=(\mathbb{C}^*)^2$ is the Calabi-Yau torus of $\mathbb{C}^3$. 
We denote by $\mathbb{K}=K(BT)=\mathbb{Z}[q_1^{\pm 1}, q_2^{\pm 1}]$
and by $\mathbb{F}$ the fractional field of $\mathbb{K}$. 
The ($T$-equivariant or not) Grothendieck groups of $\mathbb{S}(d)_w$ and $\mathbb{S}^{\mathrm{gr}}(d)_w\cong \mathbb{T}(d)_w$ are isomorphic.

The purpose of this paper is to prove several 
properties of quasi-BPS categories 
analogous to BPS sheaves in cohomological DT theory, 
and use these properties to 
compute the $T$-equivariant K-theory of quasi-BPS categories.

\subsection{Semiorthogonal decompositions}
\label{subsection:introBPS}
We briefly review semiorthogonal decompositions 
with summands given by quasi-BPS categories proved in~\cite{PT0, P}. 

In our previous paper~\cite{PT0}, 
we constructed a semiorthogonal decomposition of 
the categorification $\mathcal{DT}(d)$
of the Donaldson-Thomas invariant $\mathrm{DT}_d$,
which is a virtual count of zero-dimensional closed subschemes in $\mathbb{C}^3$. 
The category $\mathcal{DT}(d)$ is defined by 
\[\mathcal{DT}\left(d\right):=\text{MF}\left(\text{NHilb}(d), \Tr W_d \right),\]
see \cite[Subsection 1.5]{PT0}. 
Here $\text{NHilb}(d)$ is the noncommutative Hilbert scheme of points 
\[\text{NHilb}(d):=\left(V\oplus \text{Hom}(V, V)^{\oplus 3}\right)^{\text{ss}}/GL(V).\] 
More precisely, in \cite[Theorem 1.1]{PT0} we showed that there is
a semiorthogonal decomposition
\begin{equation}\label{thm:1PT0}
    \mathcal{DT}(d)=
\left\langle \boxtimes_{i=1}^k \mathbb{S}(d_i)_{v_i+d_i\left(\sum_{i>j}d_j-\sum_{j>i}d_j\right)}
\relmiddle|
	\begin{array}{c}
	0\leq v_1/d_1 < \cdots < v_k/d_k<1 \\
	d_1+\cdots+d_k=d
	\end{array}
\right\rangle.\end{equation} 

There is also a semiorthogonal decomposition of $D^b(\mathscr{C}(d))$
in graded quasi-BPS categories, 
see~\cite[Theorem 1.1]{P}:
\begin{equation}\label{thm:1P}
D^b\left(\mathscr{C}(d)\right)=\left\langle \boxtimes_{i=1}^k \mathbb{T}(d_i)_{v_i} \relmiddle|
	\begin{array}{c}
	v_1/d_1 < \cdots < v_k/d_k \\
	d_1+\cdots+d_k=d
	\end{array} \right\rangle.
\end{equation}


\subsection{Numerical and cohomological BPS invariants}
It is expected that, for any smooth Calabi-Yau threefold $X$, there are certain deformation invariant integers called \textit{BPS invariants} which determine the Donaldson-Thomas and Gromov-Witten invariants of $X$, see \cite[Section 2 and a half]{MR3221298}. 
Denote by $\Omega_d$ the BPS invariants of $\mathbb{C}^3$ for $d\geq 1$. Then
\begin{equation}\label{Omegad}
\Omega_d=-1\text{ for all }d\geq 1.
\end{equation}
The wall-crossing formula for the DT invariants of $\mathbb{C}^3$, see~\cite[Section~6.3]{JS}, \cite[Remark~5.14]{T5}, says that
\begin{align}\label{wallcrossingDT}
	\sum_{d \geq 0} \mathrm{DT}_{d}q^d=
	\prod_{d\geq 1}(1-(-q)^d)^{d \Omega_d}=\prod_{d\geq 1}\frac{1}{\left(1-(-q)^d\right)^d},
	\end{align}
	see~\cite[Subsection 1.6]{PT0} for more details.

Davison--Meinhardt~\cite{DM} constructed a perverse sheaf $\mathcal{B}PS_d$ on $\mathrm{Sym}^d(\mathbb{C}^3)$, 
called \textit{BPS sheaf}, 
whose Euler characteristic recovers the BPS invariant $\Omega_d$. Davison \cite[Theorem 5.1]{Dav} showed that 
\begin{equation}\label{supportlemma}
    \mathcal{B}PS_d=\Delta_*\mathrm{IC}_{\mathbb{C}^3},
\end{equation} where $\Delta\colon \mathbb{C}^3\hookrightarrow \mathrm{Sym}^d(\mathbb{C}^3)$ is the small diagonal. Davison \cite[Lemma 4.1]{Dav} also proved restrictions on the support of BPS sheaves for tripled quivers with potentials and used \cite[Lemma 4.1]{Dav} to prove purity results about stacks of representations of preprojective algebras \cite[Theorem A]{Dav}.
We refer to \eqref{supportlemma} as Davison's support lemma. 
The (cohomological) BPS spaces are the cohomology of the BPS sheaf.

Properties (and computations in special cases) of BPS sheaves have applications in the study of Hodge theory of various cohomological DT spaces (in particular, in proving purity of the Borel-Moore homology of moduli of objects in K$3$ categories \cite{Davi}) and in the study of Cohomological Hall algebras \cite{D, KaVa2}. 

Our point of view is that the semiorthogonal decomposition (\ref{thm:1PT0}) may be
regarded as a categorification of the wall-crossing formula (\ref{wallcrossingDT})
and the category $\mathbb{S}(d)_w$ may be regarded as a categorical analogue of the
BPS invariant (\ref{Omegad}) or BPS sheaf (\ref{supportlemma}). 
In this paper, we 
make the above heuristic more rigorous. 
Let $(d, w)\in \mathbb{N}\times\mathbb{Z}$ be integers with $\gcd(d, w)=1$.
We first prove a categorical analogue of (\ref{supportlemma}), namely 
any object of $\mathbb{S}(d)_w$ is supported over the small diagonal 
in $\mathrm{Sym}^d(\mathbb{C}^3)$. 
Further, for $n\in \mathbb{N}$, we define a K-theoretic analogue of the cohomological BPS space:
\[\mathrm{P}(nd)_{nw}\subset K_T\left(\mathbb{S}(nd)_{nw}\right)\cong K_T\left(\mathbb{T}(nd)_{nw}\right)\] 
and show that $\mathrm{P}(nd)_{nw,\mathbb{F}}:=\mathrm{P}(nd)_{nw}\otimes_{\mathbb{K}}\mathbb{F}$ is a one dimensional $\mathbb{F}$-vector space, 
compare with \eqref{Omegad}.

\subsection{Support of matrix factorizations in quasi-BPS categories}\label{ss13bis}
We prove a version of the support lemma \eqref{supportlemma} for categories $\mathbb{S}(d)_w$ with $\gcd(d, w)=1$.
We consider the quotient stack $\mathcal{X}(d)$
defined in (\ref{def:Xd}) together
with its good moduli space 
\begin{align*}
    \pi \colon \mathcal{X}(d) \to X(d):=\Hom(V, V)^{\oplus 3} \sslash GL(V). 
\end{align*}
Consider the diagram 
\begin{align*}\label{dia:Cohsym}
	\xymatrix{
		\mathcal{C}oh(\mathbb{C}^3, d) \ar@{=}[r] \ar[d]_-{\pi} & \mathrm{Crit}(\Tr W_d)  \ar@<-0.3ex>@{^{(}->}[r]&
		\mathcal{X}(d) \ar[rd]^-{\Tr W_d} \ar[d]_-{\pi} & \\
		\mathrm{Sym}^d(\mathbb{C}^3)  \ar@<-0.3ex>@{^{(}->}[rr]& & X(d) \ar[r] & \mathbb{C},
}
	\end{align*}
	where $\mathcal{C}oh(\mathbb{C}^3, d)$ is the stack of sheaves with zero dimensional support and length $d$ on $\mathbb{C}^3$.
In Section \ref{s3}, we prove the following 
(see Theorem~\ref{lem:support} for a more precise statement):

\begin{thm}\emph{(Theorem~\ref{lem:support})}\label{intro:thm1}
    For a pair $(d, w) \in \mathbb{N}\times\mathbb{Z}$ with $\gcd(d, w)=1$, 
any object in $\mathbb{S}(d)_w$ is supported on 
$\pi^{-1}(\Delta)$. 
\end{thm}

The categorical support restriction in Theorem~\ref{intro:thm1}
implies a strong constraint on $T$-equivariant K-theory classes of objects of
quasi-BPS categories. 
Let $\Theta$ be the forget-the-potential map
\begin{equation}\label{definition:forgetpotential}
    \Theta\colon K_T(\mathrm{MF}(\X(d), \Tr W_d))\to K_T(\mathrm{MF}(\X(d), 0))=
    \mathbb{K}[z_1^{\pm 1}, \ldots, z_d^{\pm 1}]^{\mathfrak{S}_d},
\end{equation} see \cite[Proposition 3.6]{P0}.
In Lemma \ref{lemma:div}, we show that if a complex $\mathcal{F}$ is supported on 
$\pi^{-1}(\Delta)$, then $\Theta([\mathcal{F}])$ is divisible by 
the element 
\begin{align}\label{divisible}
(q_1-1)^{d-1}(q_2-1)^{d-1}(q_1 q_2-1)^{d-1} \in \mathbb{K}.
\end{align}
In \cite{PT0}, we introduced certain complexes $\mathcal{E}_{d, w}\in \mathbb{T}(d)_w$ using a derived stack of pairs of commuting matrices, both of which have spectrum of cardinality one, and have an 
explicit shuffle description \eqref{elem:E}.
In particular, Lemma \ref{lemma:div} applies to $\mathcal{E}_{d,w}$ for $\gcd(d, w)=1$. 
We remark that the
divisibility by (\ref{divisible}) 
is not obvious from the shuffle description of $\Theta([\mathcal{E}_{d,w}])$.

\subsection{Integral generator of K-theory of quasi-BPS categories}
In \cite[Theorem 1.2]{PT0}, we showed that the localized (i.e. taking 
$\otimes_{\mathbb{K}} \mathbb{F}$)
$T$-equivariant K-theory of $\mathbb{T}(d)_{w}$ (and thus of $\mathbb{S}(d)_{w}$) is generated by monomials in $[\mathcal{E}_{d', w'}]$ for $d'\leq d$ and $\frac{w}{d}=\frac{w'}{d'}$. 
The divisibility by (\ref{divisible}) will be useful to compute the $T$-equivariant K-theory 
of quasi-BPS categories \textit{without} localization. 
For a $\mathbb{K}$-module $M$, we will use the notation $M':=M/(\mathbb{K}\text{-torsion})$. 
 Using Lemma \ref{lemma:div} and Theorem \ref{intro:thm1}, we show that:
\begin{thm}\emph{(Theorem~\ref{thm:generatordv})}\label{intro:thm2}
Let $(d, w) \in \mathbb{N}\times\mathbb{Z}$ with $\gcd(d, w)=1$.
The $\mathbb{K}$-module $K_T\left(\mathbb{S}(d)_w\right)'$ is free of rank one 
with generator $\left[\mathcal{E}_{d, w}\right]$.
\end{thm}

For $(d, w)\in \mathbb{N}\times \mathbb{Z}$ with $\gcd(d, w)=1$,
we believe the category $\mathbb{T}(d)_w$ is generated by $\mathcal{E}_{d,w}$, which in particular implies Theorem \ref{intro:thm2}. Further, for $\ast\in \{\emptyset, T\}$, we suspect there are equivalences $\mathbb{T}_\ast(1)_0\xrightarrow{\sim} \mathbb{T}_\ast(d)_w$, but we do not have much evidence supporting this belief.

\subsection{The coproduct}

For the reminder of the introduction, we consider a pair $(d,v)\in\mathbb{N}\times\mathbb{Z}$ with $\gcd(d, v)=1$. 
The Grothendieck group of the category $\mathbb{T}(nd)_{nv}$ for $n>1$ contains a contribution from partitions $a+b=n$ because the Hall product restricts to a functor
\[m_{a,b}\colon \mathbb{T}(ad)_{av}\otimes \mathbb{T}(bd)_{bd}\to \mathbb{T}(nd)_{nv}\] for $a, b\geq 1$ with $a+b=n$, see \cite[Lemma 4.8]{PT0}. The category:
\begin{equation}\label{ddv}
\mathcal{D}_{d, v}:=\bigoplus_{n\geq 0}\mathbb{T}(nd)_{nv}
\end{equation}
is thus monoidal.
For $n=a+b$, there is a coproduct 
\begin{equation}\label{def:introcoproduct}
\Delta_{a, b}\colon \mathbb{T}(nd)_{nv}\to \mathbb{T}(ad)_{av}\otimes \mathbb{T}(bd)_{bv},
\end{equation} 
 see also \cite[Section 5]{P}. 
 The construction also provides a $T$-equivariant version.
 In Section \ref{s4}, we prove that the Hall product is compatible with the above coproduct:

\begin{thm}\emph{(Corollary~\ref{cor:44})}
\label{intro:thm3}
    Let $(d, v)\in \mathbb{N}\times\mathbb{Z}$ be coprime,  let $a, b, c, e, n\in\mathbb{N}$ such that $a+b=c+e=n$, and let $S$ be the set of tuples $(f_1, f_2, f_3, f_4)$ such that $a=f_1+f_2$, $b=f_3+f_4$, $c=f_1+f_3$, $e=f_2+f_4$.
    The following diagram commutes:
    \begin{equation*}
    \begin{tikzcd}
    K_T(\mathbb{T}(ad)_{av})'\otimes K_T(\mathbb{T}(bd)_{bv})'\arrow[r, "m_{a, b}"]\arrow[d, "\Delta"]& K_T(\mathbb{T}(nd)_{nv})'\arrow[d, "\Delta_{c, e}"]\\
    \bigoplus_{S} \bigotimes_{i=1}^4 K_T(\mathbb{T}(f_id)_{f_iv})'\arrow[r, "m'"]& K_T(\mathbb{T}(cd)_{cv})\otimes K_T(\mathbb{T}(ed)_{ev})',
    \end{tikzcd}
\end{equation*} where we have used the shorthand notations $\Delta:=\bigoplus_S \Delta_{f_1, f_2}\otimes \Delta_{f_3, f_4}$ and $m':=\bigoplus_S (m\otimes m)(1\otimes \text{sw}_{f_2, f_3}\otimes 1)$.
\end{thm}

In the proof, 
we first construct a version of $\Delta_{a, b}$ for 
the subcategories $\mathbb{M}(nd)_{nv}\subset D^b(\X(nd))_{nv}$, see (\ref{intro:Sdw}). We then construct the functor \eqref{def:introcoproduct} by applying matrix factorizations and using the Koszul 
equivalence. 

In Propositions \ref{commutative} and \ref{1236bis}, we show that the product and the coproduct on $K_T(\mathcal{D}_{d,v})'$ are commutative and cocommutative. We then obtain an isomorphism
\begin{equation}\label{isomac}
K_T(\mathcal{D}_{d, v})_\mathbb{F}\cong\Lambda_{\mathbb{F}},
\end{equation}
where $\Lambda_{\mathbb{F}}$ is the $\mathbb{F}$-algebra of symmetric functions, see Subsection \ref{subsection:symmetric} for its definition. The elementary functions $e_n\in \Lambda_{\mathbb{F}}$ are sent, up to a factor in $\mathbb{F}$, to $\mathcal{E}_{nd, nv}$. 

\subsection{K-theoretic BPS spaces}\label{ss5}

We define the K-theoretic BPS space to be the space of primitive elements of $K_T(\mathbb{T}(nd)_{nv})$ with respect to the above coproduct:
\begin{equation}\label{def:KBPS}
\mathrm{P}(nd)_{nv}:=\mathrm{ker}\left(\bigoplus_{\substack{a+b=n\\a, b\geq 1}}\Delta_{a, b}: K_T\left(\mathbb{T}(nd)_{nv}\right)\to \bigoplus_{\substack{a+b=n\\a, b\geq 1}}K_T\left(\mathbb{T}(ad)_{av}\otimes \mathbb{T}(bd)_{bv}\right)\right).
\end{equation}
We show that 
the dimension over $\mathbb{F}$ of localized K-theoretic BPS spaces for all pairs $(nd, nv)\in \mathbb{N}\times\mathbb{Z}$ is the same as the (numerical) BPS invariants \eqref{Omegad} up to a sign:
\begin{prop}\label{prop14}\emph{(Corollary~\ref{isolambda})}
The $\mathbb{F}$-vector space $\mathrm{P}(nd)_{nv, \mathbb{F}}$ is one dimensional.
\end{prop}

Using the above proposition and \cite[Theorem 1.1]{PT0}, 
we obtain a K-theoretic analogue of the wall-crossing formula \eqref{wallcrossingDT}, 
see Subsection \ref{subsection:primitive} and \cite[Subsection 1.6]{PT0} for more details.
\begin{cor}\label{intro:corDT}\emph{(Corollary~\ref{corDT})}
There is an isomorphism of $\mathbb{N}$-graded $\mathbb{F}$-vector spaces:
\begin{equation}\label{catwc}
\bigoplus_{d\geq 0}K_T(\mathcal{DT}(d))_{\mathbb{F}}\cong \bigotimes_{\substack{0\leq v<d \\ \gcd(d,v)=1}} \left(\bigotimes_{n\geq 1}\mathrm{Sym}\big(\mathrm{P}(nd)_{nv, \mathbb{F}}\big)\right),
\end{equation}
\end{cor}

We conjecture an integral version of Proposition \ref{prop14}, which is a version of Theorem \ref{intro:thm2} for pairs $(nd, nv)$ for all $n\in\mathbb{N}$:
\begin{conj}
The $\mathbb{K}$-module $\mathrm{P}(nd)_{nv}$ is free of rank one.
\end{conj}

The torsion-free version of the above conjecture for $(d,v)=(1,0)$ and $n=2$ follows from the discussion in Subsection \ref{subsection:integral20}. 
We finally conjecture an analogue of Theorem \ref{intro:thm1} for all $n\in\mathbb{N}$.

\begin{conj}
The subspace $\mathrm{P}(nd)_{nv}\subset K_T(\mathbb{T}(nd)_{nv})\cong K_T(\mathbb{S}(nd)_{nv})$ is supported over $\pi^{-1}(\Delta)$, alternatively, the following composition is zero:
\[\mathrm{P}(nd)_{nv}\to  K_T(\mathbb{S}(nd)_{nv})\to K_T(\X(nd)\setminus \pi^{-1}(\Delta)).\]
\end{conj}


\subsection{Acknowledgements}
We thank Tasuki Kinjo, Davesh Maulik, Andrei Negu\c{t}, 
Raphaël Rouquier, Christopher Ryba,
Olivier Schiffmann, {\v S}pela {\v S}penko, Eric Vasserot, and Yu Zhao
for discussions related to this work. We thank the referee for useful suggestions.
	Y.~T.~is supported by World Premier International Research Center
	Initiative (WPI initiative), MEXT, Japan, and Grant-in Aid for Scientific
	Research grant (No.~19H01779) from MEXT, Japan.

\section{Preliminaries}\label{s2}

\subsection{Notations}\label{subsection:notation}
The spaces considered in this paper are defined 
over the complex field $\mathbb{C}$ and they are quotient stacks $\X=A/G$, where $A$ is a dg scheme, the derived zero locus of a section $s$ of a finite rank bundle vector bundle $\mathcal{E}$ on a finite type separated scheme $X$ over $\mathbb{C}$, and $G$ is a reductive group. For such a dg scheme $A$, let $\dim A:=\dim X-\text{rank}(\mathcal{E})$,
let $A^{\mathrm{cl}}:=Z(s)\subset X$ be the (classical) zero locus, and let $\X^{\mathrm{cl}}:=A^{\mathrm{cl}}/G$. We denote by $\mathcal{O}_{\X}$ or $\mathcal{O}_A$ the structure sheaf of $\X$. 

For $G$ a reductive group and $A$ a dg scheme as above, denote by $A/G$ the corresponding quotient stack and by $A\ssslash G$ the quotient dg scheme with dg-ring of regular functions $\mathcal{O}_A^G$.

For a classical stack $\X$ with a morphism $\X \to X$
to a scheme $X$ and for a closed point $p\in X$, we denote by 
$\X_p$
its formal fiber
$\X_p := \X \times_X \Spec \widehat{\mathcal{O}}_{X, p}$. 

For a (derived) stack $\mathcal{X}$ with an action of a torus $T$, we denote by $D^b_T(\mathcal{X})$
the bounded derived category of $T$-equivariant coherent sheaves on $\mathcal{X}$ and by $G_T(\X)$ its Grothendieck group. 
We denote by $\mathrm{Perf}_T(\mathcal{X})$ the subcategory of $T$-equivariant perfect complexes and by $K_T(\X)$ its Grothendieck group.
When $T$ is trivial, we drop $T$ from the notation of the above Grothendieck groups.
We introduce more notations for categories and K-theory in Subsection \ref{gradingMF}.

Consider the two dimensional torus
\begin{align}\label{torus:T}
(\mathbb{C}^{\ast})^2
\stackrel{\cong}{\to}
    T:=\{(t_1, t_2, t_3) \in (\mathbb{C}^{\ast})^{\times 3} \mid t_1 t_2 t_3=1\}\subset (\mathbb{C}^{\ast})^3. 
\end{align}
The isomorphism above is given by $(t_1, t_2) \mapsto (t_1, t_2, t_1^{-1} t_2^{-1})$. We denote by $\mathbb{K}:=K_T(\mathrm{pt})=\mathbb{Z}[q_1^{\pm 1}, q_2^{\pm 1}]$ and let $\mathbb{F}$ be the fraction field of $\mathbb{K}$. For $V$ a $\mathbb{K}$-module, we use the notations $V':=V/(\mathbb{K}\text{-torsion})$ and $V_{\mathbb{F}}:=V\otimes_{\mathbb{K}}\mathbb{F}$.

For a dg-category $\mathcal{D}$, a full dg-subcategory $\mathcal{C} \subset \mathcal{D}$
is called dense if any object in $\mathcal{D}$ is a direct summand of an object in $\mathcal{C}$. 

\subsection{Weights and partitions} 
\subsubsection{}\label{sss1}

For $d\in \mathbb{N}$, let $V$ be a $\mathbb{C}$-vector space of dimension $d$, and let $\mathfrak{g}=\mathfrak{gl}(V):=\text{End}(V)$. When the dimension is clear from the context, we drop $d$ from its notation.
Let
\[ \pi\colon \X(d):= R(d)/G(d):=\mathfrak{gl}(V)^{\oplus 3}/GL(V)\to X(d):=\mathfrak{gl}(V)^{\oplus 3}\ssslash GL(V).\]
Alternatively, $\X(d)$ is the stack of representations of dimension $d$ of the quiver $Q$ with one vertex and three loops $\{x, y, z\}$:
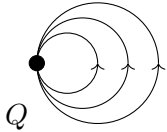
\begin{align}\notag
Q 
\begin{tikzpicture}
\draw[->] (0, 0) arc (-180:0:0.4) ;
\draw (0.8, 0) arc (0:180:0.4);
\draw[->] (0, 0) arc (-180:0:0.6) ;
\draw (1.2, 0) arc (0:180:0.6);
\draw[->] (0, 0) arc (-180:0:0.8);
\draw (1.6, 0) arc (0:180:0.8);
\draw[fill=black] (0, 0) circle (0.1);
\end{tikzpicture}
\end{align}
Consider the super-potential $W=x[y, z]$ of $Q$ and the regular function 
\begin{equation}\label{equation:NHilbTrW}
\Tr W= \Tr W_d:=\Tr A[B, C] \colon \X(d)\to \mathbb{C},
\end{equation}
where $(A, B, C)\in \mathfrak{gl}(V)^{\oplus 3}$.

 Fix the maximal torus 
 $T(d)\subset GL(d)$
 to be consisting of diagonal matrices. 
Denote by $M=\oplus_{i=1}^d \mathbb{Z} \beta_i$ the weight space of $T(d)$ and let $M(d)_{\mathbb{R}}:=M(d)\otimes_{\mathbb{Z}}\mathbb{R}$, where 
$\beta_1,\ldots, \beta_d$ is the set of simple roots.  
A weight $\chi=\sum_{i=1}^d c_i\beta_i$ is dominant 
if 
\begin{align*}
c_1\leq\ldots\leq c_d, \ 
\end{align*}
We denote by $M^+\subset M$ and $M^+_{\mathbb{R}}\subset M_{\mathbb{R}}$ the dominant chambers. When we want to emphasize the dimension vector, we write $M(d)$ etc. Denote by $N$ the coweight lattice of $T(d)$ and by $N_{\mathbb{R}}:=N\otimes_{\mathbb{Z}}\mathbb{R}$. Let $\langle\,,\,\rangle$ be the natural pairing between $N_{\mathbb{R}}$ and $M_{\mathbb{R}}$.

Let $W=\mathfrak{S}_d$ be the Weyl group of $GL(d)$. For $\chi\in M(d)^+$, let $\Gamma_{GL(d)}(\chi)$ be the irreducible representation of $GL(d)$ of highest weight $\chi$. We drop $GL(d)$ from the notation if the dimension vector $d$ is clear from the context. 
Let $w*\chi:=w(\chi+\rho)-\rho$ be the Weyl-shifted action of $w\in W$ on $\chi\in M(d)_\mathbb{R}$. We denote by $\ell(w)$ the length of $w\in W$.

Denote by $\mathcal{W}$ the multiset of $T(d)$-weights of $R(d)$. If there is a natural action of a torus $T$ on $R(d)$, we abuse notation and write $\mathcal{W}$ for the multiset of $(T\times T(d))$-weights of $R(d)$.
For $\lambda$ a cocharacter of $T(d)$, we denote by $N^{\lambda>0}$ the sum of weights $\beta$ in $\mathcal{W}$ such that $\langle \lambda, \beta\rangle>0$.

\subsubsection{}
We denote by
$\rho$ half the sum of positive roots of $GL(d)$. 
In our convention of the dominant chamber, 
it is given by 
\begin{align*}
    \rho=\frac{1}{2}\sum_{j<i}(\beta_i-\beta_j). 
\end{align*}
We denote by $1_d:=z\cdot\text{Id}$ the diagonal cocharacter of $T(d)$. 
Define the weights
\begin{align*}
    \sigma_d:=\sum_{j=1}^d\beta_j\in M,\
    \tau_d:=\frac{\sigma_d}{d}\in M_{\mathbb{R}}. 
\end{align*}


\subsubsection{}\label{ss13} 
Let $G$ be a reductive group (in this paper, $G$ will be a Levi subgroup of $GL(d)$ for some positive integer $d$), let $T_G$ be a maximal torus of $G$, 
let $X$ be a $G$-representation, and let
\[\X=X/G\] be the corresponding quotient stack. Let $\mathcal{W}$ be the multiset of $T_G$-weights of $X$.
For $\lambda$ a cocharacter of $T_G$, let $X^\lambda\subset X$ be the subspace generated by weights $\beta\in \mathcal{W}$ such that $\langle \lambda, \beta\rangle=0$, let $X^{\lambda\geq 0}\subset X$ be the subspace generated by weights $\beta\in \mathcal{W}$ such that $\langle \lambda, \beta\rangle\geq 0$, and let $G^\lambda$ and $G^{\lambda\geq 0}$ be the Levi and parabolic groups associated to $\lambda$.
Consider the fixed and attracting stacks
\begin{align*}
    \X^\lambda:=X^\lambda/ G^\lambda,\
    \X^{\lambda\geq 0}:=X^{\lambda\geq 0}/G^{\lambda\geq 0}
\end{align*}
with maps
\[\X^\lambda\xleftarrow{q_\lambda}\X^{\lambda\geq 0}\xrightarrow{p_\lambda}\X.\]
Define the integer 
\begin{equation}\label{def:nlambda}
n_\lambda:=\langle \lambda, [X^{\lambda\geq 0}]-[\mathfrak{g}^{\lambda\geq 0}]\rangle.
\end{equation}

\subsubsection{}\label{paco} 
Let $d\in \mathbb{N}$ and recall the definition of $\X(d)$ from Subsection \ref{sss1}.
For a cocharacter $\lambda:\C^*\to T(d)$, consider the maps of fixed and attracting loci
\begin{equation}\label{e}
\X(d)^\lambda \xleftarrow{q_\lambda}\X(d)^{\lambda\geq 0}\xrightarrow{p_\lambda}\X(d).\end{equation}
We say that two cocharacters $\lambda$ and $\lambda'$ are equivalent and write $\lambda\sim\lambda'$ if $\lambda$ and $\lambda'$ have the same fixed and attracting stacks as above. 

 We call $\dd:=(d_i)_{i=1}^k$ a partition of $d$ if $d_i\in\mathbb{N}$ are all non-zero and $\sum_{i=1}^k d_i=d$. 
 In Section \ref{s4}, we allow partitions $\dd=(d_i)_{i=1}^k$ to have terms $d_i$ equal to zero.
 We similarly define partitions of $(d,w)\in\mathbb{N}\times\mathbb{Z}$.
For a cocharacter $\lambda$ of $T(d)$, there is an associated partition $(d_i)_{i=1}^k$ such that
\[\X(d)^{\lambda\geq 0}\xrightarrow{q_\lambda} \X(d)^\lambda\cong\times_{i=1}^k\X(d_i).\]
 Define the length $\ell(\lambda):=k$. 

Equivalence classes of antidominant cocharacters are in bijection with ordered partitions $(d_i)_{i=1}^k$ of $d$.
For an ordered partition $\dd=(d_i)_{i=1}^k$ of $d$, fix a corresponding antidominant cocharacter $\lambda=\lambda_{\dd}$ of $T(d)$ which induces the maps
\[\X(d)^\lambda\cong\times_{i=1}^k\X(d_i)
\xleftarrow{q_\lambda}\X(d)^{\lambda\geq 0}
\xrightarrow{p_\lambda}\X(d).\] 
We also use the notations $p_\lambda=p_{\dd}$, $q_\lambda=q_{\dd}$.
The categorical Hall algebra is given 
by the functor 
$p_{\lambda*}q_\lambda^*=p_{\dd*}q_{\dd}^*$:
\begin{align}\label{prel:hall}
    m=m_{\dd}
     \colon D^b(\mathcal{X}(d_1)) \boxtimes
     \cdots \boxtimes D^b(\mathcal{X}(d_k))
     \to D^b(\mathcal{X}(d)). 
\end{align}
We may drop the subscript $\lambda$ or $\dd$ in the functors $p_*$ and $q^*$ when the cocharacter $\lambda$ or the partition $\dd$ is clear. We also use the notation $\ast$ for the Hall product.

\subsubsection{}\label{id} Let $(d_i)_{i=1}^k$ be a partition of $d$. There is an identification \[\bigoplus_{i=1}^k M(d_i)\cong M(d),\] where the simple roots $\beta_j$ in $M(d_1)$ correspond to the first $d_1$ simple roots $\beta_j$ of $d$ etc.

\subsubsection{}\label{compa}
Let $\underline{e}=(e_i)_{i=1}^l$ and $\dd=(d_i)_{i=1}^k$ be two partitions of $d\in \mathbb{N}$. We write $\ee\geq\dd$
if there exist integers \[a_0=0< a_1<\cdots<a_{k-1}\leq a_k=l\] such that for any $0\leq j\leq k-1$, we have
\[\sum_{i=a_{j}+1}^{a_{j+1}} e_i=d_{j+1}.\]
We say $\underline{e}$ is a refinement of $\dd$.
There is a similarly defined order on pairs $(d, w)\in\mathbb{N}\times\mathbb{Z}$.

\subsubsection{}\label{prime} 

Let $A$ be a partition $(d_i, w_i)_{i=1}^k$ of $(d, w)$ and consider its corresponding antidominant cocharacter $\lambda$. Define the weights
\begin{align*}
    \chi_A:=\sum_{i=1}^k w_i\tau_{d_i},\
    \chi'_A:=\chi_A+\mathfrak{g}^{\lambda>0}.
\end{align*}
Consider weights $\chi'_i\in M(d_i)_{\mathbb{R}}$ such that
\[\chi'_A=\sum_{i=1}^k \chi'_i.\]
 Let $v_i$ be the sum of coefficients of $\chi'_i$ for $1\leq i\leq k$; alternatively, $v_i:=\langle 1_{d_i}, \chi'_i\rangle$. We denote the above transformation by
  \begin{align}\label{trans:A}
     A\mapsto A', \ 
     (d_i, w_i)_{i=1}^k\mapsto (d_i, v_i)_{i=1}^k.
 \end{align}
Explicitly, the weights $v_i$ for $1\leq i\leq k$ are given by 
\begin{align}\label{w:prime}
    v_i=w_i+d_i\left(\sum_{j>i}d_j -\sum_{j<i}d_j \right). 
\end{align}




\subsection{Polytopes}\label{ss1}
The polytope $\textbf{W}(d)$ is defined as
\begin{equation}\label{W}
    \textbf{W}(d):=\frac{3}{2}\text{sum}[0, \beta_i-\beta_j]+\mathbb{R}\tau_d\subset M(d)_{\mathbb{R}},
    \end{equation}
where the Minkowski sum is after all $1\leq i, j\leq d$. For $w\in\mathbb{Z}$,
consider the hyperplane: 
\begin{equation}\label{W0}
    \textbf{W}(d)_w:=\frac{3}{2}\text{sum}[0, \beta_i-\beta_j]+w\tau_d\subset \textbf{W}(d).
    \end{equation}
    For $r>0$ and $\lambda$ a cocharacter of $T(d)$, let $F_r(\lambda)$ be the face of the polytope $2r\textbf{W}(d)$ corresponding to the cocharacter $\lambda$, so the set of weights $\chi$ in $M(d)_\mathbb{R}$ such that 
    \begin{align*}
        \chi\in 2r\textbf{W}(d), \
        \langle \lambda, \chi\rangle=r\langle \lambda, R(d)^{\lambda>0}\rangle.
    \end{align*}
    When $r=\frac{1}{2}$, we use the notations $F(\lambda)$. 
    For $\chi\in M(d)^+_{\mathbb{R}}$, its $r$-invariant $r(\chi)$ is the smallest real number $r$ such that
    \[\chi\in 2r\textbf{W}(d).
    \] 
For a cocharacter $\lambda$ of $T(d)$, denote by 
\[\textbf{W}(\lambda)_0:=\frac{3}{2}\text{sum}[0,\beta_i-\beta_j]\subset M(d)_{\mathbb{R}},\] where the sum is after all weights $1\leq i, j\leq d$ such that $\langle \lambda, \beta_i-\beta_j\rangle=0$.

\subsection{A corollary of the Borel-Weyl-Bott theorem}

For future reference, we state a result from \cite[Section 3.2]{hls}. We continue with the notations from Subsection \ref{ss13}. 
Let $M$ be the weight lattice of $T_G$.
We assume that $X$ is a symmetric $G$-representation, meaning that for any weight $\beta$ of $X$, the weights $\beta$ and $-\beta$ appear with the same multiplicity in $X$. 
Let $\chi$ be a weight in $M$. Let $\chi^+$ be the dominant Weyl-shifted conjugate of $\chi$ if it exists, and zero otherwise.
For a multiset $J\subset \mathcal{W}$, let
\[\sigma_J:=\sum_{\beta\in J}\beta.\]
Let
$w$ be the element of the Weyl group of minimal length such that $w*(\chi-\sigma_J)$ is dominant or zero. We let $\ell(J):=\ell(w)$.

\begin{prop}\label{bbw}
Let $X$ be a symmetric $G$-representation, and
let $\lambda$ be an antidominant cocharacter of $T_G$. 
Recall the fixed and attracting stacks and the corresponding maps
\[X^\lambda/G^\lambda\xleftarrow{q_\lambda}X^{\lambda\geq 0}/G^{\lambda\geq 0}\xrightarrow{p_\lambda}X/G.\]
For a weight $\chi$ in $M$, there 
is a quasi-isomorphism
\[\left(\bigoplus_{J}\mathcal{O}_{X}\otimes \Gamma_{G}\left((\chi-\sigma_J)^+\right)\left[|J|-\ell(J)\right], d\right)\xrightarrow{\sim}p_{\lambda*}q_{\lambda}^*\left(\mathcal{O}_{X^\lambda}\otimes\Gamma_{G^\lambda}(\chi)\right),\] where the complex on the left hand side has terms (shifted) vector bundles for all multisets $J\subset \{\beta\in\mathcal{W} \mid \langle \lambda, \beta\rangle<0\}$. 
\end{prop}

\subsection{Matrix factorizations}\label{MF}
In the notations from Subsection \ref{sss1}, 
we denote by 
\begin{align*}
    \mathrm{MF}(\X(d), \Tr W_d)
    \end{align*}
the dg-category of matrix factorizations of the regular function $\Tr W_d$
on the smooth stack $\X(d)$. 
Its objects consist of tuples
\begin{align*}
   \left(\alpha \colon F\rightleftarrows G\colon \beta\right)\text{ such that} \ 
   \alpha \circ \beta=\beta \circ \alpha=\cdot \Tr W_d, \ 
\end{align*}
where $F, G \in \Coh(\X(d))$, see \cite[Subsection 2.6]{PT0} for details. 
For an object $\mathcal{F} \in \mathrm{MF}(\X(d), \Tr W_d)$, its 
internal homomorphism $R\mathcal{H}om (\mathcal{F}, \mathcal{F})$
is an object of the $\mathbb{Z}/2$-graded derived category of $\Coh(\X(d))$. 
The support of $\mathcal{F}$
\begin{align*}
    \mathrm{Supp}(\F) \subset \mathcal{X}(d)
\end{align*}
is defined to be the support of $R\mathcal{H}om(\mathcal{F}, \mathcal{F})$, which is a 
closed substack of $\X(d)$. 
Alternatively, it is the smallest closed substack 
$\mathcal{Z} \subset \X(d)$
such that $\mathcal{F}|_{\mathcal{X}(d) \setminus \mathcal{Z}} \cong 0$
in $\mathrm{MF}(\X(d) \setminus \mathcal{Z}, \Tr W_d)$. 

Similarly to (\ref{prel:hall}), for $d=d_1+d_2$ 
we have the categorical Hall product 
\begin{align*}
    m=m_{d_1,d_2}\colon  \mathrm{MF}(\X(d_1), \Tr W_{d_1}) \boxtimes  \mathrm{MF}(\X(d_2), \Tr W_{d_2})
    \to  \mathrm{MF}(\X(d), \Tr W_d),
\end{align*}
see~\cite{P0} for details.
We sometimes write $a\ast b$ instead of $m(a, b)$. 

We also consider equivariant and graded matrix factorizations for the regular function \eqref{equation:NHilbTrW}, 
see~\cite[Subsection~2.6.2]{PT0} for details. The group $(\mathbb{C}^{\ast})^3$
acts on the linear maps corresponding to the edges $(x, y, z)$ 
of the quiver $Q$ by scalar multiplication. 
Consider the two dimensional subtorus 
\[T\cong (\mathbb{C}^{\ast})^2  \subset (\mathbb{C}^{\ast})^3\]
which preserves the super-potential $W=x[y, z]$, see \eqref{torus:T}.
Then $T$ acts on $\mathcal{X}(d)$
and preserves $\Tr W$. 
We will also consider graded matrix factorizations, where the 
grading is given by scaling with weight $2$ the space $\mathfrak{gl}(V)$ for $V$ a vector space. For example, we can choose an edge $e\in \{x, y, z\}$ of $Q$ and let $\mathbb{C}^*$ scale with weight $2$ the linear map corresponding to $e$. 
Contrary to the $T$-action, the regular function 
$\Tr W$ has weight $2$ with respect to such a grading. 
The corresponding categories of matrix factorizations are denoted by 
\begin{align*}
    \mathrm{MF}_{\ast}^{\bullet}(\X(d), \Tr W_d)\text{ for} \ 
    \ast \in \{\emptyset, T\}, \ \bullet \in \{\emptyset, \rm{gr}\}. 
\end{align*}

\subsection{Quasi-BPS categories}
\label{subsection:quasibps}
\subsubsection{}
For $w \in \mathbb{Z}$, 
we denote by $D^b(\mathcal{X}(d))_w$
the subcategory of $D^b(\mathcal{X}(d))$
consisting of objects of
weight $w$ with respect to the diagonal 
cocharacter $1_d$ of $T(d)$.  
We have the direct sum decomposition 
\begin{align*}
    D^b(\mathcal{X}(d))=\bigoplus_{w\in \mathbb{Z}}
    D^b(\mathcal{X}(d))_w. 
\end{align*}
We define the dg subcategories 
\begin{align*}
    \mathbb{M}(d) \subset D^b(\X(d)), \ 
    (\mbox{resp. }
    \mathbb{M}(d)_w \subset D^b(\X(d))_w)
\end{align*}
to be generated 
by the vector bundles $\OO_{\X(d)}\otimes \Gamma_{GL(d)}(\chi)$, where $\chi$ is a dominant weight of $T(d)$ such that
\begin{equation}\label{M}
    \chi+\rho\in \textbf{W}(d), \ 
    (\mbox{resp. } \chi+\rho \in \textbf{W}(d)_w). 
    \end{equation}
    Note that $\mathbb{M}(d)$
    decomposes into the direct sum of $\mathbb{M}(d)_w$
    for $w \in \mathbb{Z}$. 
    The following is an alternative description of the category $\mathbb{M}(d)_w$:
    \begin{lemma}\label{lemma:HLS}\emph{(\cite[Lemma~2.9]{hls})}
        The category $\mathbb{M}(d)_w$ 
is generated by the vector bundles $\OO_{\X(d)} \otimes \Gamma$
for $\Gamma$ a $GL(d)$-representation such that 
the $T(d)$-weights of $\Gamma$ are contained in the set
$\nabla_w$ defined by:
\begin{align*}
 \nabla_{w} &:=
    \left\{\chi \in M_{\mathbb{R}} \relmiddle| -\frac{1}{2}
    n_{\lambda} \leq \langle \lambda, \chi \rangle 
    \leq \frac{1}{2}n_{\lambda} \mbox{ for all } \lambda \colon \mathbb{C}^{\ast} \to T(d) \right\} +w\tau_d, \end{align*}
    \end{lemma}
For a partition $A=(d_i, w_i)_{i=1}^k$
of $(d, w)$, define 
\begin{align}\label{def:MA}
    \mathbb{M}_A := \boxtimes_{i=1}^k \mathbb{M}(d_i)_{w_i}. 
\end{align}

\subsubsection{}\label{gradingMF}
Recall the regular function (\ref{equation:NHilbTrW}). 
We define the subcategory
\[\mathbb{S}(d):=\text{MF}(\mathbb{M}(d), \Tr W_d)
\subset \mathrm{MF}(\mathcal{X}(d), \Tr W_d)
\] 
to be the subcategory of matrix factorizations
$\left(\alpha \colon F\rightleftarrows G\colon \beta\right)$ with $F$ and $G$ in $\mathbb{M}(d)$.
It decomposes into the direct sum of 
$\mathbb{S}(d)_w$ for $w \in \mathbb{Z}$, 
where $\mathbb{S}(d)_w$
is defined similarly to $\mathbb{S}(d)$
using $\mathbb{M}(d)_w$. 

We also consider subcategories 
for $\ast \in \{\emptyset, T\}$, 
$\bullet \in \{\emptyset, \text{gr}\}$
defined in a similar way 
\[\mathbb{S}^{\bullet}_{\ast}(d):=
\text{MF}^{\bullet}_{\ast}(\mathbb{M}(d), \Tr W_d)
\subset \text{MF}^{\bullet}_{\ast}(\X(d), \Tr W_d). 
\] 
The subcategory $\mathbb{S}^{\bullet}_{\ast}(d)_w$
is also defined in a similar way.
For a partition $A=(d_i, w_i)_{i=1}^k$
of $(d, w)$, 
the category $\mathbb{S}_{\ast, A}^{\bullet}$
is also defined similarly to (\ref{def:MA}). 
We denote the Grothendieck group of  $\mathbb{S}^{\bullet}_{\ast}(d)_w$ by
\begin{align*}
K_{\ast}(\mathbb{S}^{\bullet}(d)_w), \ 
\ast \in \{\emptyset, T\}, \ \bullet \in \{\emptyset, \mathrm{gr}\}. 
	\end{align*}
	By~\cite[Corollary~3.13]{T4}, 
	there are natural isomorphisms
	(which hold for all graded matrix factorizations as in Subsection \ref{MF}): 
\begin{align}\label{isom:DTgrade}
K(\mathbb{S}^{\mathrm{gr}}(d)_w) \stackrel{\cong}{\to} K(\mathbb{S}(d)_w), \ 
K_T(\mathbb{S}^{\mathrm{gr}}(d)_w) \stackrel{\cong}{\to} K_T(\mathbb{S}(d)_w). 
\end{align}

\subsection{Complexes in quasi-BPS categories}

Let $V$ be a $d$-dimensional complex vector space 
and recall that we denote by $\mathfrak{g}=\Hom(V, V)$ the Lie algebra of $GL(V)$. 
We set 
\begin{align*}
	\mathcal{Y}(d):=\mathfrak{g}^{\oplus 2}/GL(V),
	\end{align*}
where $GL(V)$ acts on $\mathfrak{g}$ by conjugation. 
The stack $\mathcal{Y}(d)$ is the moduli stack of 
representations of dimension $d$ of the quiver with one vertex and two loops. 
Let $s$ be the morphism 
\begin{align}\label{mor:s}
s \colon \mathcal{Y}(d) \to \mathfrak{g}, \ (X, Y) \mapsto [X, Y].
\end{align}
The morphism $s$ induces a map of vector bundles $\partial: \mathfrak{g}^{\vee}\otimes\mathcal{O}_{\mathfrak{g}^{\oplus 2}}\to \mathcal{O}_{\mathfrak{g}^{\oplus 2}}$.
Let $s^{-1}(0)$ be the derived scheme with the dg-ring of regular functions
\begin{align}\label{diff:ds2}
\mathcal{O}_{s^{-1}(0)}:=\mathcal{O}_{\mathfrak{g}^{\oplus 2}}\left[\mathfrak{g}^{\vee}\otimes\mathcal{O}_{\mathfrak{g}^{\oplus 2}}[1]; d_s\right],
\end{align}
where the differential $d_s$ is induced by the map $\partial$. Consider the (derived) stack 
\begin{equation}\label{def:cd}
\mathscr{C}(d):=s^{-1}(0)/GL(V) \hookrightarrow \mathcal{Y}(d).
\end{equation}
For a smooth variety $X$, we denote by $\mathscr{C}oh(X, d)$
the derived moduli stack of zero-dimensional sheaves on $X$
with length $d$
and by $\mathcal{C}oh(X, d)$ the classical truncation of $\mathscr{C}oh(X, d)$. 
Then 
$\mathscr{C}(d)$
is equivalent to $\mathscr{C}oh(\mathbb{C}^2, d)$. 
 
For a decomposition $d=d_1+\cdots+d_k$, 
let $\mathscr{C}(d_1, \ldots, d_k)$ be the derived moduli stack 
of filtrations of coherent sheaves on $\mathbb{C}^2$:
\begin{align}\label{filt:Q}
0=Q_0 \subset	Q_1 \subset Q_2 \subset \cdots \subset Q_k
	\end{align}
such that each subquotient $Q_i/Q_{i-1}$ is a zero-dimensional 
sheaf on $\mathbb{C}^2$ with length $d_i$. 
There exist evaluation morphisms 
\begin{align*}
	\mathscr{C}(d_1) \times \cdots \times \mathscr{C}(d_k) \stackrel{q}{\leftarrow} 
	\mathscr{C}(d_1, \ldots, d_k) \stackrel{p}{\to} \mathscr{C}(d),
	\end{align*}
where $p$ is proper and $q$ is quasi-smooth. 
The above diagram for $k=2$
defines the categorical Hall product
\begin{align}\label{hall:ast}
	m=m_{d_1,d_2}=p_{\ast}q^{\ast} \colon 
	D^b(\mathscr{C}(d_1)) \boxtimes D^b(\mathscr{C}(d_2)) \to 
	D^b(\mathscr{C}(d)),
	\end{align}
which is a special case of the product of categorical Hall algebras for surfaces defined by Porta--Sala~\cite{PoSa}. 
	
Let $T$ be the two-dimensional torus in (\ref{torus:T})
which acts on $\mathbb{C}^2$ 
by $(t_1, t_2) \cdot (x, y)=(t_1 x, t_2 y)$. 
It naturally induces an action on $\mathscr{C}(d)$. 
There is also a $T$-equivariant Hall product:
\begin{align}\label{hall:ast2}
	m=m_{d_1,d_2}=p_{\ast}q^{\ast} \colon 
	D^b_T(\mathscr{C}(d_1)) \boxtimes D^b_T(\mathscr{C}(d_2)) \to 
	D^b_T(\mathscr{C}(d)).
	\end{align}
Here, the box product is taken over $BT$. 
In what follows, whenever we take a box-product in the $T$-equivariant setting, 
we take it over $BT$. We also use the notation $\ast$ for the Hall product.

\subsection{Subcategories \texorpdfstring{$\mathbb{T}(d)_v$}{Tdv}}\label{subsection:Tdv}
Let 
\begin{equation}\label{def:mapi}
i \colon \mathscr{C}(d) \hookrightarrow \mathcal{Y}(d)
\end{equation}
be the natural closed immersion. 
Define the full triangulated subcategory \[\widetilde{\mathbb{T}}(d)_v\subset D^b(\mathcal{Y}(d))\] 
generated by the vector bundles $\mathcal{O}_{\mathcal{Y}(d)}\otimes \Gamma_{GL(d)}(\chi)$
for a dominant weight $\chi$ satisfying
\begin{align*}
	\chi+\rho \in \textbf{W}(d)_{v}. 
	\end{align*}
Define the full triangulated subcategory 
\begin{align}\label{def:N}
	\mathbb{T}(d)_v \subset D^b(\mathscr{C}(d))
	\end{align}
with objects $\mathcal{E}$ such that 
 $i_{\ast}\mathcal{E}$ is in $\widetilde{\mathbb{T}}(d)_v$.
 In \cite[Lemma 4.8]{PT0}, we showed that the Hall product restricts to functors
\[m: \mathbb{T}(d_1)_{v_1}\otimes \mathbb{T}(d_2)_{v_2}\to \mathbb{T}(d)_v\] for $(d, v)=(d_1, v_1)+(d_2, v_2)$ and $\frac{v_1}{d_1}=\frac{v_2}{d_2}$.
Also, there is a semiorthogonal decomposition, see~\cite[Corollary~3.3]{P2}:
\begin{align}\label{sod:C}
	D^b(\mathscr{C}(d))=\left\langle \mathbb{T}(d_1)_{v_1} \boxtimes 
	\cdots \boxtimes \mathbb{T}(d_k)_{v_k} \relmiddle|
	\begin{array}{c}
	v_1/d_1 < \cdots < v_k/d_k \\
	d_1+\cdots+d_k=d
	\end{array} \right\rangle. 
	\end{align}
	In the above, each fully-faithful functor 
	\begin{align*}
	 \mathbb{T}(d_1)_{v_1} \boxtimes 
	\cdots \boxtimes \mathbb{T}(d_k)_{v_k}
	\hookrightarrow D^b(\mathscr{C}(d))
	\end{align*}
	is given by the categorical Hall product (\ref{hall:ast}).

Consider the grading induced by the action of $\mathbb{C}^*$ on $\X(d)$ scaling the linear map corresponding to $Z$ with weight $2$.
The Koszul duality equivalence, also called dimensional reduction in the literature, gives the following equivalence \cite{I, Hirano, T}:
\begin{align}\label{equiv:Phi}
	\Phi \colon D^b(\mathscr{C}(d)) \stackrel{\sim}{\to}
	\mathrm{MF}^{\mathrm{gr}}(\mathcal{X}(d), \Tr W). 
	\end{align}
Under this equivalence, we have that $\Phi\colon \mathbb{T}(d)_v\stackrel{\sim}{\to} \mathbb{S}^{\mathrm{gr}}(d)_v$.

\subsection{Constructions of objects in \texorpdfstring{$\mathbb{T}(d)_v$}{TD}}
Here we review the construction of objects $\mathcal{E}_{d, v} \in \mathbb{T}(d)_v$ (which also produces an object in $\mathbb{T}_T(d)_v$)
following~\cite[Subsection~4.3]{PT0}. 
Let $\mathcal{Z} \subset \mathscr{C}(1, 1, \ldots, 1)$ be the closed substack
defined as follows. 
Let $\lambda$ be the cocharacter 
\begin{align}\label{lambda:cochar}
	\lambda \colon \mathbb{C}^{\ast} \to GL(V), \ 
	t \mapsto (t^d, t^{d-1}, \ldots, t). 
	\end{align}
	The attracting stack of $\mathcal{Y}(d)$
	with respect to $\lambda$ is given by 
	\begin{align}\label{mor:slambda}
	   \mathcal{Y}(d)^{\lambda \geq 0}:=
	   \left(\mathfrak{g}^{\lambda \geq 0}\right)^{\oplus 2}\big/GL(V)^{\lambda \geq 0}, 
	\end{align}
	where $GL(V)^{\lambda \geq 0} \subset GL(V)$ is the subgroup of 
	upper triangular matrices. 
Then the morphism 
(\ref{mor:s})
restricts to the morphism 
\begin{align}\label{mor:slambda2}
	s^{\lambda \geq 0} \colon 
	\mathcal{Y}(d)^{\lambda \geq 0} \to \mathfrak{g}^{\lambda \geq 0}
	\end{align}
whose derived zero locus 
$\mathscr{C}(d)^{\lambda \geq 0}$ is equivalent to $\mathscr{C}(1, \ldots, 1)$. 
Let $X=(x_{i,j})$ and $Y=(y_{i,j})$ be elements of $\mathfrak{g}^{\lambda \geq 0}$
for $1\leq i, j\leq d$, 
where $x_{i,j}=y_{i,j}=0$ for $i>j$. 
Then the 
equation $s^{\lambda \geq 0}(X, Y)=0$ is equivalent to the equations
\begin{align*}
\sum_{i\leq a \leq j} x_{i,a}y_{a,j}=\sum_{i\leq a \leq j} y_{i,a}x_{a,j}, 
	\end{align*}
for each $(i, j)$ with $i\leq j$. 
We call the above equation $\mathbb{E}_{i, j}$. 
The equation $\mathbb{E}_{i, i}$ is 
$x_{i, i} y_{i, i}-y_{i, i} x_{i, i}=0$, which always holds but 
imposes a non-trivial derived structure on 
$\mathscr{C}(1, \ldots, 1)$. 
The equation $\mathbb{E}_{i, i+1}$ is 
\begin{align*}
	(x_{i,i}-x_{i+1, i+1})y_{i, i+1}-(y_{i,i}-y_{i+1, i+1})x_{i, i+1}=0. 
	\end{align*}
The above equation is satisfied if the following 
equation $\mathbb{F}_{i, i+1}$ is satisfied: 
\begin{align*}
	\{x_{i, i}-x_{i+1, i+1}=0, y_{i, i}-y_{i+1, i+1}=0\}. 
	\end{align*}
We define the closed derived substack 
\begin{align}\label{defZ}
	\mathcal{Z}:=\mathcal{Z}(d) \subset \mathcal{Y}^{\lambda \geq 0}(d)
	\end{align}
to be the derived zero locus of 
the equations $\mathbb{F}_{i, i+1}$ for all $i$ and 
$\mathbb{E}_{i, j}$ for all $i+2 \leq j$. We usually drop $d$ from the notation if the dimension is clear from the context.
Then $\mathcal{Z}$ is a closed substack of 
$\mathscr{C}(d)^{\lambda \geq 0}=\mathscr{C}(1, \ldots, 1)$. 
Note that, set theoretically, the closed substack $\mathcal{Z}$
corresponds to filtrations (\ref{filt:Q}) 
such that each $Q_i/Q_{i-1}$ is isomorphic to $\mathcal{O}_x$ for some $x \in \mathbb{C}^2$
independent of $i$. 

We have the diagram of attracting loci 
\begin{align*}
	\mathscr{C}(1)^{\times d}=\mathscr{C}(d)^{\lambda} \stackrel{q}{\leftarrow}
	\mathscr{C}(d)^{\lambda \geq 0} \stackrel{p}{\to} \mathscr{C}(d),
	\end{align*}
where $p$ is a proper morphism. We set
\begin{align}\label{def:mi}
	m_i :=\left\lceil \frac{vi}{d} \right\rceil -\left\lceil \frac{v(i-1)}{d} \right\rceil
	+\delta_i^d -\delta_i^1 \in \mathbb{Z},
	\end{align}
	where $\delta^j_i$ is the Kronecker delta function defined by $\delta^j_i=1$ if $i=j$ and $\delta^j_i=0$ otherwise.
		 For a weight
	 $\chi=\sum_{i=1}^d n_i \beta_i$ with $n_i \in \mathbb{Z}$, 
		    we denote by $\mathbb{C}(\chi)$
		    the one dimensional $GL(V)^{\lambda \geq 0}$-representation given by 
		    \begin{align*}
		       GL(V)^{\lambda \geq 0} \to GL(V)^{\lambda}=T(d) \stackrel{\chi}{\to} \mathbb{C}^{\ast},
		    \end{align*}
		    where the first morphism is the projection.
		    \begin{defn}(\cite[Definition~4.2]{PT0})\label{definition:Edw}
We define the complex $\mathcal{E}_{d, v}$ by 
\begin{align}\label{def:Edw}
\mathcal{E}_{d, v}:=p_{\ast}\left(\mathcal{O}_{\mathcal{Z}} \otimes 
\mathbb{C}(m_1, \ldots, m_d)\right)
\in D^b(\mathscr{C}(d))_v. 
	\end{align}
The construction above is $T$-equivariant, so we also obtain an object $\mathcal{E}_{d, v}\in D^b_T(\mathscr{C}(d))_v$.
	\end{defn}

In \cite[Lemma 4.3]{PT0}, we showed that $\mathcal{E}_{d, v}$ is an object of $\mathbb{T}(d)_v$ and $\mathbb{T}_T(d)_v$.

\subsection{Shuffle algebras}\label{subsection:shufflealgebra}

\subsubsection{}
Consider the $\mathbb{N}$-graded $\mathbb{K}$-module: \[\mathcal{S}h:=\bigoplus_{d\geq 0}\mathbb{K}\left[z_1^{\pm 1}, \ldots, z_d^{\pm 1}\right]^{\mathfrak{S}_d}.\]
We define a shuffle product on $\mathcal{S}h$ as follows. Let $\xi(x)$ be defined by 
\begin{align*}
	\xi(x):=\frac{(1-q_1^{-1}x)(1-q_2^{-1}x)(1-q^{-1}x^{-1})}{1-x},
	\end{align*}
where $q:=q_1 q_2$. 
For $f \in \mathbb{K}\left[z_1^{\pm}, \ldots, z_a^{\pm}\right]$
and $g \in \mathbb{K}\left[z_{a+1}^{\pm}, \ldots, z_{a+b}^{\pm}\right]$, we set 
\begin{align}\label{def:shuffle}
	f \ast g:=\frac{1}{a! b!}
	\mathrm{Sym}\left(fg \cdot \prod_{\substack{1\leq i\leq a,\\ a<j\leq a+b}}
	\xi(z_i z_j^{-1})\right),
	\end{align}
where we denote by $\mathrm{Sym}(h(z_1, \ldots, z_d))$ the sum of 
$h(z_{\sigma(1)}, \ldots, z_{\sigma(d)})$ after all permutations $\sigma \in \mathfrak{S}_d$. Let $\mathcal{S}\subset \mathcal{S}h$ be the subalgebra generated by $z_1^l$ for $l\in \mathbb{Z}$.
Let $\mathcal{S}_{\mathbb{F}}:=\mathcal{S}\otimes_{\mathbb{K}}\mathbb{F}$. It is proved in~\cite[Theorem~4.6]{N2}
	that 
	there is an isomorphism
	\begin{align}\label{isom:S}
		i_{\ast} \colon \bigoplus_{d\geq 0}
		G_T(\mathscr{C}(d))\otimes_{\mathbb{K}}\mathbb{F}
		 \stackrel{\cong}{\to} \mathcal{S}_{\mathbb{F}}. 
		\end{align}
		The above isomorphism is induced by the algebra 
		homomorphism which will be defined in (\ref{i:ast}). 


\subsubsection{}

Let 
	\begin{align*}
		\mathcal{S}' \subset \bigoplus_{d\geq 0} \mathbb{K}\left(z_1, \ldots, z_d\right)^{\mathfrak{S}_d}
		\end{align*}
	be the $\mathbb{K}$-subalgebra generated by elements of the form 
	\begin{equation}\label{Ambullet}
		A'_{k_{\bullet}}:=
		\mathrm{Sym}\left(\frac{z_1^{k_1} \cdots z_d^{k_d}}{(1-q^{-1}z_1^{-1}z_2)\cdots 
			(1-q^{-1}z_{d-1}^{-1}z_d)} 
		\cdot \prod_{j>i}w(z_i z_j^{-1})  \right)
		\end{equation}
	for various $(k_1, \ldots, k_d) \in \mathbb{Z}^d$ and $d\geq 1$, for the shuffle product (\ref{def:shuffle}) where we replace $\xi(x)$ with $w(x)$ defined by
	\begin{align*}
		w(x):=\frac{(1-q_1^{-1}x)(1-q_2^{-1}x)}{(1-x)(1-q^{-1}x)}.
		\end{align*}
		Let $\mathcal{S}'_\mathbb{F}:=\mathcal{S}'\otimes_{\mathbb{K}}\mathbb{F}$.
		Consider the morphism 
	\begin{align}\label{mor:F}
	\bigoplus_{d\geq 0} \mathbb{F}[z_1^{\pm 1}, \ldots, z_d^{\pm 1}]^{\mathfrak{S}_d}
	\to \bigoplus_{d\geq 0} \mathbb{F}(z_1, \ldots, z_d)^{\mathfrak{S}_d}	
		\end{align}
	defined by 
	\begin{align*}
		f(z_1, \ldots, z_d) \mapsto f(z_1, \ldots, z_d) \cdot 
		\prod_{i\neq j}(1-q^{-1} z_i z_j^{-1})^{-1}. 
		\end{align*}
	Then (\ref{mor:F}) induces an algebra homomorphism $\mathcal{S}\to \mathcal{S}'$.
	There is an isomorphism
	\begin{equation}\label{SS'}
	    \mathcal{S}_\mathbb{F} \stackrel{\cong}{\to}
	\mathcal{S}'_{\mathbb{F}},
	\end{equation} see
	\cite[Proof of Lemma 4.11]{PT0}. 
	For $(d, v)\in\mathbb{N}\times\mathbb{Z}$, we set $A'_{d, v}$ to be $A'_{m_{\bullet}}$ 
	for the choice of
	$m_{\bullet}$ in (\ref{def:mi}). 
	By \cite[Equation~(2.12)]{N}, we have the following isomorphism of $\mathbb{K}$-modules:
	\begin{align}\label{basis:S'}
	\mathcal{S}'=\bigoplus_{v_1/d_1 \leq \cdots \leq v_k/d_k}
		\mathbb{K} \cdot A'_{d_1, v_1} \ast \cdots \ast A'_{d_k, v_k},
		\end{align}
		where the tuples $(d_i, v_i)_{i=1}^k$ appearing above
		are unordered for subtuples $(d_i, v_i)_{i=a}^b$
		with $v_a/d_a=\cdots=v_b/d_b$. 
We also define 
\begin{equation}\label{def:Ambullet2}
A_{d, v}:=\mathrm{Sym}\left(\frac{z_1^{m_1} \cdots z_d^{m_d}}{(1-q^{-1}z_1^{-1}z_2)\cdots 
			(1-q^{-1}z_{d-1}^{-1}z_d)} 
		\cdot \prod_{j>i}\xi(z_i z_j^{-1})  \right),
		\end{equation}
	where the exponents $m_i$ for $1\leq i\leq d$ are given by \eqref{def:mi}.

\subsubsection{}
The $T$-equivariant Hall product \eqref{hall:ast2} induces an associative algebra structure 
\begin{align}\label{G:prod}
m \colon G_T(\mathscr{C}(d_1)) \otimes_{\mathbb{K}}G_T(\mathscr{C}(d_2)) \to G_T(\mathscr{C}(d)).
	\end{align}
Let $i \colon \mathscr{C}(d) \hookrightarrow \mathcal{Y}(d)$ be the closed immersion. 
The pull-back 
via $\mathcal{Y}(d) \to BGL(d)$ 
gives the isomorphism 
\begin{align*}
    \bigoplus_{d\geq 0}K_T(BGL(d))=
\bigoplus_{d\geq 0}\mathbb{K}\left[z_1^{\pm 1}, \ldots, z_d^{\pm 1}\right]^{\mathfrak{S}_d}
\stackrel{\cong}{\to}
\bigoplus_{d\geq 0} K_T(\mathcal{Y}(d)). 
\end{align*}
Therefore the push-forward by $i$ induces a morphism 
\begin{align}\label{i:ast}
	i_{\ast} \colon 
	\bigoplus_{d\geq 0} G_T(\mathscr{C}(d)) \to 
\bigoplus_{d\geq 0}\mathbb{K}\left[z_1^{\pm 1}, \ldots, z_d^{\pm 1}\right]^{\mathfrak{S}_d}. 
	\end{align}
The product (\ref{G:prod}) is compatible with a shuffle product
defined on the right hand side of \eqref{i:ast}, see \cite[Subsection 4.5]{PT0}.
In \cite[Lemma 4.11]{PT0}, we showed that
	\begin{equation}\label{elem:E}
		i_{\ast}[\mathcal{E}_{d, v}] 
		=(1-q_1^{-1})^{d-1}(1-q_2^{-1})^{d-1}A_{d, v}.
		\end{equation}

\subsection{Compatibility of the Hall product under the Koszul equivalence}\label{subsection211}

In this subsection, we denote by $m$ the Hall product \eqref{hall:ast} and by $\widetilde{m}$ the Hall product for the quiver with potential $(Q, W)$ from Subsection \ref{sss1}.
Using the results in~\cite[Section~2.4]{T}, 
	Koszul 
	duality equivalences are compatible with
	the Hall products by the following commutative diagram (see~\cite[Proposition~3.1]{P2}):
	\begin{align}\label{com:hall}
		\xymatrix{
	D^b(\mathscr{C}(d_1)) \boxtimes D^b(\mathscr{C}(d_2)) \ar[r]^-{m} \ar[d]^-{\widetilde{\Phi}} &
	D^b(\mathscr{C}(d)) \ar[d]^-{\Phi} \\
	  \text{MF}^{\text{gr}}(\mathcal{X}(d_1), \Tr W_{d_1})
	  \boxtimes  \text{MF}^{\text{gr}}(\mathcal{X}(d_2),  \Tr W_{d_2})
	  \ar[r]^-{\widetilde{m}} &  \text{MF}^{\text{gr}}(\mathcal{X}(d), \Tr W_d),
	}
	\end{align}
where the left arrow $\widetilde{\Phi}$ is the composition of Koszul duality equivalences (\ref{equiv:Phi}) with the 
tensor product of 
\begin{equation}\label{def:factordimred}
	\det((\mathfrak{g}^{\nu>0})^{\vee}(2))[-\dim \mathfrak{g}^{\nu>0}]
	=(\det V_1)^{-d_2} \otimes (\det V_2)^{d_1}[d_1 d_2].
	\end{equation}
	The cocharacter 
	$\nu \colon \mathbb{C}^{\ast} \to T(d)$
is $\nu(t)=(\overbrace{t, \ldots, t}^{d_1}, \overbrace{1, \ldots, 1}^{d_2})$, the vector spaces $V_i$ have 
$\dim V_i=d_i$ for $i=1, 2$, and 
	$(1)$ is a twist by the weight one $\mathbb{C}^{\ast}$-character, 
	which is isomorphic to the shift functor $[1]$
	of the category of graded matrix factorizations.

\subsection{Symmetric polynomials}\label{subsection:symmetric}
Let $\Lambda$ be the $\mathbb{Z}$-algebra of symmetric polynomials \cite[Chapter I, Section~2]{MacDonald}, \cite[Subsection 2.4]{S}:
\[\Lambda\cong \varprojlim \mathbb{Z}\left[x_1,\ldots, x_n\right]^{\mathfrak{S}_n},\] with multiplication defined by 
\[f(x_1,\ldots, x_a)\star g(x_{a+1},\ldots, x_{a+b}):=\sum_{\mathfrak{S}_{a+b}/\mathfrak{S}_a\times \mathfrak{S}_b}w\left(f(x_1,\ldots, x_a) g(x_{a+1},\ldots, x_{a+b})\right)\]
and comultiplication induced by the restriction map
\[\mathbb{Z}\left[x_1,\ldots, x_{a+b}\right]^{\mathfrak{S}_{a+b}}\to \mathbb{Z}\left[x_1,\ldots, x_a\right]^{\mathfrak{S}_a}\otimes \mathbb{Z}\left[x_{a+1},\ldots, x_{a+b}\right]^{\mathfrak{S}_b}.\]
Alternatively, $\Lambda$ is isomorphic to 
the Grothendieck group of the monoidal category 
\[\mathcal{R}:=\bigoplus_{n\geq 0}\text{Rep}(\mathfrak{S}_n),\] 
where $\text{Rep}(\mathfrak{S}_n)$ is the abelian category of finite dimensional $\mathfrak{S}_n$-representations,  multiplication is given by the induction functor \[\text{Ind}: \text{Rep}(\mathfrak{S}_a\times \mathfrak{S}_b)\to \text{Rep}(\mathfrak{S}_{a+b}),\] and comultiplication is given by the restriction functor \[\text{Res}: \text{Rep}(\mathfrak{S}_{a+b})\to \text{Rep}(\mathfrak{S}_a\times \mathfrak{S}_b).\]
The isomorphism $\mathcal{R} \stackrel{\cong}{\to} \Lambda$
is given by sending an irreducible $\mathfrak{S}_n$-representation $W_{\lambda}$
corresponding to a partition $\lambda$ of $n$ 
to the Schur function $s_{\lambda}$, see~\cite[Chapter I, Equation (7.5)]{MacDonald}.

For $R$ a ring with a map $\mathbb{Z}\to R$, denote by $\Lambda_R:=R\otimes_{\mathbb{Z}}\Lambda$ the $R$-algebra
with multiplication and comultiplication induced from those of $\Lambda$. 
Consider the elementary symmetric functions 
\[e_n:=\sum_{i_1<\ldots<i_n}x_{i_1}\ldots x_{i_n}\in \Lambda\]
and the power sum functions
\[p_n:=\sum_{i}x_i^n\in \Lambda.\]
We also denote by $e_n$ and $p_n$ the images of these symmetric functions in $\Lambda_R$.
Let $t$ be a formal variable.
These functions are connected via the identity
\begin{equation}\label{elempowersum}
\sum_{n\geq 0}e_nt^n=\text{exp}\left(\sum_{n\geq 1}\frac{(-1)^{n+1}}{n}p_nt^n\right).\end{equation}
There are isomorphisms (see~\cite[Chapter I, Equations (2.4), (2.14)]{MacDonald}):
\begin{equation}\label{lambdaep}
\Lambda_{\mathbb{Q}}\cong \mathbb{Q}[e_1, e_2, \ldots]\cong \mathbb{Q}[p_1, p_2, \ldots].
\end{equation}
For $n\geq 1$, let $P(n)$ be the free one dimensional $\mathbb{Z}$-module with generator $p_n$. By \eqref{lambdaep}, there is an isomorphism of $\mathbb{N}$-graded $\mathbb{Q}$-vector spaces:
\begin{equation}\label{PBWmacdonald}
\Lambda_{\mathbb{Q}}\cong \bigotimes_{n\geq 1}\mathrm{Sym}\big(P(n)_{\mathbb{Q}}\big).
\end{equation}

\section{The support of complexes in quasi-BPS categories}\label{s3}

Recall the regular function \eqref{equation:NHilbTrW}. 
Consider the commutative diagram 
\begin{align}\label{dia:Cohsym}
	\xymatrix{
		\mathcal{C}oh(\mathbb{C}^3, d) \ar@{=}[r] \ar[d]_-{\pi} & \mathrm{Crit}(\Tr W_d)  \ar@<-0.3ex>@{^{(}->}[r]&
		\mathcal{X}(d) \ar[rd]^-{\Tr W_d} \ar[d]_-{\pi} & \\
		\mathrm{Sym}^d(\mathbb{C}^3)  \ar@<-0.3ex>@{^{(}->}[rr]& & X(d) \ar[r] & \mathbb{C},
}
	\end{align}
where the vertical arrows are good moduli space morphisms and the horizontal arrows 
are closed immersions. 
The left vertical arrow 
sends a zero-dimensional sheaf to its support. 
Let $\Delta\colon \mathbb{C}^3 \subset \mathrm{Sym}^d(\mathbb{C}^3)$
be the small diagonal 
\begin{align*}
\mathbb{C}^3 \subset \mathrm{Sym}^d(\mathbb{C}^3), \ 
x \mapsto (x, \ldots, x). 
\end{align*}
We abuse notation and denote the image of $\Delta$ also by $\Delta\cong \mathbb{C}^3$.
We consider the pull-back $\pi^{-1}(\Delta) \subset \mathcal{C}oh(\mathbb{C}^3, d)$, 
which is a 
closed substack of $\mathrm{Crit}(\Tr W_d) \subset \mathcal{X}(d)$.  
Davison~\cite[Theorem~5.1]{Dav} showed that the BPS sheaf for the moduli stack of degree $d$ sheaves on $\mathbb{C}^3$ is 
\[\mathcal{B}PS_d=\Delta_*\mathrm{IC}_{\mathbb{C}^3}.\]
Recall the notations involving formal completions from Subsection \ref{subsection:notation}. 
The following is the main result of this section, which will be proved in Subsection \ref{ss31}:

\begin{thm}\label{lem:support}
Consider a pair $(d, w)\in \mathbb{N}\times \mathbb{Z}$ and let $\mathcal{F}$ be an object in $\mathbb{S}(d)_w$. Assume there exists $p\in \mathrm{Sym}^d(\mathbb{C}^3)\setminus \Delta$ such that the support of $\mathcal{F}$ intersects $\pi^{-1}(p)$.  Write $p=\sum_{i=1}^l p^{(i)}$, $p^{(i)}=d^{(i)}x^{(i)}$, $d^{(i)}\in \mathbb{Z}_{>0}$, $x^{(i)}\neq x^{(i')}$ for $i\neq i'$ and $l\geq 2$. 
Then there exist non-zero objects
$\mathcal{F}_i \in \mathrm{MF}(\mathcal{X}_{p^{(i)}}(d^{(i)}), \Tr W_{d^{(i)}})_{w^{(i)}}$
 for $1\leq i\leq l$ with \[\frac{w^{(1)}}{d^{(1)}}=\cdots=\frac{w^{(l)}}{d^{(l)}}\] 
 and $w^{(1)}+\cdots+w^{(l)}=w$.
In particular, if $\gcd(d, w)=1$, 
then any object in $\mathbb{S}(d)_w$ is supported on 
$\pi^{-1}(\Delta)$. The same result holds for the categories $\mathbb{S}^{\bullet}_{\ast}(d)_w$ introduced in Subsection \ref{gradingMF}.
\end{thm}

In Subsection \ref{ss32}, we discuss divisibility properties of complexes supported on $\pi^{-1}(\Delta)$, see Lemma \ref{lemma:div}. In Subsection \ref{ss33}, we use Theorem \ref{lem:support} and Lemma \ref{lemma:div} to 
show that $K_T(\mathbb{S}(d)_w)'$ is a free $\mathbb{K}$-module with generator $[\mathcal{E}_{d,w}]$ when $\gcd(d, w)=1$.

\subsection{Proof of the main result}\label{ss31}

\begin{proof}[Proof of Theorem \ref{lem:support}]
		Suppose that 
		an object $\mathcal{F} \in \mathbb{S}(d)_w$
		has support not contained in $\pi^{-1}(\Delta)$. 			Then there is $p \in \mathrm{Sym}^d(\mathbb{C}^3) \setminus \Delta$
such that 
	$\mathcal{F}|_{\mathcal{X}_p(d)} \neq 0$
	in $\mathrm{MF}(\X_p(d), \Tr W_d)$, where 
	$\mathcal{X}_p(d)$ is the formal fiber of $p$ along 
	the good moduli space morphism $\mathcal{X}(d) \to X(d)$. 
	Since $p \notin \Delta$, it is written as 
\begin{align*}
	p=\sum_{i=1}^l p^{(i)}, \ p^{(i)}=d^{(i)}x^{(i)}, \ 
	d^{(i)} \in \mathbb{Z}_{>0}, \ 
	x^{(i)} \neq x^{(i')} \mbox{ for } i \neq i', \ l\geq 2.  
\end{align*}
	The unique closed point in $\pi^{-1}(p)$ corresponds to 
	the semisimple $Q$-representation 
	\begin{align*}
		R=\bigoplus_{i=1}^l V^{(i)} \otimes R^{(i)},
		\end{align*}
	where $R^{(i)}$ is the one-dimensional $Q$-representation 
	corresponding to $\mathcal{O}_{x^{(i)}}$ and $V^{(i)}$ is a 
	$d^{(i)}$-dimensional vector space
	such that $d^{(1)}+\cdots+d^{(l)}=d$. 
	Below we write a basis of $V^{(i)}$ as
	$\beta_1^{(i)}, \ldots, \beta_{d^{(i)}}^{(i)}$ and 
	set 
	\begin{align*}
		\{\beta_1, \ldots, \beta_d\}=\{\beta_1^{(1)}, \ldots, \beta_{d^{(1)}}^{(1)}, \beta_1^{(2)}, \ldots, \beta_{d^{(2)}}^{(2)},
		 \ldots \}. 
		\end{align*}
	The \'{e}tale slice theorem implies that 
	\begin{align}\label{isom:Xp}
		\mathcal{X}_p(d) \cong  \widehat{\Ext}_Q^1(R, R)/G_p,
		\end{align}
	where $G_p=\mathrm{Aut}(R)=\prod_{i=1}^l GL(V^{(i)})$
	and 
	$\widehat{\Ext}^1_Q(R, R)$ is the formal fiber of the origin 
	along the morphism $\Ext^1_Q(R, R) \to \Ext_Q^1(R, R)/\!\!/G_p$.  
	By Lemma~\ref{sublem0},
	the Ext-group $\Ext^1_Q(R, R)$ is computed as follows:
	\begin{align}\label{Ext1}
		\Ext_Q^1(R, R)=\bigoplus_{i=1}^l \mathrm{End}(V^{(i)}, V^{(i)})^{\oplus 3}
		\oplus \bigoplus_{i \neq j} \Hom(V^{(i)}, V^{(j)})^{\oplus 2}. 
		\end{align}
	Note that the maximal torus $T(d) \subset GL(V)$ is contained in $G_p$. 
	
	We define 
	\begin{align*}
		\mathbb{S}_p(d)_{w} \subset \mathrm{MF}(\mathcal{X}_p(d), \Tr W_d)
		\end{align*}
	to be the full subcategory generated by matrix factorizations 
	whose entries are of the form $\Gamma_{G_p}(\chi) \otimes \mathcal{O}$, 
	where $\chi$ is a $G_p$-dominant $T(d)$-weight
	satisfying 
	\begin{align}\label{wei:chip}
		\chi+\rho_p \in \textbf{W}_p(d)_w. 
		\end{align}
	Here, $\rho_p$ is half the sum of positive roots of $G_p$
	and $\textbf{W}_p(d)_{w}$ is defined as in (\ref{W0}) 
	for the $G_p$-representation $\Ext_Q^1(R, R)$:
	\begin{align*}
	    \textbf{W}_p(d)_{w}:=\frac{1}{2}\mathrm{sum}[0, \beta]+w \tau_d,
	\end{align*}
	where the Minkowski sum is after all weights $\beta$ in $\Ext_Q^1(R, R)$. 
	By Lemma~\ref{sublem1}, we have $\mathcal{F}|_{\X_p(d)} \in \mathbb{S}_p(d)_w$, 
	in particular $\mathbb{S}_p(d)_{w} \neq 0$. 

	Below, in order to simplify the notation, we treat the case $l=2$. 
	Since any zero-dimensional sheaf $Q$ on $\mathbb{C}^3$ supported on 
	$p$ decomposes into 
	$Q^{(1)} \oplus Q^{(2)}$, where $Q^{(i)}$ is supported on $x^{(i)}$, we have 
	\begin{align}\label{coh:tr}
		\widehat{\mathcal{C}oh}_p(\mathbb{C}^3, d)=
		\mathrm{Crit}( \Tr W_d|_{\mathcal{X}_p(d)})=\widehat{\mathcal{C}oh}_{p^{(1)}}(\mathbb{C}^3, d^{(1)})
		\times \widehat{\mathcal{C}oh}_{p^{(2)}}(\mathbb{C}^3, d^{(2)}).
	\end{align}	
	Here $\widehat{\mathcal{C}oh}_p(\mathbb{C}^3, d)$
	is the formal fiber of the left vertical arrow in (\ref{dia:Cohsym}) at $p$. 
	Indeed, by Lemma~\ref{sublem:1.5}, we can show that, by replacing the isomorphism (\ref{isom:Xp}) if necessary, 
	the regular function $\Tr W_d$ restricted to $\mathcal{X}_p(d)$ is written as 
	\begin{align}\label{W:Xpd}
	 \Tr W_d|_{\mathcal{X}_p(d)}= \Tr W_{d^{(1)}}\boxplus 
	 \Tr W_{d^{(2)}} \boxplus q.
	\end{align}
	Here, $\Tr W_{d^{(i)}}$ is the regular function (\ref{equation:NHilbTrW}) on $\mathcal{X}(d^{(i)})$
	restricted to $\mathcal{X}_{p^{(i)}}(d^{(i)})$ and
	$q$ is a non-degenerate $G_p$-invariant 
	quadratic form on $U\oplus U^{\vee}$
	given by $q(u, v)=\langle u, v \rangle$, where $U$ is the 
	following self-dual $G_p$-representation 
	\begin{align*}
		U:=\Hom(V^{(1)}, V^{(2)}) \oplus \Hom(V^{(2)}, V^{(1)}). 	
	\end{align*}
	The decomposition (\ref{W:Xpd}) in particular 
	implies (\ref{coh:tr}). 
	
	We have the following diagram 
	\begin{align*}
		\xymatrix{
			\mathcal{U} \ar@<-0.3ex>@{^{(}->}[r]^-{i} \ar[d]_-{p} 
			& \mathcal{U} \oplus \mathcal{U}^{\vee} & \ar[l]_-{j} \mathcal{X}_p(d) \\
			\mathcal{X}_{p^{(1)}}(d^{(1)}) \times \mathcal{X}_{p^{(2)}}(d^{(2)}), & &	
		}
	\end{align*}
	where $\mathcal{U}$ is the vector bundle on $\mathcal{X}_{p^{(1)}}(d^{(1)}) \times \mathcal{X}_{p^{(2)}}(d^{(2)})$ determined by the $G_p$-representation $U$, $i$ is the closed immersion $x \mapsto (x, 0)$, 
	and $j$ is the natural morphism 
	induced by the formal completion which induces the isomorphism on critical loci of $\Tr W_d$. 
	We have the following functors 
	\begin{align}\label{Kn:eq}
		\Psi := j^{\ast}i_{\ast}p^{\ast} \colon 
		\mathrm{MF}(\mathcal{X}_{p^{(1)}}(d^{(1)}), \Tr W_{d^{(1)}}) &\boxtimes 
		\mathrm{MF}(\mathcal{X}_{p^{(2)}}(d^{(2)}), \Tr W_{d^{(2)}}) \\
		\notag &\stackrel{\sim}{\to}\mathrm{MF}(\mathcal{U} \oplus \mathcal{U}^{\vee}, \Tr W_d) 
		\stackrel{j^{\ast}}{\hookrightarrow}
		 \mathrm{MF}(\mathcal{X}_p(d), \Tr W_d). 
	\end{align}
	Here the first arrow is an equivalence 
	by Kn\"orrer periodicity (see~\cite[Theorem~4.2]{Hirano}), 
	and the second arrow is fully-faithful with dense image (see~\cite[Lemma~6.4]{T3}). 
		By Lemma~\ref{lemma:restS},
	the above functor
	restricts to the fully-faithful functor:
	\begin{align}\label{rest:M}
		\bigoplus_{\begin{subarray}{c}w^{(1)}+w^{(2)}=w \\
				w^{(1)}/d^{(1)}=w^{(2)}/d^{(2)}
		\end{subarray}}
		\mathbb{S}_{p^{(1)}}(d^{(1)})_{w^{(1)}} \boxtimes 
		\mathbb{S}_{p^{(2)}}(d^{(2)})_{w^{(2)}}
		\to \mathbb{S}_{p}(d)_w. 
	\end{align}
	
	Below we show that the above functor (\ref{rest:M}) has dense image, and thus the conclusion follows. 
	By the semiorthogonal decomposition (\ref{sod:C}) together with (\ref{com:hall}), 
	we have 
	\begin{align}\label{sod:mf}
		&\mathrm{MF}(\mathcal{X}(d), \Tr W_d) \\
		\notag &= 
		\left\langle \ast_{i=1}^k \mathbb{S}(d_i)_{v_i+d_i(\sum_{i>j}d_j-\sum_{j>i}d_j)} \relmiddle|
		\frac{v_1}{d_1}<\cdots<\frac{v_k}{d_k}, d_1+\cdots+d_k=d\right\rangle. 
	\end{align}
The weights $w_i:=v_i+d_i(\sum_{i>j}d_j-\sum_{j>i}d_j)$
are computed as in Equation (\ref{w:prime}):
\begin{align}\label{def:wi}
	\sum_{i=1}^k w_i \tau_{d_i}=\sum_{i=1}^k v_i \tau_{d_i}
	-\mathfrak{g}^{\lambda>0},
	\end{align}
where $\mathfrak{g}$ is the Lie algebra of $GL(d)$ and 
$\lambda$ is the antidominant cocharacter 
corresponding to the decomposition $d=d_1+\cdots+d_k$:
\begin{align}\label{lambda:d}
    \lambda(t)=(\overbrace{t^k, \ldots, t^k}^{d_1}, \overbrace{t^{k-1}, \ldots, t^{k-1}}^{d_2}, \ldots, \overbrace{t, \ldots, t}^{d_k}). 
\end{align}
For $x \in \mathbb{C}^3$, let $q=d[x]:=(x,\ldots, x) \in \mathrm{Sym}^d(\mathbb{C}^3)$. 
Since $G_q=GL(V)$ and $\X_q(d)\cong 
\mathfrak{g}^{\oplus 3}_{d[0]}/GL(V)$, the argument showing (\ref{sod:mf})
applies verbatim to show the semiorthogonal decomposition 
\begin{align}\label{sod:mf2}
		&\mathrm{MF}(\mathcal{X}_q(d), \Tr W_d) \\
		\notag &= 
		\left\langle \ast_{i=1}^k \mathbb{S}_q(d_i)_{v_i+d_i(\sum_{i>j}d_j-\sum_{j>i}d_j)} \relmiddle|
		\frac{v_1}{d_1}<\cdots<\frac{v_k}{d_k}, d_1+\cdots+d_k=d\right\rangle. 
	\end{align}
	
		From (\ref{sod:mf2}), the left hand side of (\ref{Kn:eq}) admits a 
	semiorthogonal decomposition S whose summands are of the form 
	\begin{align}\label{summand:Mp}
		\ast_{i=1}^a \mathbb{S}_{p_i^{(1)}}(d_i^{(1)})_{v_i^{(1)}+d_i^{(1)}(\sum_{i>j}d_j^{(1)}-\sum_{j>i}d_j^{(1)})}
		\boxtimes 	\ast_{i=1}^b \mathbb{S}_{p_i^{(2)}}(d_i^{(2)})_{v_i^{(2)}+d_i^{(2)}(\sum_{i>j}d_j^{(2)}-\sum_{j>i}d_j^{(2)})}. 	\end{align}
	Here, 
	the left hand side in (\ref{rest:M}) comes to the rightmost part of the semiorthogonal 
	decomposition S, 
	 $p_i^{(j)}=d_i^{(j)}x^{(j)}$,
	 the integers
	$d_i^{(j)}$ satisfy 
	\begin{align*}
		d_1^{(1)}+ \cdots+ d_a^{(1)}=d^{(1)}, \ 
		d_1^{(2)}+ \cdots+ d_b^{(2)}=d^{(2)}, 
	\end{align*}
	and we have the inequalities 
	\begin{align*}
		\frac{v_1^{(1)}}{d_1^{(1)}}<\cdots<	\frac{v_a^{(1)}}{d_a^{(1)}}, \ 
		\frac{v_1^{(2)}}{d_1^{(2)}}<\cdots<	\frac{v_b^{(2)}}{d_b^{(2)}}.
	\end{align*}
	We write 
	\begin{align*}
		\left\{ \frac{v_1^{(1)}}{d_1^{(1)}}, \ldots, \frac{v_a^{(1)}}{d_a^{(1)}}, \ldots, 
		\frac{v_1^{(2)}}{d_1^{(2)}}, \ldots, \frac{v_b^{(2)}}{d_b^{(2)}}    \right\}=
		\{\mu_1<\cdots<\mu_k\},
	\end{align*}
	where $k$ is the number of distinct elements in the left hand side. 
	For $1\leq i\leq k$, we replace $(d_i^{(j)}, v_i^{(j)})$ by 
	\begin{align*}
		(d_i^{(j)}, v_i^{(j)}) \mapsto \begin{cases}
			(d_l^{(j)}, v_l^{(j)}), & \mbox{ if } \mu_i=v_l^{(j)}/d_l^{(j)} \mbox{ for some }l, \\
			(0, 0), & \mbox{ otherwise}. 
		\end{cases}	
	\end{align*}
	The subcategory (\ref{summand:Mp}) of (\ref{rest:M}) is unchanged under the above replacement. 
	Therefore we may assume that $a=b=k$ and, by setting $(d_i, v_i)=(d_i^{(1)}, v_i^{(1)})
	+(d_i^{(2)}, v_i^{(2)})$, we have   
	\begin{align*}
		\frac{v_i}{d_i}=\frac{v_i^{(j)}}{d_i^{(j)}}, \ 
		\frac{v_1}{d_1}< \cdots< \frac{v_k}{d_k}. 
	\end{align*}
	Here the first identity holds whenever $(d_i^{(j)}, v_i^{(j)}) \neq (0, 0)$. 
	
	Let us take
	decompositions 
	$V^{(j)}=\oplus_{i=1}^k V_i^{(j)}$ with 
	$\dim V_i^{(j)}=d_i^{(j)}$ and 
	a cocharacter as in (\ref{lambda:d})
	\begin{align}\label{lambda:12}
		\lambda \colon \mathbb{C}^{\ast} \to GL(V^{(1)}) \times GL(V^{(2)})
	\end{align}
	which acts on $V_i^{(j)}$ with weight $k+1-i$. 
	The diagrams of attracting loci with respect to $\lambda$ give the 
	commutative diagram, see~\cite[Proposition~2.5]{T3}:
	\begin{align*}
		\xymatrix{
			\boxtimes_{i=1}^k (\mathrm{MF}(\mathcal{X}_{p_i^{(1)}}(d_i^{(1)}), \Tr W_{d_i^{(1)}}) \boxtimes 
			(\mathrm{MF}(\mathcal{X}_{p_i^{(2)}}(d_i^{(2)}), \Tr W_{d_i^{(2)}})
			) \ar[r] \ar[d]_{\ast \boxtimes \ast} & \boxtimes_{i=1}^k \mathrm{MF}(\mathcal{X}_{p_i}(d_i),  \Tr W_{d_i}) \ar[d]_-{\ast} \\
			\mathrm{MF}(\mathcal{X}_{p^{(1)}}(d^{(1)}), \Tr W_{d^{(1)}}) \boxtimes 
			\mathrm{MF}(\mathcal{X}_{p^{(2)}}(d^{(2)}), \Tr W_{d^{(2)}}) \ar[r] & 
			\mathrm{MF}(\mathcal{X}_p(d), \Tr W_d).
		}
	\end{align*}
	Here, the vertical arrows are categorical Hall products determined by 
	$\lambda$, $p_i:=p_i^{(1)}+p_i^{(2)}$, 
	the bottom horizontal arrow is (\ref{Kn:eq}), and the top horizontal arrow is the composition of the Kn\"orrer periodicity equivalences with 
	the tensor product of (a shift of)
	\begin{align*}
		\det \left((U^{\lambda>0})^{\vee}\right)=
		\bigotimes_{i=1}^k\left( (\det V_i^{(1)})^{\sum_{i>j}d_j^{(2)}-\sum_{j>i}d_j^{(2)}}
		\otimes  (\det V_i^{(2)})^{\sum_{i>j}d_j^{(1)}-\sum_{j>i}d_j^{(1)}}\right). 
	\end{align*} 
 By the above commutative diagram together with the fact that (\ref{Kn:eq})
 restricts to the functor (\ref{rest:M}), 
the functor (\ref{Kn:eq})
	sends (\ref{summand:Mp}) to 
	\begin{align}\label{subcat:Mk}
	\ast_{i=1}^k	\mathbb{S}_{p_i}(d_i)_{v_i+d_i(\sum_{i>j}d_j-\sum_{j>i}d_j)}
		\subset \mathrm{MF}(\mathcal{X}_p(d), \Tr W_d). 
	\end{align}
	Let us take a decomposition \[w=v_1+\cdots+v_k=w_1+\cdots+w_k,\]
	where $w_i$ is given in (\ref{def:wi}). 
	Then by Lemma~\ref{sublem2}, 
	the subcategory (\ref{subcat:Mk}) for $k\geq 2$ 
	is right orthogonal to $\mathbb{S}_p(d)_w$. 
	Together with the 
	semiorthogonal decomposition $\mathrm{S}$ with summands (\ref{summand:Mp}), 
	we conclude that the functor (\ref{rest:M})
	has dense image. Indeed, for an object $A \in \mathbb{S}_p(d)_w$, 
	there is a distinguished triangle
	\begin{align}\label{dist:AB}
	    A_1 \xrightarrow{\alpha} A \xrightarrow{\beta} A_2
	\end{align}
	where 
	$A_1$ is a direct summand of an object $\Psi(B_1)$
	for some $B_1$
	in the left hand side of (\ref{rest:M}) and
	$A_2$ is a direct summand of an object 
	$\Psi(B_2)$ for some $B_2$ in 
	(\ref{summand:Mp}) for $a=b\geq 2$. 
	Then $\Psi(B_2)$ is an object of (\ref{subcat:Mk}) for $k\geq 2$, 
	hence $\beta$ is a zero map
	by Lemma~\ref{sublem2}. 
	So $\alpha$ is an isomorphism, hence 
	the functor (\ref{rest:M}) has dense image. 
	\end{proof}
	
	We have postponed several lemmas, which are given below: 

\begin{lemma}
    \label{sublem0}
    The $\Ext$-group $\Ext_Q^1(R, R)$ is computed as (\ref{Ext1}). 
\end{lemma}
\begin{proof}
    By the Euler pairing computation, we have 
    $\chi_Q(R^{(i)}, R^{(j)})=-2$. 
    Since $\Hom(R^{(i)}, R^{(j)})=\mathbb{C} \delta_{ij}$, 
    we have 
    \begin{align*}
        \Ext_Q^1(R^{(i)}, R^{(j)})=\begin{cases} \mathbb{C}^2, & i \neq j, \\
        \mathbb{C}^3, & i=j.
        \end{cases}
    \end{align*}
    Therefore (\ref{Ext1}) holds. 
\end{proof}

\begin{lemma}\label{sublem:1.5}
By replacing the isomorphism (\ref{isom:Xp}) if necessary, 
the identity (\ref{W:Xpd}) holds. 
\end{lemma}
\begin{proof}
    We may assume $x^{(1)}=(0, 0, 1)$ and $x^{(2)}=(0, 0, 0)$. 
    We write an element of 
    $\Ext_Q^1(R, R)$ as 
    \begin{align}\label{elm:Ext}
        (X^{(1)}, Y^{(1)}, Z^{(1)}, X^{(2)}, Y^{(2)}, Z^{(2)}, 
        A^{(12)}, A^{(21)}, B^{(12)}, B^{(21)}),
    \end{align}
    where $X^{(i)} \in \mathrm{End}(V^{(i)})$
    and $A^{(ij)} \in \Hom(V^{(i)}, V^{(j)})$. 
    We have the morphism of algebraic stacks
    \begin{align*}
    \nu \colon 
    \widehat{\Ext}_Q^1(R, R)/G_p \to \mathcal{X}_p(d)
    \end{align*}
    which sends (\ref{elm:Ext}) to 
    \begin{align*}
    \left(
    \begin{pmatrix}
X^{(1)} & A^{(12)} \\
A^{(21)} & X^{(2)} \\
\end{pmatrix}, \ 
\begin{pmatrix}
Y^{(1)} & B^{(12)} \\
B^{(21)} & Y^{(2)} \\
\end{pmatrix}, \
\begin{pmatrix}
Z^{(1)}+I & 0 \\
0 & Z^{(2)} \\
\end{pmatrix}
\right) \in \mathfrak{g}^{\oplus 3}. 
    \end{align*}
    Indeed, the above correspondence is 
    $GL(V)$-equivariant using the embedding $G_p \subset GL(V)$, 
    so it determines a morphism $\nu$. 
    Note that $\nu(0)$ corresponds to the 
    polystable $Q$-representation $R=(V^{(1)} \otimes R^{(1)}) \oplus (V^{(2)} \otimes R^{(2)})$, 
    where $R^{(i)}$ corresponds to $\mathcal{O}_{x^{(i)}}$. 
    
    We now explain that the morphism $\nu$ is \'{e}tale at $\nu(0)$. 
    The tangent complex of $\mathcal{X}(d)$
    at $\nu(0)$ is 
    \begin{align*}
        \mathbb{T}_{\mathcal{X}(d)}|_{\nu(0)}=
        \left(\mathrm{End}(V) \to \mathrm{End}(V)^{\oplus 3}   \right), \ 
        \alpha \mapsto (0, 0, [\alpha, u]), \ 
        u=\begin{pmatrix}
I & 0 \\
0 & 0 \\
\end{pmatrix}.
        \end{align*}
        The kernel of the above map is 
        $\mathrm{End}(V^{(1)}) \oplus \mathrm{End}(V^{(2)})$, 
        and the cokernel is $\Ext_Q^1(R, R)$, 
        so the morphism $\nu$ induces a quasi-isomorphism on tangent 
        complexes at $\nu(0)$. 
        
        A straightforward computation shows that 
        \begin{align*}
            \nu^{\ast} \Tr W_d=&
            \mathrm{Tr}(Z^{(1)}[X^{(1)}, Y^{(1)}])
            + \mathrm{Tr}(Z^{(2)}[X^{(2)}, Y^{(2)}]) \\
            &+\mathrm{Tr}(A^{(12)}(B^{(21)}+B^{(21)}Z^{(1)}-Z^{(2)}B^{(21)})) \\
            &+\mathrm{Tr}(B^{(12)}(Z^{(2)}A^{(21)}-A^{(21)}Z^{(1)}-A^{(21)})). 
        \end{align*}
        By the following $G_p$-equivariant variable change
        \begin{align*}
            A^{(21)} \mapsto Z^{(2)}A^{(21)}-A^{(21)}Z^{(1)}-A^{(21)}, \ 
            B^{(21)} \mapsto B^{(21)}+B^{(21)}Z^{(1)}-Z^{(2)}B^{(21)},
        \end{align*}
        we obtain the identity (\ref{W:Xpd}). 
\end{proof}

\begin{lemma}\label{lemma:restS}
    The functor (\ref{Kn:eq}) restricts 
    to the functor (\ref{rest:M}). 
\end{lemma}
\begin{proof}
(cf. the proof of~\cite[Theorem~2.7]{KoTo})
For a cocharacter $\lambda \colon 
\mathbb{C}^{\ast} \to T(d)$, 
let
\begin{align*}
n_{\lambda, p} :=\left\langle \lambda, 
\mathbb{L}^{\lambda>0}_{\mathcal{X}_p(d)}\big|_{0}
\right\rangle, \
    n_{\lambda, p}':=\left\langle \lambda, 
    \mathbb{L}^{\lambda>0}_{\mathcal{X}_{p^{(1)}}(d^{(1)}) \times 
    \mathcal{X}_{p^{(2)}}(d^{(2)})}\big|_{0} \right\rangle.
\end{align*}
Define the sets of weights
\begin{align*}
    \nabla_{p, w} &:=
    \left\{\chi \in M_{\mathbb{R}} \relmiddle| -\frac{1}{2}
    n_{\lambda, p} \leq \langle \lambda, \chi \rangle 
    \leq \frac{1}{2}n_{\lambda, p} \mbox{ for all } \lambda \right\} +w\tau_d, \\
   \nabla'_{p, w} &:=
    \left\{\chi \in M_{\mathbb{R}}  \relmiddle| -\frac{1}{2}
    n'_{\lambda, p} \leq \langle \lambda, \chi \rangle 
    \leq \frac{1}{2}n'_{\lambda, p} \mbox{ for all } \lambda \right\}+w\tau_d.
\end{align*}
Then an object in the right (resp. left)
hand side of (\ref{Kn:eq})
lies in the right (resp. left)
hand side of (\ref{rest:M})
if and only if its $T(d)$-weights are contained in 
$\nabla_{p, w}$ (resp.~$\nabla_{p, w}'$), see~\cite[Lemma~2.9]{hls}. 
Let $\mathcal{E}$ be an object in the left hand 
side of (\ref{rest:M}). 
Using a Koszul resolution, the complex $\Psi(\mathcal{E})$ is generated by 
$p^{\ast}(\mathcal{E})\otimes \wedge^{\bullet}U$. 
Let $\chi$ be a $T(d)$-weight of 
$p^{\ast}(\mathcal{E})\otimes \wedge^{\bullet}U$. 
We then have that
\begin{align*}
    -\frac{1}{2}n_{\lambda, p}'+\langle \lambda, 
    U^{\lambda<0} \rangle \leq 
    \langle \lambda, \chi \rangle 
    \leq \frac{1}{2}n_{\lambda, p}'+\langle 
    \lambda, U^{\lambda>0} \rangle. 
\end{align*}
Since $U$ is self-dual, we have 
$\langle \lambda, U^{\lambda<0}\rangle=
-\langle \lambda, U^{\lambda>0}\rangle$. 
Moreover, from (\ref{Ext1}), it is easy to see that 
\begin{align*}
\frac{1}{2}n_{\lambda, p}'+\langle \lambda, U^{\lambda>0}\rangle=\frac{1}{2}n_{\lambda, p}. 
    \end{align*}
    Therefore $\Psi(\mathcal{E})$ is an object 
    of the right hand side of (\ref{rest:M}). 
\end{proof}
\begin{lemma}\label{sublem1}
If a $T(d)$-weight $\chi$ satisfies (\ref{wei:chip}), then 
\begin{align}\label{chi:rho}
	\chi+\rho \in \mathbf{W}(d)_w. 
	\end{align}
Conversely if a $GL(V)$-dominant $T(d)$-weight $\chi$ satisfies (\ref{chi:rho}), 
then it satisfies (\ref{wei:chip}). 
\end{lemma}
\begin{proof}
	If a $T(d)$-weight $\chi$ satisfies (\ref{wei:chip}), it is written as 
	\begin{align}\label{chi:rhop}
		\chi+\rho_p=\sum_{i, j, a} c_{ij}^{(a)}(\beta_i^{(a)}-\beta_j^{(a)})
		+\sum_{i, j, a \neq b} c_{ij}^{(ab)}(\beta_i^{(a)}-\beta_j^{(b)})
		\end{align}
	where $0\leq c_{ij}^{(a)} \leq 3/2$ and $0\leq c_{ij}^{(ab)} \leq 1$. 
	Since we have 
	\begin{align}\label{rho-rhop}
		\rho-\rho_p=\frac{1}{2}\sum_{i, j, a<b}(\beta_i^{(b)}-\beta_j^{(a)}),
		\end{align}
	the weight $\chi+\rho$ satisfies (\ref{chi:rho}). Conversely, if $\chi$ is a 
	$GL(V)$-dominant $T(d)$-weight satisfying (\ref{chi:rho}), from~\cite[Proposition~3.5]{PT0}
	it is written as 
	\begin{align*}
		\chi+\rho=\sum_{i>j} c_{ij}(\beta_i-\beta_j)
		\end{align*}
	for $0\leq c_{ij} \leq 3/2$. 
	Therefore from (\ref{rho-rhop}), the weight $\chi+\rho_p$ is written as (\ref{chi:rhop}). 
	\end{proof}
\begin{lemma}\label{sublem2}
	The subcategory (\ref{subcat:Mk}) for $k\geq 2$ 
	is right orthogonal to 
	$\mathbb{S}_p(d)_{w}$ for $w=w_1+ \cdots+w_k$, where $w_i$
	is given in (\ref{def:wi}). 
	\end{lemma}
\begin{proof}
Recall that $w_i=v_i+d_i(\sum_{i>j}d_j-\sum_{j>i}d_j)$. 
Choose weights $\psi_i \in \textbf{W}_{p_i}(d_i)_{0}$ for $1\leq i\leq k$. 
From the proof of semiorthogonality in~\cite[Proposition~4.3]{P}, 
it is enough to show that the $r$-invariant of the weight 
\begin{align}\label{weight:b}
	\sum_{i=1}^k \psi_i+
	\sum_{i=1}^k w_i \tau_{d_i}-\sum_{i=1}^k \rho_{p_i}+\rho_p
	\end{align}	
with respect to the polytope $\textbf{W}_p(d)$
is bigger than $1/2$. 
Here, the polytope $\textbf{W}_p(d)$ is defined as follows: 
\begin{align*}
    \textbf{W}_p(d):=\frac{1}{2} \mathrm{sum}[0, \beta]+\mathbb{R}\tau_d,
    \end{align*}
  where the sum is after all weights in $\Ext_Q^1(R, R)$.  
  
Suppose the contrary, i.e. the weight \eqref{weight:b} lies in 
$\textbf{W}_p(d)_{w}$. Let $\lambda$ be the cocharacter
as in (\ref{lambda:12}), 
and write it as $\lambda=(\lambda^{(1)}, \lambda^{(2)})$, 
where $\lambda^{(j)}$ is the cocharacter of $GL(V^{(j)})$. 
We set $\mathfrak{g}=\mathrm{End}(V)$
and $\mathfrak{g}^{(j)}=\mathrm{End}(V^{(j)})$. 
The weight (\ref{weight:b}) is written as 
\begin{align*}
	\sum_{i=1}^k \psi_i+
	\sum_{i=1}^k v_i \tau_{d_i}
	-\mathfrak{g}^{\lambda>0}-\frac{1}{2}(\mathfrak{g}^{(1)})^{\lambda^{(1)}>0}
	-\frac{1}{2}(\mathfrak{g}^{(2)})^{\lambda^{(2)}>0}. 
	\end{align*}
	By the assumption, the above weight 
	is an element of $\textbf{W}_p(d)_w$. 
	Then using an argument as in the proof of Lemma~\ref{sublem1}, we have that 
\begin{align}\label{lambda0}
\sum_{i=1}^k \psi_i+\sum_{i=1}^k v_i \tau_{d_i}-\frac{3}{2}\mathfrak{g}^{\lambda>0}
	\in \textbf{W}(d)_w. 
	\end{align}
Note that we have 
\begin{align*}
	\left\langle \lambda, 	\sum_{i=1}^k \psi_i+\sum_{i=1}^k v_i \tau_{d_i}-\frac{3}{2}\mathfrak{g}^{\lambda>0}
	-w \tau_d \right\rangle 
	=\sum_{i=1}^k (k+1-i)d_i\left(\frac{v_i}{d_i}-\frac{w}{d}\right)
	-\left\langle \lambda, \frac{3}{2}\mathfrak{g}^{\lambda>0} \right\rangle.
	\end{align*}
	For $1\leq i\leq k$, define \[\tilde{v}_i:=d_i\left(\frac{v_i}{d_i}-\frac{w}{d}\right).\] 
	Then $\tilde{v}_1+\cdots+\tilde{v}_k=0$
	and $\tilde{v}_1+\cdots+\tilde{v}_l<0$ for $1\leq l<k$.
	Therefore 
	\begin{align*}
	    \sum_{i=1}^k (k+1-i)d_i\left(\frac{v_i}{d_i}-\frac{w}{d}\right)
	    =\sum_{l=1}^k \left(\sum_{i=1}^l \tilde{v}_i \right)<0. 	\end{align*}
	It follows that 
	\begin{align}\label{lambda1}
	    \left\langle \lambda, 	\sum_{i=1}^k \psi_i+\sum_{i=1}^k v_i \tau_{d_i}-\frac{3}{2}\mathfrak{g}^{\lambda>0}
	-w \tau_d \right\rangle <-\left\langle \lambda, \frac{3}{2}\mathfrak{g}^{\lambda>0} \right\rangle.
	\end{align}
On the other hand, we claim that for any weight $\chi \in \textbf{W}(d)_{0}$, we have that 
\begin{equation}\label{boundchi}
    \langle \lambda, \chi\rangle \geq -\left\langle \lambda, \frac{3}{2}\mathfrak{g}^{\lambda>0} \right\rangle.
\end{equation}
We thus obtain a contradiction with \eqref{lambda1} and \eqref{lambda0}. 
To prove \eqref{boundchi}, write
$\chi=\sum_{i, j}c_{ij}(\beta_i-\beta_j)$
with $0\leq c_{ij} \leq 3/2$. Then 
\begin{align*}
\langle \lambda, \chi \rangle \geq \sum_{\langle \lambda, \beta_i-\beta_j \rangle<0}
c_{ij}\langle \lambda, \beta_i-\beta_j \rangle 
\geq \frac{3}{2}\sum_{\langle \lambda, \beta_i-\beta_j \rangle<0}
\langle \lambda, \beta_i-\beta_j \rangle= 
-\left\langle \lambda, \frac{3}{2}\mathfrak{g}^{\lambda>0} \right\rangle.
\end{align*}
\end{proof}

\section{Integral generator of equivariant K-theory of quasi-BPS categories}
Recall the graded quasi-BPS category $\mathbb{S}^{\rm{gr}}(d)_v$, 
which is equivalent via Koszul duality to \[\mathbb{T}(d)_v \subset D^b(\mathscr{C}(d)).\]
In~\cite{PT0}, we proved that 
the monomials $[\mathcal{E}_{d_1, v_1}] \ast \cdots \ast [\mathcal{E}_{d_k, v_k}]$
for $v_i/d_i=v/d$
give a basis of the $\mathbb{F}$-vector space
$K_T(\mathbb{T}(d)_v)\otimes_{\mathbb{K}} \mathbb{F}$. 
In this section, we consider the torsion free 
equivariant K-theory defined by 
\begin{align*}
    K_T(\mathbb{T}(d)_v)':=K_T(\mathbb{T}(d)_v)/(\mathbb{K}\text{-torsion}). 
\end{align*}
 It is conjectured by Schiffmann--Vasserot~\cite[Conjecture~7.13]{SV} that $K_T(\mathscr{C}(d))$ is torsion free as a $\mathbb{K}$-module, 
so conjecturally $K_T(\mathbb{T}(d)_v)'$ is isomorphic to 
$K_T(\mathbb{T}(d)_v)$. 
We compute $K_T(\mathbb{T}(d)_v)'$ when $\gcd(d, v)=1$, and when
$(d, v)=(2, 0)$. Theorem~\ref{lem:support} plays a key role 
in the case of $\gcd(d, v)=1$. 

\subsection{Sheaves supported over the small diagonal}\label{ss32}

Let 
\begin{align*}
    \Omega_{\mathscr{C}(d)}[-1]:=
    \Spec \mathrm{Sym}\left(\mathbb{T}_{\mathscr{C}(d)}[1]\right)
\end{align*}
be
the $(-1)$-shifted cotangent of $\mathscr{C}(d)$. 
An object $\mathcal{E} \in D^b(\mathscr{C}(d))$ has singular support~\cite{AG}:
\begin{align*}
    \mathrm{Supp}^{\rm{sg}}(\mathcal{E})
    \subset  (\Omega_{\mathscr{C}(d)}[-1])^{\rm{cl}}. 
\end{align*}
We identify the right hand side with 
$\mathrm{Crit}(\Tr W_d) \subset \mathcal{X}(d)$, 
which is isomorphic to the moduli stack of 
zero-dimensional coherent sheaves on $\mathbb{C}^3$
of length $d$. 
Under the Koszul duality equivalence (\ref{equiv:Phi}), we have by~\cite[Proposition~2.3.9]{T} that
\begin{align*}
    \mathrm{Supp}^{\rm{sg}}(\mathcal{E})
    =\mathrm{Supp}(\Phi(\mathcal{E})). 
\end{align*}
In the following lemma, we show that the 
K-theory class of an object in $D^b_T(\mathscr{C}(d))$ with singular support 
contained in $\pi^{-1}(\Delta)$ has a certain divisibility property. 
The proof is inspired by the proof of wheel conditions 
in~\cite[Theorem~2.9, Corollary~2.10]{Zhao}, \cite[Proposition~2.11]{N2}:
\begin{lemma}\label{lemma:div}
	For any $\mathcal{E} \in D^b_T(\mathscr{C}(d))$ whose singular support is 
	contained in $\pi^{-1}(\Delta)$, the element 
	$i_{\ast}[\mathcal{E}] \in K_T(\mathcal{Y}(d))=
	\mathbb{K}[z_1^{\pm}, \ldots, z_d^{\pm 1}]^{\mathfrak{S}_d}$
	is divisible by 
		\begin{align}\label{divis}
	(q_1-1)^{d-1}(q_2-1)^{d-1}(q_1 q_2-1)^{d-1} \in \mathbb{K}. 
	\end{align}
	\end{lemma}
\begin{proof}
By \cite[Lemma 4.9]{PT0}, it is enough to show that, for
	$\mathcal{F} \in \mathrm{MF}_T(\mathcal{X}(d), \Tr W)$ supported
	on $\pi^{-1}(\Delta)$, 
	its image under the forget-the-potential map \eqref{definition:forgetpotential}: 
	\begin{align}\label{elem:F}
		\Theta([\mathcal{F}]):=[\mathcal{F}^{0}]-[\mathcal{F}^1] \in K_T(\mathrm{MF}(\mathcal{X}(d), 0))
		=\mathbb{K}[z_1^{\pm 1}, \ldots, z_d^{\pm d}]^{\mathfrak{S}_d}
	\end{align}
is divisible by $(q_1-1)^{d-1}(q_2-1)^{d-1}(q_1 q_2-1)^{d-1}$. 
We consider the morphism 
\begin{align*}
	h \colon \mathbb{C}^{d-1} \setminus \{0\} &\to \mathrm{Crit}(\Tr W_d) \subset 
	\mathfrak{g}^{\oplus 3}\\
	(t_1, \ldots, t_{d-1}) &\mapsto	\left(0, 0, (t_1, \ldots, t_{d-1}, 0) \right),
	\end{align*}
where $(t_1, \ldots, t_{d-1}, 0)$ is the diagonal matrix. 
Let $T=(\mathbb{C}^{\ast})^2$ act on $\mathbb{C}^{d-1}$ with weight $q_1^{-1}q_2^{-1}$ and 
let the maximal torus $T(d) \subset GL(d)$ act on $\mathbb{C}^{d-1}$ trivially. 
Then the above morphism $h$ is $T \times T(d)$-equivariant, so it induces a 
morphism 
\begin{align*}
    h \colon 
    (\mathbb{C}^{d-1}\setminus \{0\})/(T \times T(d))
    \to \mathfrak{g}^{\oplus 3}/(T \times T(d))
    \to \mathfrak{g}^{\oplus 3}/(T \times GL(d)). 
\end{align*}
By the localization sequence 
\begin{align*}
	K_{T \times T(d)}(\{0\}) \to K_{T \times T(d)}(\mathbb{C}^{d-1})
	\to K_{T \times T(d)}(\mathbb{C}^{d-1} \setminus \{0\}) \to 0,
	\end{align*}
we have an isomorphism 
\begin{align*}
	K_{T \times T(d)}(\mathbb{C}^{d-1} \setminus \{0\}) \cong 
	\mathbb{K}[z^{\pm 1}_1, \ldots, z^{\pm 1}_d]/(1-q_1^{-1}q_2^{-1})^{d-1}. 
	\end{align*}
	Note that we have 
	\begin{align*}
	    \pi \circ h(t_1, \ldots, t_{d-1})=\{(0, 0, t_1), \ldots, (0, 0, t_{d-1}), (0, 0, 0)\} \in 
	    \mathrm{Sym}^d(\mathbb{C}^3), 
	\end{align*}
	in particular the image of $h$ does not intersect $\pi^{-1}(\Delta) \subset 
	\mathrm{Crit}(\Tr W_d)$. 
	We thus have that
\[h^{\ast}(\Theta([\mathcal{F}]))=0\text{ in }K_{T \times T(d)}(\mathbb{C}^{d-1} \setminus \{0\}).\]
Therefore the element (\ref{elem:F}) is divisible by $(q_1 q_2-1)^{d-1}$. 
By replacing $h$ with
\begin{align*}
	(t_1, \ldots, t_{d-1}) &\mapsto (0, (t_1, \ldots, t_{d-1}, 0), 0), \\
		(t_1, \ldots, t_{d-1}) & \mapsto 
	((t_1, \ldots, t_{d-1}, 0), 0, 0),
	\end{align*}
the element (\ref{elem:F}) is also divisible by $(q_1-1)^{d-1}$, $(q_2-1)^{d-1}$, respectively. 
\end{proof}
By combining the above lemma 
with Theorem~\ref{lem:support}, we obtain 
the following: 
\begin{cor}\label{cor:div}
	For any object $\mathcal{E} \in \mathbb{T}_T(d)_v$ with 
	$\gcd(d, v)=1$, the element 
	\begin{align*}
	i_{\ast}[\mathcal{E}] \in K_T(\mathcal{Y}(d))=\mathbb{K}[z_1^{\pm 1}, \ldots, z_d^{\pm 1}]^{\mathfrak{S}_d}
	\end{align*}
	is divisible by (\ref{divis}). 
	In particular, the element 
	$i_{\ast}[\mathcal{E}_{d, v}]$
	in (\ref{elem:E}) for $\gcd(d, v)=1$
	is divisible by (\ref{divis}). 
	\end{cor}
	\begin{remark}
	It is not clear from the expression (\ref{elem:E}) that $i_{\ast}[\mathcal{E}_{d, v}]$ is divisible 
	by (\ref{divis}) when $\gcd(d, v)=1$. 
	Further, the condition $\gcd(d, v)=1$ is necessary 
	for the above divisibility. 
	Indeed, for $d=2$, a direct computation shows that 
	\begin{align*}
	    i_{\ast}[\mathcal{E}_{2, 0}]&=
	    (1-q_1^{-1})(1-q_2^{-1})(1-q_1^{-1}-q_2^{-1}-q_1^{-1}q_2^{-1}+z_1^{-1}z_2+z_1 z_2^{-1}), \\
	    i_{\ast}[\mathcal{E}_{2, 1}]&=
	    (1-q_1^{-1})(1-q_2^{-1})(1-q_1^{-1}q_2^{-1})(z_1+z_2). 
	\end{align*}
	The element $i_{\ast}[\mathcal{E}_{2, 0}]$
	is not divisible by (\ref{divis}). 
	In particular, the singular support of $\mathcal{E}_{2, 0}$ is not included in 
	$\pi^{-1}(\Delta)$. 
	\end{remark}

\subsection{Integral generator for the coprime case}\label{ss33}
Recall that $K_T(\mathbb{T}(d)_v)$ is expected to be torsion 
free as a $\mathbb{K}$-module. 
We also expect that $K_T(\mathbb{T}(d)_v)$ is freely generated by 
$[\mathcal{E}_{d, v}]$ if $\gcd(d, v)=1$. 
The following theorem, which is an application of Corollary~\ref{cor:div}, 
gives evidence towards this expectation. 

\begin{thm}
\label{thm:generatordv}
Consider a pair $(d, v)\in \mathbb{N}\times\mathbb{Z}$ with $\gcd(d, v)=1$. 
There is an isomorhism of $\mathbb{K}$-modules: 
	\begin{align*}
		K_T(\mathbb{T}(d)_v)' \cong\mathbb{K}[\mathcal{E}_{d, v}]. 
	\end{align*}
	\end{thm}
\begin{proof}
By the isomorphism (\ref{isom:S}), 
the $\mathbb{K}$-module $K_T(\mathbb{T}(d)_v)'$
is isomorphic to the image of 
\begin{align*}
	i_{\ast} \colon K_T(\mathbb{T}(d)_v) \to \mathbb{K}[z_1^{\pm 1}, \ldots, z_d^{\pm}]^{\mathfrak{S}_d}.
	\end{align*}
	It is enough to show that the image of the 
	above morphism 
is generated by $i_{\ast}[\mathcal{E}_{d, v}]$ as a $\mathbb{K}$-module. 
By Corollary~\ref{cor:div}, we have 
\begin{align*}
	\mathrm{Im}(i_{\ast}) \subset (q_1-1)^{d-1}(q_2-1)^{d-1}(q_1 q_2-1)^{d-1}
	\mathbb{K}[z_1^{\pm 1}, \ldots, z_d^{\pm 1}]^{\mathfrak{S}_d}. 
	\end{align*}
Moreover, $\mathrm{Im}(i_{\ast})\otimes_{\mathbb{K}} \mathbb{F}$ is
generated by $i_{\ast}[\mathcal{E}_{d, v}]$ over $\mathbb{F}$
by~\cite[Theorem~4.12]{PT0}. 
It is thus enough to show that $i_{\ast}[\mathcal{E}_{d, v}]$ is written as 
\begin{align*}
	i_{\ast}[\mathcal{E}_{d, v}]
	=(q_1-1)^{d-1}(q_2-1)^{d-1}(q_1 q_2-1)^{d-1} \cdot E,
	\end{align*}
where $E$ is not divisible by a non-unit element in $\mathbb{K}$. 

By (\ref{elem:E}), we have that
\begin{align}\label{sym:m}
\mathrm{Sym}\left(\frac{z_1^{m_1}\ldots z_d^{m_d}}{(1-q^{-1}z_1^{-1}z_2)\cdots 
	(1-q^{-1}z_{d-1}^{-1}z_d)} 
\cdot \prod_{j>i}\xi(z_i z_j^{-1})\right)=(q_1 q_2-1)^{d-1} \cdot E.
\end{align}
By setting 
\begin{multline}\label{def:fd}
	f_d(z_1, \ldots, z_d):=\prod_{i=1}^{d-1}(z_{i+1}-q_1^{-1}z_i)(z_{i+1}-q_2^{-1}z_i) \cdot\\
		\prod_{j>i+1}(z_j-q_1^{-1}z_i)(z_j-q_2^{-1}z_i)(z_j-q z_i),
\end{multline}
we can write (\ref{sym:m}) as 
\begin{align}\label{sym:m2}
	(-q)^{-\frac{1}{2}(d-1)(d-2)}(z_1 \cdots z_d)^{2-d} 
	\cdot \mathrm{Sym}\left(\frac{z_1^{m_1}\cdots z_{d-1}^{m_{d-1}}z_d^{m_d-1} f_d(z_1, \ldots, z_d)}{\prod_{j>i}(z_j-z_i)}   \right). 
	\end{align}
Plug in $z_i=q_1^i$ for $1\leq i\leq d$ in the formula \eqref{sym:m2}. The only non-zero term in the sum above corresponds to the identity permutation. The factors of this term not in $\mathbb{K}^*$ divide 
\begin{equation}\label{zeroesq1}
q_1^{a+1}-1,\, q_1^a-q_2^{-1},\text{ or }q_1^a-q_2\text{ for some }a\geq 1.
\end{equation}
Next plug in $z_i=q_2^i$ for $1\leq i\leq d$ in the formula \eqref{sym:m2}. The only non-zero term in the sum of \eqref{sym:m2} corresponds to the identity permutation. The factors of this term not in $\mathbb{K}^*$ divide 
\begin{equation}\label{zeroesq2}
q_2^{a+1}-1,\, q_2^a-q_1^{-1},\text{ or }q_2^a-q_1\text{ for some }a\geq 1.
\end{equation}
The only factors which divide terms in both sets \eqref{zeroesq1} and \eqref{zeroesq2} are $q_1-q_2$ and $q_1q_2-1$. The factor $q_1q_2-1$ appears with multiplicity $d-1$ as it corresponds to $z_{i+1}-q_2^{-1}z_i$ for $1\leq i\leq d-1$.

It suffices to show that $q_1-q_2$ does not divide \eqref{sym:m2}. 
We will be using computations from \cite[Section 2]{N}. Note that $q_1, q_2, q$ from our paper correspond to $q_1^{-1}, q_2^{-1}, q^{-1}$ in loc. cit. Further, the weight $v\in \mathbb{Z}$ corresponds to $k\in\mathbb{Z}$ in loc. cit.
By \cite[Equation (2.35)]{N}, the equality $P_{d, k}=E_{d, k}$ for $\gcd(d, k)=1$ in loc. cit., see \cite[Equations (2.6) and (2.35)]{N}, and the isomorphism between shuffle algebras $\mathcal{S}_\mathbb{F}\xrightarrow{\sim}\mathcal{S}'_\mathbb{F}$, we can write 
\[(q_1q_2-1)^{d-1}\cdot E=y+t
\text{ in }\mathcal{S}\otimes_{\mathbb{K}}\mathbb{K}\left[\frac{1}{1-q_2^{-1}}, \frac{1}{1-q^{-1}}\right]
\] for $t$ a $\mathbb{K}$-torsion element and 
\begin{equation}\label{sym:mm}
y:=\frac{(1-q_2^{-1})(1-q)^v}{(1-q_2^{-1})^v(1-q^{-1})}\cdot
\mathrm{Sym}\left(\frac{z_1^{m_1}\ldots z_d^{m_d}}{(1-q_2z_1^{-1}z_2)\cdots 
	(1-q_2z_{d-1}^{-1}z_d)} 
\cdot \prod_{j>i}\xi(z_i z_j^{-1})\right).
\end{equation}
It suffices to show that $q_1-q_2$ does not divide $y$.
By setting 
\begin{multline}\label{def:gd}
	g_d(z_1, \ldots, z_d):=\prod_{i=1}^{d-1}(z_{i+1}-q_1^{-1}z_i)(z_{i+1}-qz_i) \cdot\\
		\prod_{j>i+1}(z_j-q_1^{-1}z_i)(z_j-q_2^{-1}z_i)(z_j-q z_i),
\end{multline}
we can write the $\mathrm{Sym}(-)$ term of $y$ as 
\begin{align}\label{sym:mm2}
	(-q_2)^{-\frac{1}{2}(d-1)(d-2)}(z_1 \cdots z_d)^{2-d} 
	\cdot \mathrm{Sym}\left(\frac{z_1^{m_1}\ldots z_{d-1}^{m_{d-1}}z_d^{m_d-1} g_d(z_1, \ldots, z_d)}{\prod_{j>i}(z_j-z_i)}   \right). 
	\end{align}
	It suffices to show that $q_1-q_2$ does not divide \eqref{sym:mm2}.
	Let $z_i=q^{-i}$ for $1\leq i\leq d$. The only non-zero term in the sum of \eqref{sym:mm2} corresponds to the identity permutation. The factors of this term not in $\mathbb{K}^*$ divide
	\[q^{-a-1}-1, q^{-a}-q_1^{-1}, q^{-a}-q_1\text{ for some }a\geq 1.\] 
	None of these polynomials is divisible by $q_1-q_2$, and the conclusion thus follows. 
\end{proof}

\subsection{Integral generator for \texorpdfstring{$K_T(\mathbb{T}(2)_0)'$}{T2}}
\label{subsection:integral20}
The computation of $K_T(\mathbb{T}(d)_v)'$
is more subtle when $\gcd(d, v)>1$, 
since the monomials $[\mathcal{E}_{d_1, v_1}] \ast \cdots \ast [\mathcal{E}_{d_k, v_k}]$
for $v_i/d_i=v/d$
do not generate it over $\mathbb{K}$,
see Remark~\ref{rmk:overK}. We need to find
other objects giving $\mathbb{K}$-basis of 
$K_T(\mathbb{T}(d)_v)'$. 
Here we give a computation for $(d, v)=(2, 0)$. 

Let $V$ be a two-dimensional vector space and let
$\mathfrak{sl} \subset \mathfrak{g}=\mathrm{End}(V)$ be its 
traceless part. Note that 
\begin{align*}
	[\mathfrak{sl}]=1+z_1^{-1}z_2+z_1 z_2^{-1} \in K(BGL(2))=\mathbb{Z}[z_1^{\pm 1}, z_2^{\pm 1}]^{\mathfrak{S}_2}. 
	\end{align*}
The structure sheaf of the classical truncation $\mathscr{C}(2)^{\rm{cl}}$ 
fits into the exact sequence 
\begin{align}\label{exact:1}
	0 \to \mathcal{O}_{\mathcal{Y}(2)}(q_1^{-1}q_2^{-2}) \oplus \mathcal{O}_{\mathcal{Y}(2)}(q_1^{-2}q_2^{-1}) \stackrel{A}{\to} \mathfrak{sl} \otimes \mathcal{O}_{\mathcal{Y}(2)}(q_1^{-1}q_2^{-1}) \stackrel{B}{\to} \mathcal{O}_{\mathcal{Y}(2)} \to 
	\mathcal{O}_{\mathscr{C}(2)^{\rm{cl}}} \to 0.
\end{align}
Here over $(X, Y) \in \mathfrak{g}^{\oplus 2}$, the maps $A$, $B$ are given by 
\begin{align*}
	A|_{(X, Y)}=(2X-\Tr X \cdot I, 2Y-\Tr Y \cdot I), \ 
	B|_{(X, Y)}(Z)=\Tr(Z[X, Y]). 
	\end{align*}
We set $M_1:=\mathcal{O}_{\mathscr{C}(2)^{\rm{cl}}}(q_1 q_2)$. 
We also define 
a coherent sheaf $M_2$ on $\mathscr{C}(2)^{\rm{cl}}$ by the exact sequence 
\begin{align}\label{exact:2}
		0 \to \mathcal{O}_{\mathcal{Y}(2)}(q_1^{-1}q_2^{-1}) \stackrel{B^{\vee}}{\to} \mathfrak{sl} \otimes \mathcal{O}_{\mathcal{Y}(2)} \stackrel{A^{\vee}}{\to}
		 \mathcal{O}_{\mathcal{Y}(2)}(q_1) \oplus 
	\mathcal{O}_{\mathcal{Y}(2)}(q_2) \to M_2 \to 0. 
	\end{align}
	The sequences (\ref{exact:1}), (\ref{exact:2}) 
	are the Eagon-Northcott complex and the Buchsbaum-Rim complex
	associated with $A^{\vee} \colon \mathfrak{sl} \otimes \mathcal{O}_{\mathcal{Y}(2)}
	\to \mathcal{O}_{\mathcal{Y}(2)}(q_1) \oplus 
	\mathcal{O}_{\mathcal{Y}(2)}(q_2)$, 
	respectively. In particular they are exact, see~\cite[Theorem~A2.10]{Ebud}.
From the exact sequences (\ref{exact:1}), (\ref{exact:2}),
we have $M_1, M_2 \in \mathbb{T}(2)_0$. 
\begin{prop}\label{prop:20}
	As a $\mathbb{K}$-module, we have  
	\begin{align}\label{KT(2)}
		K_T(\mathbb{T}(2)_0)'=
		\mathbb{K}[M_1] \oplus \mathbb{K}[M_2]. 
		\end{align}
	\end{prop}
\begin{proof}
	It is enough to show that the image of 
	\begin{align}\label{image:i}
		i_{\ast} \colon K_T(\mathbb{T}(2)_0) \to 
		\mathbb{K}[z_1^{\pm 1}, z_2^{\pm 1}]^{\mathfrak{S}_2}
		\end{align} 
	is generated by $i_{\ast}M_1$, $i_{\ast}M_2$ and that $i_{\ast}M_1, i_{\ast}M_2$
	are linearly independent over $\mathbb{F}$. 
	Below we omit $i_{\ast}$
	from the notation $i_{\ast}M_j$. 
	From the exact sequences (\ref{exact:1}), (\ref{exact:2}), we have 
	\begin{align}\label{compute:M}
	[M_1]=q_1 q_2+q_1^{-1}+q_2^{-1}-\mathfrak{sl}, \ 
		[M_2]=q_1^{-1} q_2^{-1}+q_1+q_2-\mathfrak{sl}. 
		\end{align}
	A direct computation shows that 
	\begin{align}\label{dcompute}
		[\mathcal{E}_{1, 0} \ast \mathcal{E}_{1, 0}]&=
1+q_1^{-1}+q_2^{-1}+q_1^{-1}q_2^{-2}+q_1^{-2}q_2^{-1}+q_1^{-2}q_2^{-2}-2\mathfrak{sl} q_1^{-1}q_2^{-1} \\
\notag	&=q_1^{-1}q_2^{-1}([M_1]+[M_2]), \\
\notag [\mathcal{E}_{2, 0}]&=
	 (1-q_1^{-1})(1-q_2^{-1})(\mathfrak{sl}-q_1^{-1}-q_2^{-1}-q_1^{-1}q_2^{-1}) \\
\notag	 &=(q_1^{-1}+q_2^{-1})[M_1]-(q_1^{-1}q_2^{-1}+1)[M_2]. 	
	\end{align}
 Here 
 $[\mathcal{E}_{1, 0} \ast \mathcal{E}_{1, 0}]$
 is computed from (\ref{def:shuffle}) and 
 $[\mathcal{E}_{2, 0}]$ is computed from (\ref{elem:E}).  
Therefore we have 
\begin{align}\label{emb:M}
	\mathbb{K}[\mathcal{E}_{2, 0}] \oplus 
	\mathbb{K}[\mathcal{E}_{1, 0} \ast \mathcal{E}_{1, 0}] 
	\subset \mathbb{K}[M_1] \oplus \mathbb{K}[M_2]. 
	\end{align}
	Since we have 
	\begin{align}\label{det:q}
	\det	\begin{pmatrix}
			q_1^{-1}q_2^{-1} & q_1^{-1} q_2^{-1} \\
			q_1^{-1}+q_2^{-1} & -(q_1^{-1}q_2^{-1}+1) \\
		\end{pmatrix}
	=-q_1^{-1}q_2^{-1}(1+q_1^{-1})(1+q_2^{-1}) \in \mathbb{K} \setminus \{0\},                                          			\end{align}
	the embedding (\ref{emb:M}) is an isomorphism after taking 
	$\otimes_{\mathbb{K}} \mathbb{F}$. 
	In particular, $[M_1]$, $[M_2]$ are linearly independent over $\mathbb{F}$. 
	
	By the above argument, we have 
	\begin{align*}
	K_T(\mathbb{T}(2)_0)\otimes_{\mathbb{K}}{\mathbb{F}}=
	\mathbb{F}[\mathcal{E}_{2, 0}] \oplus 
	\mathbb{F}[\mathcal{E}_{1, 0} \ast \mathcal{E}_{1, 0}]= 
	\mathbb{F}[M_1] \oplus \mathbb{F}[M_2],
	\end{align*}
	where the first isomorphism is proved in~\cite[Theorem~4.12]{PT0}. 
   It follows that any element in the image of (\ref{image:i}) 
   is written as $a_1[M_1]+a_2[M_2]$ for $a_1, a_2 \in \mathbb{F}$. 
   As it lies in $\mathbb{K}[z_1^{\pm 1}, z_2^{\pm 1}]$, from (\ref{compute:M}) 
   we have 
   \begin{align*}
   	a_1(q_1 q_2+q_1^{-1}+q_2^{-1})+a_2(q_1^{-1}q_2^{-1}+q_1+q_2) \in \mathbb{K}, \ 
   	a_1+a_2 \in \mathbb{K}. 
   	\end{align*}
   By solving the above equation, there exist $b, c \in \mathbb{K}$ such that 
   \begin{align}\label{a12}
   	a_1=\frac{b}{\Delta}, \ a_2=-\frac{b}{\Delta}+c
   	\end{align}
   where $\Delta$ is given by 
   \begin{align*}
   	\Delta &:=q_1^{-1}+q_2^{-1}+q_1q_2-q_1-q_2-q_1^{-1} q_2^{-1} \\
   	&=q_1^{-1}q_2^{-1}(q_1-1)(q_2-1)(q_1 q_2-1). 
   	\end{align*}
   It is enough to show that $b$ is divisible by $\Delta$. 
   If $a_1, a_2$ are given by (\ref{a12}), then 
   $a_1[M_1]+a_2[M_2]=b+c[M_2]$. As we assumed that $a_1[M_1]+a_2[M_2]$ lies in 
   the image of (\ref{image:i}), by the wheel condition~\cite[Proposition~2.11]{N2} we have 
   \begin{align}\label{wcond}
   	(b+c[M_2])\big|_{z_i=q_1^{-1}z_j=q_1^{-1}q_2^{-1}z_k}=
   		(b+c[M_2])\big|_{z_i=q_2^{-1}z_j=q_1^{-1}q_2^{-1}z_k}=0
   	\end{align}
   	unless $i=j=k$, see Remark~\ref{rmk:wheel}. 
   	   By setting $i=k=1$ and $j=2$,
   we obtain $b|_{q_1 q_2=1}=0$. Similarly we also obtain 
   $b|_{q_1=1}=b|_{q_2=1}=0$, so $b$ is divisible by $\Delta$.   
	\end{proof}
\begin{remark}\label{rmk:wheel}
    	In~\cite[Proposition~2.11]{N2}, it is assumed that 
   	$i$, $j$, and $k$ are pairwise distinct
   	for the identity (\ref{wcond}), 
   	but the proof in loc. cit.
   	works unless $i=j=k$. Indeed, using the notation in loc. cit., let $\phi_e=(x_{ij})$ for $x_{bc} \neq 0$, $x_{ij}=0$ for $(i, j) \neq (b, c)$, 
   	$\phi_{e}^{\ast}=(y_{ij})$ for $y_{ab} \neq 0$, $y_{ij}=0$ for $(i, j) \neq (a, b)$. 
   	Then $\phi_e^{\ast} \phi_e$ is concentrated on $(a, c)$ with value 
   	$y_{ab}x_{bc} \neq 0$, $\phi_e \phi_e^{\ast}$ is concentrated on $(b, b)$ (either zero or non-zero). 
   	Therefore if $[\phi_e, \phi_e^{\ast}]=0$ holds, then we must have $a=b=c$.
   	Unless $a=b=c$, we have $\mu_n^{-1}(0) \cap V_e=\emptyset$
   	in the notation of the proof of~\cite[Proposition~2.11]{N2}. The rest of the argument is 
   	verbatim.  
    
\end{remark}
	\begin{remark}
Note that \[\mathscr{C}(2)^{\rm{cl}}\cong[R/GL(2)] \times \mathbb{C}^2,\]
where $R$ is the determinantal variety of $(3 \times 2)$-matrices, 
and $M_1 \oplus M_2$ gives a noncommutative resolution of 
$R \times \mathbb{C}^2$ by~\cite[Theorem~A]{BLBergh}. 
\end{remark}

\begin{remark}\label{rmk:overK}
	Since (\ref{det:q}) is not invertible in $\mathbb{K}$, 
	the inclusion (\ref{emb:M}) is not an isomorphism if 	we do not take $\otimes_{\mathbb{K}} \mathbb{F}$.  	\end{remark}
	
	\begin{remark}\label{rmk:Pdw}
For $(d, v)\in \mathbb{N}\times\mathbb{Z}$ with $\gcd(d, v)=1$ and $n\geq 1$, let $P_{nd, nv}$
be defined as in~\cite[Equation (1.2)]{Neshuffle}:
\begin{multline}\label{def:Pdv}
P_{nd,nv}:= \frac{(q_1^{-1}-1)^{nd}(q_2^{-1}-1)^{nd}}{(q_1^{-n}-1)(q_2^{-n}-1)}\cdot\\
	\text{Sym}\left(
	\frac{\prod_{i=1}^{nd} z_i^{\big\lfloor \frac{iv}{d}\big\rfloor-\big\lfloor \frac{(i-1)v}{d}\big\rfloor}}{\prod_{i=1}^{nd-1}\left(1-q^{-1}z_{i+1}z_i^{-1}\right)}\sum_{s=0}^{n-1}q^{-s}\frac{z_{d(n-1)+1}\ldots z_{d(n-s)+1}}{z_{d(n-1)}\ldots z_{d(n-s)}}\prod_{i<j}\xi\left(\frac{z_i}{z_j}\right)
	\right).
\end{multline}
Then we have 
\begin{align}\label{compute:P20}
	P_{2, 0} =q_1^{-2}q_2^{-2}(q_1-1)(q_2-1)(q_1 q_2-1)
	=q_1^{-1}q_2^{-1}(M_1-M_2). 	
\end{align}
Together with (\ref{dcompute}), we have 
\begin{align}\label{inclP}
	\mathbb{K}[P_{1, 0} \ast P_{1, 0}] \oplus \mathbb{K}[P_{2, 0}]
	\subsetneq \mathbb{K}[M_1] \oplus \mathbb{K}[M_2]
\end{align}
with cokernel $\mathbb{Z}/2$. The inclusion \eqref{inclP} is an isomorphism after $\otimes_{\mathbb{Z}} \mathbb{Q}$. 

More generally, we expect that 
\begin{align*}
K_T(\mathbb{S}(nd)_{nv})_{\mathbb{Q}}'=
\bigoplus_{n_1+\cdots+n_k=n} \mathbb{K}_{\mathbb{Q}}
[P_{n_1 d, n_1 v} \ast \cdots \ast P_{n_k d, n_k v}]. 
\end{align*}
More details will be discussed in~\cite{PTsym}.
\end{remark}

\section{The coproduct on quasi-BPS categories and K-theoretic BPS spaces}\label{s4}

Recall that $Q$ is the quiver with one vertex and three edges $x, y, z$. Recall the regular function \eqref{equation:NHilbTrW} induced from the super-potential $W:=x[y, z]$.
We will denote by $\widetilde{m}$ the Hall product \eqref{prel:hall} of $(Q, W)$ and by $m$ the Hall product \eqref{hall:ast}.

In Subsection \ref{s41}, we define the coproduct $\widetilde{\Delta}$ for the categories $\mathbb{S}^{\bullet}_{\ast}(d)_w$. In Subsection \ref{s42}, 
we prove the compatibility of the product and coproduct for the quasi-BPS categories $\mathbb{S}^{\bullet}_T(d)_w$, see Theorem \ref{prodcoprod}. In Subsection \ref{s43}, 
we use the Koszul equivalence to define the coproduct $\Delta$ for the categories $\mathbb{T}_{\ast}(d)_v$ and to
check the compatibility between the product and the coproduct for the quasi-BPS categories $\mathbb{T}_T(d)_v.$ 

Recall the definition of $\mathcal{D}_{d,v}$ for $(d, v)\in \mathbb{N}\times\mathbb{Z}$ for $\gcd(d, v)=1$ from \eqref{ddv}.
In Subsection \ref{s44},
we show that $K(\mathcal{D}_{d,v})_{\mathbb{F}}$ is isomorphic to the $\mathbb{F}$-algebra of symmetric polynomials, see Proposition \ref{isolambda}. Further, 
we 
consider the space of primitive elements \[\mathrm{P}(nd, nv)\subset K_T(\mathbb{T}(nd)_{nv})\] which we regard as an analogue of cohomological BPS spaces in K-theory, see Proposition \ref{isolambda} and the equality \eqref{Omegad}.

Note that the definition of Hall multiplication involves attracting stacks of antidominant cocharacters. The definition of the coproduct is through attracting stacks of dominant cocharacters. In this section, we will use dominant cocharacters.

\subsection{Preliminaries}
\label{s41}

Let $d\in \mathbb{N}$. 
Recall the stack of representation of $Q$ of dimension $d$:
\[\X(d):=R(d)/G(d):=\mathfrak{g}^{\oplus 3}/GL(d).\]
In this section, we allow partitions which have terms equal to zero.
For $w\in \mathbb{Z}$, let $H_{d, w}$ be the set of partitions $(d_i, w_i)_{i=1}^k$ of $(d, w)$ with $d_i\geq 1$ for $i\in \{1,\ldots, k\}$ such that for $v_i:=w_i-d_i\left(\sum_{i>j}d_j-\sum_{j>i}d_j\right)$ for $i\in \{1,\ldots, k\}$, we have 
\begin{equation}\label{def:equalslopes}
    \frac{v_1}{d_1}=\ldots=\frac{v_k}{d_k}.
\end{equation}
For a partition $A=(d_i, w_i)_{i=1}^k$ with terms possibly equal to zero,
let $I\subset \{1,\ldots, k\}$ be the subset of $i$ with $d_i\geq 1$, and 
define the partition of $(d, w)$ with non-zero terms $\overline{A}:=(d_i, w_i)_{i\in I}$. 

\subsubsection{} 
For $\lambda$ and $\mu$ two cocharacters, let $A_\lambda$ be the set of $(T(d)\times T)$-weights $\beta$ of $R(d)$ such that $\langle \lambda, \beta\rangle>0$, let $I^\mu_\lambda\subset A_\lambda$ be the set of weights such that $\langle \mu, \beta\rangle<0$, and let $A^\mu_\lambda\subset A_\lambda$ be the set of weights such that $\langle \mu, \beta\rangle=0$. Let $J^\mu_\lambda$ be the set of weights $\beta$ of $\mathfrak{g}$ such that $\langle \lambda, \beta\rangle>0$ and $\langle \mu, \beta\rangle<0$.
Define the weights
\begin{align*}
    N^\mu_\lambda:=\sum_{I^\mu_\lambda}\beta,\
    \mathfrak{g}^\mu_\lambda:=\sum_{J^\mu_\lambda}\beta.
\end{align*}

\subsubsection{}\label{nu} Let $\lambda$ and $\mu$ be dominant cocharacters of $T(d)$
with associated partitions $A=(d_i)_{i=1}^k$ and $B=(e_i)_{i=1}^s$, respectively.
Let $W\cong \mathfrak{S}_d$ be the Weyl group of $GL(d)$,
let $W^\lambda\cong \times_{i=1}^k \mathfrak{S}_{d_i}$ be the Weyl
group of $GL(d)^\lambda$, and let $W^\mu$ be the Weyl group of $GL(d)^\mu$.
Define the set of cosets 
\[S^\mu_\lambda:= W^\mu\backslash W/W^\lambda.\]
A coset $C\in S^\mu_\lambda$ corresponds to partitions $(f_{ij})$ for $1\leq i\leq k$ and $1\leq j\leq s$ such that 
\begin{align*}
    \sum_{j=1}^s f_{ij}=d_i\mbox{ for }1\leq i\leq k,
    \
    \sum_{i=1}^k f_{ij}=e_j\mbox{ for }1\leq j\leq s.
\end{align*}
Let $\nu$ be a dominant cocharacter corresponding to the partition 
\begin{equation}\label{permutations10}
(f_{11},\cdots, f_{1s}, f_{21},\cdots, f_{2s}, \cdots, f_{k1},\cdots, f_{ks})
\end{equation}
and let $\kappa$ be a dominant cocharacter corresponding to the partition
\begin{equation}\label{permutations11}
(f_{11},\cdots, f_{k1}, f_{12},\cdots, f_{k2}, \cdots, f_{1s},\cdots, f_{ks}).
\end{equation}
Consider the permutation $w$ of $\mathfrak{S}_d$ of minimal length which permutes blocks of consecutive integers
\begin{multline}\label{permutations}
w=w_C\colon (f_{11},\cdots, f_{1s}, f_{21},\cdots, f_{2s}, \cdots, f_{k1},\cdots, f_{ks})\mapsto\\ (f_{11},\cdots, f_{k1}, f_{12},\cdots, f_{k2}, \cdots, f_{1s},\cdots, f_{ks}).
\end{multline}
Consider the partition 
\begin{equation}\label{partitionD}
  D=(f_{ij}, u_{ij}),  
\end{equation}
 where $f_{ij}$ are ordered as on \eqref{permutations10}, and assume that $\overline{D}$ is in $H_{d,w}$.
Then there also exists a partition $E$ with terms $(f_{ij}, \widetilde{u}_{ij})$ where $f_{ij}$ are ordered as on \eqref{permutations11} such that $\overline{E}$ is in $H_{d, w}$. Consider the order $1\leq l\leq ks$ of the pairs $(i, j)$ as on the first line on \eqref{permutations}.
Define the functor
\begin{align}\label{sw}
\text{sw}_C: \boxtimes_{l=1}^{ks} \mathbb{M}(f_{l})_{u_{l}}&\to 
 \boxtimes_{l=1}^{ks}\mathbb{M}(f_{w(l)})_{\widetilde{u}_{w(l)}}, \\ 
\notag\boxtimes_{l=1}^{ks} \mathcal{A}_{l} &\mapsto \boxtimes_{l=1}^{ks} \mathcal{A}_{w(l)}. 
\end{align}
which permutes the factors as in \eqref{permutations}.
Define 
\begin{align*}
\widetilde{\text{sw}}_C\colon \boxtimes_{i=1}^k \boxtimes_{j=1}^s \mathbb{M}(f_{ij})_{u_{ij}}&\to \boxtimes_{j=1}^s \boxtimes_{i=1}^k \mathbb{M}(f_{ij})_{\widetilde{u}_{ij}},\\
\mathcal{A}&\mapsto \mathrm{sw}_{C}\left(
\mathcal{A}\otimes
\mathcal{O}(-N^{w^{-1}\mu}_\lambda+\mathfrak{g}^{w^{-1}\mu}_\lambda)[|I^{w^{-1}\mu}_\lambda|-|J^{w^{-1}\mu}_\lambda|]
\right),
\end{align*}
where $\mathcal{O}(-N^{w^{-1}\mu}_\lambda+\mathfrak{g}^{w^{-1}\mu}_\lambda)$ is a one dimensional representation of $G^\nu\cong G^\kappa$.
Consider the maps
\[\X(d)^\kappa\xleftarrow{q_{\kappa^{-1}\mu}}\left(\X(d)^\mu\right)^{\kappa^{-1}\geq 0}\xrightarrow{p_{\kappa^{-1}\mu}}\X(d)^\mu.\]
Finally, define 
\[\widehat{m}_{DE}=p_{\kappa^{-1}\mu*}q_{\kappa^{-1}\mu}^*
\widetilde{\text{sw}}_C\colon \mathbb{M}_C\to \mathbb{M}_B.\]
There are analogous such functors for categories of (equivariant and/ or graded) matrix factorizations \[\widehat{m}_{CB}\colon \mathbb{S}^\bullet_{\ast, C}\to \mathbb{S}^\bullet_{\ast, B}.\]

\subsubsection{}\label{Mac}
We introduce some more maps and functors needed in the rest of this section. Let $\lambda$ and $\mu$ be dominant cocharacters and let $\nu$ be a dominant cocharacter corresponding to a partition in $S^\mu_\lambda$. The map $p_{\lambda^{-1}}$ factors as $p_{\lambda^{-1}}=\pi_{\lambda^{-1}} \iota_{\lambda^{-1}}$, where 
\begin{align*}
    \pi_{\lambda^{-1}}&: R(d)/G(d)^{\lambda^{-1}\geq 0}\to R(d)/G(d),\\
    \iota_{\lambda^{-1}}&: R(d)^{\lambda^{-1}\geq 0}/G(d)^{{\lambda^{-1}}\geq 0}\to R(d)/G(d)^{{\lambda^{-1}}\geq 0}.
\end{align*}
There are similarly defined maps
\begin{align*}
    \pi_{\kappa^{-1}\mu}&: R(d)^\mu/\left(G(d)^{\mu}\right)^{{\kappa^{-1}}\geq 0}\to R(d)^\mu/G(d)^\mu,\\
    \iota_{\kappa^{-1}\mu}&: \left(R(d)^\mu\right)^{\kappa^{-1}\geq 0}/\left(G(d)^{\mu}\right)^{\kappa^{-1}\geq 0}\to R(d)^\mu/\left(G(d)^{\mu}\right)^{\kappa^{-1}\geq 0}.
\end{align*}

\subsubsection{} Let $\mu$ be a dominant cocharacter of $T(d)$,
let $b\in\mathbb{Z}$, and let $D^b\left(\X(d)^{\mu\geq 0}\right)_{\leq b}$ be the subcategory of $D^b\left(\X(d)^{\mu\geq 0}\right)$ generated by complexes $q_\mu^*\mathcal{A}$ for $\mathcal{A}\in D^b\left(\X(d)^\mu\right)_i$ and $i\leq b$. There is a semiorthogonal decomposition 
\[D^b\left(\X(d)^{\mu\geq 0}\right)_{\leq b}=\Big\langle D^b\left(\X(d)^{\mu\geq 0}\right)_{\leq b-1}, D^b\left(\X(d)^{\mu\geq 0}\right)_{b} \Big\rangle\]
and there are equivalences $q_\mu^*: D^b\left(\X(d)^{\mu}\right)_{b}\xrightarrow{\sim} D^b\left(\X(d)^{\mu\geq 0}\right)_{b}$. 
We define the functor 
 \[\beta_b: D^b\left(\X(d)^{\mu\geq 0}\right)_{\leq b}\to D^b\left(\X(d)^{\mu\geq 0}\right)_{b}\]
 to be the projection with respect
 to the above semiorthogonal decomposition. 
\subsubsection{} 
Let $B=(d_i, w_i)_{i=1}^k$ be a partition
of $(d, w)$ in $H_{d, w}$. Let $\mu$ be a dominant cocharacter for the partition $(d_i)_{i=1}^k$ and let $b:=\frac{n_{\mu}}{2}=\langle \mu, \mathfrak{g}^{\mu>0} \rangle$, see \eqref{def:nlambda} for the definition of $n_\lambda$.
\begin{lemma}
There is a functor 
\begin{align}\label{def:coprod0}
\widetilde{\Delta}_B:=\left(q_\mu^*\right)^{-1}\beta_b\, p_\mu^*: \mathbb{M}(d)_w\to \mathbb{M}_B:=\boxtimes_{i=1}^k \mathbb{M}(d_i)_{w_i}.
\end{align}
\end{lemma}
\begin{proof}
First, note that the image of $p_\mu^*\left(\mathbb{M}(d)_w\right)$ is in $D^b(\X(d)^{\mu\geq 0})_{\leq b}$ from the description of the category $\mathbb{M}(d)_w$ in Lemma \ref{lemma:HLS}. Thus the image of $\widetilde{\Delta}_B$ is in $D^b(\X(d)^\mu)_b$.
Also, note that the category $\mathbb{M}_B$ is a subcategory of $D^b(\X(d)^\mu)_b$ because $B=(d_i, w_i)_{i=1}^k$ is in $H_{d,w}$.

Let $\chi$ be a dominant weight of $T(d)$ such that $\chi+\rho\in \textbf{W}(d)_{w}$. 
If $\chi+\rho$ is not on the face $F(\mu)$ of the polytope $\textbf{W}(d)$, then \[\beta_b p_\mu^*\left(\mathcal{O}_{\X(d)}\otimes \Gamma_{GL(d)}(\chi)\right)=0.\] If $\chi+\rho$ is on the face $F(\mu)$, then by \cite[Corollary 3.4]{P} we can write
\begin{equation}\label{dechalf}
    \chi+\rho-w\tau_d=\frac{1}{2}N^{\mu>0}+\sum_{i=1}^k\left(\psi_i+\rho_i\right),
    \end{equation}
    where $\psi_i \in M(d_i)$ and 
    $\psi_i+\rho_i\in \textbf{W}(d_i)_0$.
    In particular, we have that \[\chi=\sum_{i=1}^k \psi_i +\mathfrak{g}^{\mu>0}+w\tau_d.\]
    Write $\chi=\sum_{i=1}^k \chi_i$. We have 
    \begin{align*}
       \langle 1_{d_i}, \chi_i \rangle=
        \frac{d_i w}{d}+\langle 1_{d_i}, \mathfrak{g}^{\mu>0}\rangle=v_i+\langle 1_{d_i}, \mathfrak{g}^{\mu>0}\rangle=w_i. \end{align*}  Further, we have that $\chi_i=\psi_i+w_i\tau_{d_i}$, so $\chi_i+\rho_i\in \textbf{W}(d_i)_{w_i}$.
    Therefore we have
\[\left(q_\mu^*\right)^{-1}\beta_b\, p_\mu^*\left(\mathcal{O}_{\X(d)}\otimes \Gamma_{GL(d)}(\chi)\right)=\boxtimes_{i=1}^k\left(\mathcal{O}_{\X(d_i)}\otimes \Gamma_{GL(d_i)}(\chi_i)\right)\in \mathbb{M}_B.\]
\end{proof}
Let $\chi$ be a dominant weight with $\chi+\rho\in\textbf{W}(d)$. Note that
\begin{equation}\label{computationdeltab}
    \widetilde{\Delta}_B\left(\mathcal{O}_{\X(d)}\otimes \Gamma_{GL(d)}(\chi)\right)=
    \begin{cases} \mathcal{O}_{\X(d)^\mu}\otimes \Gamma_{GL(d)^\mu}(\chi),\text{ if }\chi\in F(\mu),\\
    0,\text{ otherwise}.
    \end{cases}
\end{equation}

The functor (\ref{def:coprod0}) induces a functor 
\begin{equation}\label{def:coproduct}
\widetilde{\Delta}_B:=\left(q_\mu^*\right)^{-1}\beta_b\, p_\mu^*: \mathbb{S}^\bullet_\ast(d)_w\to \mathbb{S}^\bullet_{\ast, B}.
\end{equation}

Let $A=(d_i, w_i)_{i=1}^k$ and $C=(f_i, u_i)_{i=1}^{l}$ be partitions in $H_{d,w}$ such that $C$ is a refinement of $A$, see Subsection \ref{compa}.
One can analogously define functors
\begin{align*}
    \widetilde{\Delta}_{AC}\colon \mathbb{M}_A\to \mathbb{M}_C, \ 
    \widetilde{\Delta}_{AC}\colon \mathbb{S}^\bullet_{\ast, A}\to \mathbb{S}^\bullet_{\ast, C}
\end{align*} for categories $\mathbb{S}^{\bullet}_{\ast}$ as in Subsection \ref{gradingMF}.

\subsection{Compatibility between the product and the coproduct}\label{s42}

In this section, we show that $\widetilde{m}$ and $\widetilde{\Delta}$ are compatible. Recall the forget-the-potential map \eqref{definition:forgetpotential}:
\[\Theta\colon K_T\left(\mathrm{MF}\left(\X(d), \Tr W\right)\right)\to K_T\left(\mathrm{MF}(\X(d), 0)\right).\]
Recall that $K_T(\mathbb{S}_A)_\mathbb{F}\hookrightarrow K_T(\mathbb{M}_A)_\mathbb{F}$, see \cite[Theorem 4.12 and Equation (4.36)]{PT0}. 
Let 
\[K_T(\mathbb{S}_A)':=
K_T\left(\mathbb{S}_A\right)/
(\mathbb{K}\text{-torsion})
\stackrel{\cong}{\to}
\text{image}\left(\Theta\colon K_T(\mathbb{S}_A)\to K_T(\mathbb{M}_A)
\right).\]

\begin{thm}\label{prodcoprod}
Consider a pair $(d, w)\in \mathbb{N}\times\mathbb{Z}$.
Let $\lambda$ and $\mu$ be dominant cocharacters with associated partitions $A=(d_i, w_i)_{i=1}^k$ and $B=(e_i, v_i)_{i=1}^s$ in $H_{d, w}$.
Let $S^B_A\subset S^\mu_\lambda$ be the set of partitions $C=(f_i, u_i)_{i=1}^l$ with $l=ks$ with $\overline{C}$ in $H_{d,w}$ and such that 
\begin{align*}
    f_{(i-1)s+1}+f_{(i-1)s+2}+\ldots+f_{is}&=d_i\text{ for }1\leq i\leq k,\\
    f_j+f_{s+j}+\ldots+f_{(k-1)s+j}&=e_j\text{ for }1\leq j\leq s.
\end{align*}
Then the following diagram commutes:
\begin{equation*}
    \begin{tikzcd}
    K_T(\mathbb{S}_A)'\arrow[r, "\widetilde{m}_A"]\arrow[d, "\bigoplus \widetilde{\Delta}_{AC}"]& K_T(\mathbb{S}(d)_w)'\arrow[d, "\widetilde{\Delta}_B"]\\
    \bigoplus_{C\in S^B_A} K_T(\mathbb{S}_C)'\arrow[r, "\bigoplus \widehat{m}_{CB}"]& K_T(\mathbb{S}_B)'.
    \end{tikzcd}
\end{equation*} 
\end{thm}

Theorem \ref{prodcoprod} is a $T$-equivariant version of \cite[Theorem 5.2]{P} and the same proof works to show this statement. 
However, we present an alternative proof which first proves a categorical statement about complexes in $D^b(\X(d)^\mu)$ which is stronger than the results in loc. cit. and is of independent interest for computations in categorical Hall algebras. 

Note that the compatibility of the product and coproduct for localized K-theory, either in the above setting or in the setting of Corollary \ref{cor:44}, follows by a direct computation and Propositions \ref{commutative} and \ref{1236bis}.  

\begin{prop}\label{prop42}
Let $A, B$ be as in Theorem \ref{prodcoprod}.
For $1\leq i\leq k$,
let $\chi_i$ be a dominant weight of $T(d_i)$ for $1\leq i\leq k$ such that $\chi_i+\rho_i\in \mathbf{W}(d_i)$. 
Let $\chi:=\sum_{i=1}^k \chi_i$.
For any $C\in S^B_A$, there are natural maps:
\begin{equation}\label{naturalmap}
\widehat{m}_{CB}
\widetilde{\Delta}_{AC}\left(\mathcal{O}_{\X(d)^\lambda}\otimes \Gamma_{GL(d)^\lambda}(\chi)\right)\to \widetilde{\Delta}_B \widetilde{m}_A\left(\mathcal{O}_{\X(d)^\lambda}\otimes \Gamma_{GL(d)^\lambda}(\chi)\right)
\end{equation}
such that there is an isomorphism:
\begin{equation}\label{isosba}
\bigoplus_{C\in S^B_A}
\widehat{m}_{CB}\widetilde{\Delta}_{AC}\left( \mathcal{O}_{\X(d)^\lambda}\otimes \Gamma_{GL(d)^\lambda}(\chi)\right)\xrightarrow{\sim} 
\widetilde{\Delta}_B \widetilde{m}_A\left(\mathcal{O}_{\X(d)^\lambda}\otimes \Gamma_{GL(d)^\lambda}(\chi)\right).
\end{equation}
\end{prop}

\begin{proof}[Proof of Theorem \ref{prodcoprod}]
By Proposition \ref{prop42}, the following diagram commutes:
\begin{equation*}
    \begin{tikzcd}
    K_T(\mathbb{M}_A)\arrow[r, "\widetilde{m}_A"]\arrow[d, "\bigoplus \widetilde{\Delta}_{AC}"]& K_T(\mathbb{M}(d)_w)\arrow[d, "\widetilde{\Delta}_B"]\\
    \bigoplus_{C\in S^B_A} K_T(\mathbb{M}_C)\arrow[r, "\bigoplus \widehat{m}_{CB}"]& K_T(\mathbb{M}_B).
    \end{tikzcd}
    \end{equation*}
    The maps $\widetilde{m}$ and $\widetilde{\Delta}$ are compatible with the forget-the-potential map \eqref{definition:forgetpotential} by \cite[Proposition 3.6]{P0} and \cite[Proposition 5.1]{P}, respectively. The conclusion thus follows.
\end{proof}

\begin{proof}[Proof of Proposition \ref{prop42}]
The argument follows closely the proof of \cite[Theorem 5.2]{P3}. We give an overview of the proof. We use a Koszul resolution to compute $\widetilde{m}_A\left(\mathcal{O}_{\X(d)^\lambda}\otimes \Gamma_{GL(d)^\lambda}(\chi)\right)$ in terms of $\mathcal{O}\otimes\Gamma_{GL(d)}(\theta)$, where $\theta=(\chi-\sigma_I)^+$ for $\sigma_I$ a partial sum of weights pairing positively with $\lambda$. 
Let $\mathcal{O}\otimes\Gamma_{GL(d)}(\theta)$
be a vector bundle appearing in the Koszul resolution with non-zero $\widetilde{\Delta}_B$. Then the weight $\theta$ is on a face of $\textbf{W}(d)$, see \eqref{computationdeltab}. 
We use Proposition \ref{propboundary} to characterize the highest weights $\theta$ on a face of $\textbf{W}(d)$ in terms of partitions $C\in S^B_A$.  
The proof then follows from a direct comparison with the right hand side of \eqref{isosba}. Let $\sigma_I$ be a sum such that $\chi-\sigma_I$ is on a face of $\textbf{W}(d)$ corresponding to $C\in S^B_A$ with associated permutation $w$, see \eqref{permutations}.
The swap morphism appears because conjugating $\chi-\sigma_I$ to $\theta=(\chi-\sigma_I)^+$ first requires to act by $w$.

The multiplication $\widetilde{m}_A$ is defined as $\widetilde{m}_A=p_{\lambda^{-1}*}q_{\lambda^{-1}}^*$.
Let $\chi:=\sum_{i=1}^k \chi_i$.
Consider the Koszul resolution:
\begin{align*}
\widetilde{m}_A\left(\mathcal{O}_{\X(d)^\lambda}\otimes\Gamma_{GL(d)^\lambda}(\chi)\right)&=
\pi_{\lambda^{-1}*}\iota_{\lambda^{-1}*}q_{\lambda^{-1}}^*\left(\mathcal{O}_{\X(d)^\lambda}\otimes\Gamma_{GL(d)^\lambda}(\chi)\right)\\
&\cong 
\pi_{\lambda^{-1}*}\left(\bigoplus_{I\subset A_\lambda} \mathcal{O}_{R(d)}\otimes\Gamma_{GL(d)^{\lambda\leq 0}}(\chi-\sigma_I)[|I|], d\right),
\end{align*}
see Proposition \ref{bbw} for the notation. The differential $d$ is induced by multiplication with generators $(e_i)_{i=1}^t$ of the polynomial ring $\mathbb{C}\left[R(d)^{\lambda<0}\right]\cong \mathbb{C}[e_1,\ldots, e_t]$. 
Fix $C\in S^B_A$, consider the associated dominant cocharacters $\nu$ and $\kappa$, and let $w=w_C\in \mathfrak{S}_d$ as in \eqref{permutations}. 
Let $\varphi:=w^{-1}\mu$.
There are natural maps of complexes 
\begin{multline}\label{C}
\pi_{\lambda^{-1} *}\left(
\bigoplus_{J\subset A^{\varphi}_\lambda}
\mathcal{O}_{R(d)}\otimes\Gamma_{GL(d)^{\lambda\leq 0}}(\chi)
\otimes\mathcal{O}\left(-N^{\varphi}_\lambda-\sigma_J\right)
[\lvert I^{\varphi}_\lambda \rvert+\lvert J \rvert], d
\right)
\to\\ \pi_{\lambda^{-1} *}\left(\bigoplus_{I\subset A_\lambda} \mathcal{O}_{R(d)}\otimes\Gamma_{GL(d)^{\lambda\leq 0}}(\chi)\otimes \mathcal{O}(-\sigma_I)[|I|], d\right)
\end{multline}
induced by the inclusion of sets \[I_C:=\{N^{\varphi}_\lambda+\sigma_J \mid J\subset A^{\varphi}_\lambda\}\subset \{\sigma_I \mid I\subset A_\lambda\}.\] 
The differential $d$ of the complex on the first line is induced by multiplication with generators of the polynomial ring $\mathbb{C}\left[\left(R(d)^{\varphi}\right)^{\lambda<0}\right]$.
For $C, C'$ different elements in $S^B_A$, we have that $I_C\cap I_{C'}=\emptyset$. 
If $\sigma\in \{\sigma_I \mid I\subset A_\lambda\}\setminus \left(\bigcup_{C\in S^B_A}I_C\right)$, then $\mathcal{O}_{R(d)}\otimes\Gamma_{GL(d)^{\lambda\leq 0}}(\chi)\otimes\mathcal{O}(-\sigma_I)$ has $u\mu$-weights strictly less than $\frac{n_{\mu}}{2}=\langle \mu, \mathfrak{g}^{\mu>0}\rangle$ for all $u\in \mathfrak{S}_d$,
see Proposition \ref{propboundary}.
It follows that 
\begin{equation}\label{reszero}
    \widetilde{\Delta}_B \pi_{\lambda^{-1}*}\left(\mathcal{O}_{R(d)}\otimes\Gamma_{GL(d)^{\lambda\leq 0}}(\chi)\otimes\mathcal{O}(-\sigma_I)\right)=0
    \end{equation}
    for such sums $\sigma$, see \eqref{computationdeltab}. 
    It suffices to show that, for $C\in S^B_A$, we have natural isomorphisms of complexes:
\begin{multline}
    \widetilde{\Delta}_B\pi_{\lambda^{-1}*}\left(
\bigoplus_{J\subset A^{\varphi}_\lambda}
\mathcal{O}_{R(d)}\otimes\Gamma_{GL(d)^{\lambda\leq 0}}(\chi)
\otimes\mathcal{O}\left(-N^{\varphi}_\lambda-\sigma_J\right)
[|I^{\varphi}_\lambda|+|J|], d
\right)\\
\cong \widehat{m}_{CB}\widetilde{\Delta}_{AC}\left(\mathcal{O}_{\X(d)^\lambda}\otimes\Gamma_{GL(d)^{\lambda}}(\chi)\right).
\end{multline}
Using \eqref{C}, we obtain the natural maps \eqref{naturalmap}, and further we obtain the isomorphism \eqref{isosba} using the vanishing \eqref{reszero}.

For $J\subset A^{\varphi}_\lambda$, we have that \[w*\left(\chi-N^{\varphi}_\lambda-\sigma_J\right)+\rho\in F(\mu)\subset\textbf{W}(d),\] see Proposition \ref{propboundary}.
Let $\widetilde{w}_J\in \mathfrak{S}_{d}$ be the element of minimal length such that $\widetilde{w}_J*\left(\chi-N^{\varphi}_\lambda-\sigma_J\right)$ is dominant or zero. 
Observe that $\widetilde{w}_\emptyset=w$. However, for general $J$, we have that $\widetilde{w}_J=u_J\circ w$ for a permutation $u_J$ in $\times_{i=1}^s\mathfrak{S}_{e_i}\cong W^\mu$ and $\ell(\widetilde{w}_J)=\ell(w)+\ell(u_J)$.
By the Borel-Bott-Weyl theorem, there is a natural isomorphism
\begin{align*}
&\pi_{\lambda^{-1}*}\left(
\mathcal{O}_{R(d)}\otimes\Gamma_{GL(d)^{\lambda\leq 0}}\left(\chi-N^{\varphi}_\lambda-\sigma_J\right)
\right)\cong\\ &\pi_{\lambda^{-1}*}\left(
\mathcal{O}_{R(d)}\otimes\Gamma_{GL(d)^{\lambda\leq 0}}\left(w*(\chi-N^{\varphi}_\lambda-\sigma_J)\right)[-|J^{\varphi}_\lambda|]
\right)
\end{align*}
and further there are natural isomorphisms
\begin{align}\label{84}
   & \widetilde{\Delta}_B\pi_{\lambda^{-1}*}\left(
\bigoplus_{J\subset A^{\varphi}_\lambda}
\mathcal{O}_{R(d)}\otimes\Gamma_{GL(d)^{\lambda\leq 0}}\left(\chi-N^{\varphi}_\lambda-\sigma_J\right)
[|I^{\varphi}_\lambda|+|J|], d
\right) \\
\notag&\cong \widetilde{\Delta}_B\pi_{\lambda^{-1}*}\left(
\bigoplus_{J\subset A^{\varphi}_\lambda}
\mathcal{O}_{R(d)}\otimes\Gamma_{GL(d)^{\lambda\leq 0}}\left(w*(\chi-N^{\varphi}_\lambda-\sigma_J)\right)
[|I^{\varphi}_\lambda|-|J^{\varphi}_\lambda|+|J|], d
\right) \\
\notag&\cong
\pi_{\kappa^{-1}\mu*}\left(\bigoplus_{J\subset A^{\varphi}_\lambda} 
\mathcal{O}_{R(d)^\mu}\otimes\Gamma_{(GL(d)^\mu)^{\kappa\leq 0}}\left(w*(\chi-N^{\varphi}_\lambda-\sigma_J)\right)
[|I^{\varphi}_\lambda|-|J^{\varphi}_\lambda|+|J|], d
\right).
\end{align}
We have that 
\begin{equation}\label{kappanu}
    \X(d)^\nu=\X(d)^\kappa\text{ and }G(d)^\nu=G(d)^\kappa.
\end{equation}
For a weight $\theta$ of $T(d)$, denote by
   \begin{equation}\label{defPsi}
   \Psi\left(\mathcal{O}_{\X(d)^\kappa}\otimes \Gamma_{G(d)^\kappa}(\theta)\right):=\mathcal{O}_{\X(d)^\kappa}\otimes \Gamma_{G(d)^\kappa}(w*\theta).
    \end{equation}
For a subset $J\subset A^{\varphi}_\lambda$, let $J'=\{w\beta \mid \beta\in J\}$ be the corresponding subset of $I_{\kappa\mu}:=\{\beta\text{ weight of }R(d)^\mu \mid 
\langle \kappa, \beta\rangle>0\}$, see Proposition \ref{prop:help}.
There is an isomorphism 
\begin{align}\label{isomultline}
\mathcal{O}_{(\X(d)^\mu)^{\kappa\leq 0}}\otimes
\Gamma_{(GL(d)^\mu)^{\kappa\leq 0}}
\left(w*(\chi-N^{\varphi}_\lambda-\sigma_J)\right)
[-|J^{\varphi}_\lambda|]
\cong\\ \notag
q_{\kappa^{-1}\mu}^*
\left(\Psi
\left(\mathcal{O}_{\X(d)^\nu}\otimes\Gamma_{GL(d)^\nu}
\left(\chi-N^{\varphi}_\lambda\right)
[-|J^{\varphi}_\lambda|]\right)\otimes\mathcal{O}(-\sigma_{J'})\right).\end{align}
For $\mathcal{B}$ a complex in $D^b(\X(d)^\nu)$, by the definition of $\widetilde{\text{sw}}_C$ we have that  
\begin{equation}\label{884}
\widetilde{\text{sw}}_{C}(\mathcal{B}):=
\Psi\left(\mathcal{B}\otimes \mathcal{O}\left(-N^{\varphi}_\lambda\right)\right)[|I^{\varphi}_\lambda|-|J^{\varphi}_\lambda|].
\end{equation}
We next want to use Proposition \ref{bbw} for the map $\iota_{\kappa^{-1}\mu}$. For this, it is convenient to use the following notation
\[\mathrm{F}\left(\mathcal{O}_{(R(d)^\mu)^{\kappa\leq 0}}\otimes \Gamma_{(G(d)^\mu)^{\kappa\leq 0}}(\theta)\right):= \mathcal{O}_{R(d)^\mu}\otimes \Gamma_{(G(d)^\mu)^{\kappa\leq 0}}(\theta)
\]
for $\theta$ a dominant weight of $T(d)$. 
We consider the Koszul resolution:
\begin{multline}\label{88}
\iota_{\kappa^{-1}\mu*}q_{\kappa^{-1}\mu}^*\left(\widetilde{\text{sw}}_{C} \left(\mathcal{O}_{\X(d)^\kappa}\otimes\Gamma_{GL(d)^\kappa}(\chi)\right)\right)\cong \\\left(\bigoplus_{J'\subset I_{\kappa\mu}}
\mathrm{F} q_{\kappa^{-1}\mu}^*\left(\widetilde{\text{sw}}_{C}\left(\mathcal{O}_{\X(d)^\kappa}\otimes\Gamma_{GL(d)^\kappa}(\chi)\right)\right)\otimes \mathcal{O}(-\sigma_{J'})[|J'|], d\right),\end{multline}
where the differential $d$ is induced by multiplication with generators of the polynomial ring $\mathbb{C}\left[\left(R(d)^{\mu}\right)^{\kappa<0}\right]$.
Rewrite \eqref{84} using \eqref{defPsi}, \eqref{isomultline}, \eqref{884}:
\begin{align}\label{84new}
    & \widetilde{\Delta}_B\pi_{\lambda^{-1}*}\left(
\bigoplus_{J\subset A^{\varphi}_\lambda}
\mathcal{O}_{R(d)}\otimes\Gamma_{GL(d)^{\lambda\leq 0}}\left(\chi-N^{\varphi}_\lambda-\sigma_J\right)
[|I^{\varphi}_\lambda|+|J|], d
\right) \\
\notag&\cong \pi_{\kappa^{-1}\mu*}
\left(\bigoplus_{J'\subset I_{\kappa\mu}}\mathrm{F} q_{\kappa^{-1}\mu}^*\left(\widetilde{\text{sw}}_{C}\left(\mathcal{O}_{\X(d)^\kappa}\otimes \Gamma_{GL(d)^{\kappa}}(\chi))
\right)\otimes \mathcal{O}(-\sigma_{J'})[|J'|], d\right)\right).
\end{align}
There are isomorphisms
\begin{align*}
    &\widetilde{\Delta}_B \pi_{\lambda^{-1}*}
    \left(
\bigoplus_{J\subset A^{\varphi}_\lambda}
\mathcal{O}_{R(d)}\otimes \Gamma_{GL(d)^{\lambda\leq 0}}(\chi)
\otimes
\mathcal{O}\left(-N^{\varphi}_\lambda-\sigma_J\right)
[|I^{\varphi}_\lambda|+|J|], d
\right) \\
&\overset{(1)}{\cong}
\pi_{\kappa^{-1}\mu*}
\left(\bigoplus_{J'\subset I_{\kappa\mu}}\mathrm{F} q_{\kappa^{-1}\mu}^*\left(\widetilde{\text{sw}}_{C}\left(\mathcal{O}_{\X(d)^\kappa}\otimes \Gamma_{GL(d)^{\kappa}}(\chi)
\right)\right)\otimes \mathcal{O}(-\sigma_{J'})[|J'|], d\right)\\
&\overset{(2)}{\cong}\pi_{\kappa^{-1}\mu*}
\iota_{\kappa^{-1}\mu*}
q_{\kappa^{-1}\mu}^*
\left(\widetilde{\text{sw}}_{C} \left(\mathcal{O}_{\X(d)^\kappa}\otimes \Gamma_{GL(d)^{\kappa}}(\chi)\right)
\right)
\\ 
&\overset{(3)}{\cong}\widetilde{m}_{CB} 
\left(\widetilde{\text{sw}}_{C} \left(\mathcal{O}_{\X(d)^\kappa}\otimes \Gamma_{GL(d)^{\kappa}}(\chi)\right)
\right)\\ 
&\overset{(4)}{\cong}\widetilde{m}_{CB} 
\left(\widetilde{\text{sw}}_{C} \left(\widetilde{\Delta}_{AC}(\mathcal{O}_{\X(d)^\lambda}\otimes \Gamma_{GL(d)^{\lambda}}(\chi))\right)
\right).
\end{align*}
Recall that $\widetilde{m}_{\kappa\mu}=\widetilde{m}_{CB}$ and $\widetilde{\Delta}_{\nu\lambda}=\widetilde{\Delta}_{AC}$.
The isomorphism (1) is the isomorphism \eqref{84new} and it respects the differentials. 
The isomorphism (2) follows from \eqref{88}. The isomorphism (3) follows from the definition of $\widetilde{m}_{CB}$. The isomorphism (4) follows from \eqref{kappanu} and the equality \[\widetilde{\Delta}_{AC}\left(\mathcal{O}_{\X(d)^\lambda}\otimes \Gamma_{GL(d)^{\lambda}}(\chi)\right)=\mathcal{O}_{\X(d)^\nu}\otimes \Gamma_{GL(d)^{\nu}}(\chi)=\mathcal{O}_{\X(d)^\kappa}\otimes \Gamma_{GL(d)^{\kappa}}(\chi).\]
\end{proof}

\begin{prop}\label{propboundary}
Let $\lambda$ and $\mu$ be dominant cocharacters of $GL(d)$ and let $w\in W$. Assume that $\chi$ is a dominant weight with $\chi+\rho\in F(\lambda)$ and that $I\subset A_{\lambda}$ such that $w*(\chi-\sigma_I)+\rho\in F(\mu)$. Let $\varphi:=w^{-1}\mu$. Then 
\begin{equation}\label{IJ}
I=\{\beta\in A_\lambda \mid \langle \varphi, \beta\rangle<0\}\sqcup J
\end{equation}
for a subset $J\subset A^{\varphi}_\lambda$.

Conversely, for all $I\subset A_\lambda$ as in \eqref{IJ}, we have $w*(\chi-\sigma_I)+\rho\in F(\mu)$.
\end{prop}

\begin{proof}
For two cocharacters $\tau$ and $\tau'$, 
we use the notations \[N^{\tau>0}:=\sum_{A_\tau}\beta,\ N^{\tau'=0, \tau>0}=\sum_{A^{\tau'}_\tau}\beta.\]
Let $v=\langle 1_d, \chi\rangle$. Write 
\begin{align*}
    \chi-v\tau_d+\rho&=\frac{1}{2}N^{\lambda>0}+\psi,\\
    w*(\chi-v\tau_d-\sigma_I)+\rho&=\frac{1}{2}N^{\mu>0}+\phi',\\
    \chi-v\tau_d-\sigma_I+\rho&=\frac{1}{2}N^{\varphi>0}+\phi,
\end{align*}
where $\psi\in \textbf{W}(\lambda)_0$ and $\phi\in \textbf{W}(\varphi)_0$, see \cite[Proposition 3.4]{P2}.
Then 
\begin{equation}\label{lambdavarphi}
\sigma_I=N_\lambda^\varphi+\left(\frac{1}{2}N^{\varphi=0, \lambda>0}-\phi\right)+\left(\psi-\frac{1}{2}N^{\lambda=0, \varphi>0}\right).
\end{equation}
The weight $\widetilde{\phi}:=\frac{1}{2}N^{\varphi=0, \lambda>0}-\phi$ is a sum with nonnegative coefficients of weights $\beta$ such that $\langle \varphi, \beta\rangle=0$ and $\langle \lambda, \beta\rangle>0$
and weights $\beta'$ such that $\langle \varphi, \beta'\rangle=\langle\lambda, \beta'\rangle=0$. The weight $\widetilde{\psi}:=\psi-\frac{1}{2}N^{\lambda=0, \varphi>0}$ is a sum with nonnegative coefficients of weights $\beta$ such that $\langle \lambda, \beta\rangle=0$ and $\langle \varphi, \beta\rangle<0$ and weights $\beta'$ such that $\langle \varphi, \beta'\rangle=\langle\lambda, \beta'\rangle=0$. We denote by $\sigma^{\lambda+}_{\varphi0}$ a sum with nonnegative coefficients of weights $\beta$ such that $\langle\lambda, \beta\rangle>0$ and $\langle \varphi, \beta\rangle=0$ etc. 
Then we can write \begin{equation}\label{sigmaN}
    \sigma_I-N^\varphi_\lambda=\sigma^{\lambda+}_{\varphi0}+\sigma^{\lambda+}_{\varphi+}-n^{\lambda+}_{\varphi-},
\end{equation} where all the sums on the right hand side are further partial sums of weights in $A_\lambda$. We can rewrite \eqref{lambdavarphi} as
\begin{equation}\label{lambdavarphi2}
\sigma^{\lambda+}_{\varphi0}+\sigma^{\lambda+}_{\varphi+}=n^{\lambda+}_{\varphi-}+\widetilde{\phi}^{\lambda+}_{\varphi0}+\widetilde{\phi}^{\lambda0}_{\varphi0}+\widetilde{\psi}^{\lambda0}_{\varphi-}+\widetilde{\psi}^{\lambda0}_{\varphi0}.
\end{equation}
The $\varphi$-weight of the left hand is nonnegative, while the $\varphi$-weight of the right hand side is nonpositive. Thus $\sigma^{\lambda+}_{\varphi+}=n^{\lambda+}_{\varphi-}=\widetilde{\psi}^{\lambda0}_{\varphi-}=0$. In particular, \eqref{sigmaN} becomes
\[\sigma_I=N^\varphi_\lambda+\sigma^{\lambda+}_{\varphi0},\]
which implies the first direction. The converse follows in a similar way.
\end{proof}

\begin{prop}\label{prop:help}
For a subset $J\subset A^{\varphi}_\lambda$, the set $J'=\{w\beta \mid \beta\in J\}$ is a subset of $I_{\kappa\mu}:=\{\beta\text{ weight of }R(d)^\mu\mid\langle \kappa, \beta\rangle>0\}$.
This transformation induces a bijection of sets $A^{\varphi}_\lambda\xrightarrow{\sim} I_{\kappa\mu}$.
\end{prop}

\begin{proof}
It suffices to check the first claim. To construct an inverse of this transformation, send $L\subset I_{\kappa\mu}$ to $L^\circ=\{w^{-1}\beta\mid\beta\in L\}$.

Let $\beta\in A^\varphi_\lambda$. Recall that $\varphi=w^{-1}\mu$, so we have that $\langle w^{-1}\mu, \beta\rangle=0$, and thus $\langle \mu, w\beta\rangle=0$.

It suffices to show that if a weight $\beta_i-\beta_j$ is in $A^{\varphi}_\lambda$, then $\langle \kappa, \beta_{w(i)}-\beta_{w(j)}\rangle>0$. 
For simplicity, we discuss the case when $k=s=2$. Rename $f_{11}=f_1$, $f_{12}=f_2$, $f_{21}=f_3$, $f_{22}=f_4$. 
The permutation $w$ is
\begin{align*}
    w(i)=\begin{cases}
    i+f_3,\text{ if }f_1+1\leq i\leq f_1+f_2,\\
    i-f_2,\text{ if }f_1+f_2+1\leq i\leq f_1+f_2+f_3,\\
    i,\text{ otherwise}.
    \end{cases}
\end{align*}
We have that
\[\langle \lambda, \beta_i-\beta_j\rangle>0\text{ and }\langle \mu, \beta_{w(i)}-\beta_{w(j)}\rangle=0,\]
so one of two possibilities happens: 
\begin{align*}
    (1)\, & f_1+f_2+f_3+1\leq i\leq f_1+f_2+f_3+f_4\text{ and }f_1+1\leq j\leq f_1+f_2,\\
    (2)\, & f_1+f_2+1\leq i\leq f_1+f_2+f_3\text{ and }1\leq j\leq f_1.
\end{align*}
In case (1), we have that $i=w(i)$ and $f_1+f_3+1\leq w(j)\leq f_1+f_3+f_2$, and then $\langle \kappa, \beta_{w(i)}-\beta_{w(j)}\rangle>0$. In case (2), we have that $f_1+1\leq w(i)\leq f_1+f_3$ and $j=w(j)$, and then $\langle \kappa, \beta_{w(i)}-\beta_{w(j)}\rangle>0$.
\end{proof}



\subsection{The bialgebra structure under the Koszul equivalence}\label{s43}
Let $(d, v)\in \mathbb{N}\times\mathbb{Z}$ be coprime integers. 
For $n\in\mathbb{N}$, denote by $R_n$ the set of ordered partitions $A=(n_i)_{i=1}^k$ of $n$ with $n_i\geq 1$. 
For each such partition $A$ of $n$, denote also by $A$ the partition $(n_id, n_iv)_{i=1}^k$ of $(nd, nv)$. Let $\gamma_i$ be the weights of $T$ for $i\in\{1, 2\}$ with $q^{\gamma_i}=q_i\in \mathbb{K}$. Let $\gamma=\gamma_1+\gamma_2$ and let $q^\gamma=q_1q_2$. For simplicity, we will denote $q^\gamma$ just by $q$. 


\subsubsection{}\label{coproductdiml}
The coproduct \eqref{def:coproduct} induces coproduct maps 
\[\Delta_{AC} \colon \mathbb{T}_A\to \mathbb{T}_C\] for $C$ a refinement of $A$ as follows.
Recall the Koszul equivalence 
\begin{equation}\label{koszuleqev}
\Phi: \mathbb{T}(e)_v\cong \mathbb{S}^{\text{gr}}(e)_v
\end{equation}
for all pairs $(e, v)\in \mathbb{N}\times\mathbb{Z}$. 
Let $A=(n_id, w_i)_{i=1}^k$ be a partition of $(nd, nv)$ with associated prime partition $A'=(n_id, n_iv)_{i=1}^k$, see Subsection \ref{prime}. Let $\mathbb{T}_A:=\otimes_{i=1}^k \mathbb{T}(n_id)_{n_iv}$.
There is a Koszul equivalence 
\[\Phi_A\colon \mathbb{T}_A
\xrightarrow{\sim}\mathbb{S}^{\text{gr}}_A.\]
Let $\Phi_A^{-1}$ be its inverse. 
For $V$ a vector space of dimension $nd$, let $\mathfrak{l}:=\mathrm{End}(V)^{\lambda_A}$. 
Let $\lambda_C$ be the antidominant cocharacter of $T(nd)\subset GL(nd)^{\lambda_A}$ corresponding to $C$,
let $\omega_{AC}:=\det\left( (\mathfrak{l}^{\lambda_C>0})^{\vee}\right)[-\dim \mathfrak{l}^{\lambda_C>0}]
$, where $T$ acts on $\mathfrak{l}$ with weight $\gamma:=\gamma_1+\gamma_2$. 
Note that $\det\left( (\mathfrak{l}^{\lambda_C>0})^{\vee}\right)$ is a character of 
$GL(d)^{\lambda_C}$, hence it determines a line bundle on $\mathcal{X}(nd)^{\lambda_C}$. 
We define the functor 
\begin{align*}
    \Delta_{AC}\colon
    \mathbb{T}_{A}\to \mathbb{T}_C
    \end{align*}
by the commutative diagram 
\begin{align*}
    \xymatrix{
    \mathbb{T}_A \ar[rr]^-{\Delta_{AC}} \ar[d]^-{\sim}_-{\Phi_A} & & \mathbb{T}_C \ar[d]_-{\sim}^-{\Phi_C} \\
    \mathbb{S}_A^{\rm{gr}} \ar[rr]_{\widetilde{\Delta}_{AC}(-) \otimes \omega_{AC}^{-1}} & & \mathbb{S}_C^{\rm{gr}}. 
    }
\end{align*}

When $A=(nd, nv)$ and $C$ is a two term partition $(n_id, n_iv)_{i=1}^2$, the term $\omega_{n_1,n_2}:=\omega_{AC}$ is \eqref{def:factordimred} for $\dim V_i=n_id$ for $i\in\{1, 2\}$. For these partitions, we let $\Delta_{n_1 n_2}:=\Delta_{AC}$.
For such partitions $A$ and $C$, consider the $(T\times T(d))$-weight 
\begin{align*}\label{def:nusum}
   \nu_{n_1, n_2}:=\sum_{i>n_1d\geq j}(\beta_i-\beta_j-\gamma).
\end{align*}
Then $\omega_{n_1, n_2}=(-1)^{(n_1d)\cdot (n_2d)}q^{\nu_{n_1, n_2}}$. Alternatively, $\omega_{n_1, n_2}$ measures a ratio (called renormalized twist in \cite[Proof of Lemma 2.3.7]{VaVa}) constructed from the shuffle product for $\widetilde{m}$ with kernel 
\[	\xi'(x):=\frac{(1-q_1^{-1}x)(1-q_2^{-1}x)(1-qx)}{1-x},\] and the shuffle product for $m$ with kernel $\xi(x)$, see the computation in \cite[Proof of Lemma 4.9]{PT0}.

\subsubsection{}
Let $(d, v)\in \mathbb{N}\times\mathbb{Z}$ be coprime integers and let $n\in\mathbb{N}$.
Consider partitions $A=(d_i)_{i=1}^k$ and $B=(e_i)_{i=1}^s$ of $n$. 
Let $S^B_A$ be the set of partitions $C=(f_{i})_{i=1}^{l}$ of $n$ with $l=ks$ such that 
\begin{align*}
    f_{(i-1)s+1}+f_{(i-1)s+2}+\ldots+f_{is}&=d_i\text{ for }1\leq i\leq k,\\
    f_j+f_{s+j}+\ldots+f_{(k-1)s+j}&=e_j\text{ for }1\leq j\leq s.
\end{align*}
Let $D$ be the partition on $n$ constructed as in \eqref{partitionD}. Define $m'_{BC}:=m_{BC}\circ\text{sw}_{C}$, for $\text{sw}_{C}$ as in \eqref{sw}.
Theorem \ref{prodcoprod} implies the following: 
 
 \begin{cor}\label{cor:44}
 In the above setting, the following diagram commutes:
\begin{equation*}
    \begin{tikzcd}
    K_T(\mathbb{T}_A)'\arrow[r, "\widetilde{m}_A"]\arrow[d, "\bigoplus \widetilde{\Delta}_{AC}"]& K_T(\mathbb{T}(d)_w)'\arrow[d, "\widetilde{\Delta}_B"]\\
    \bigoplus_{C\in S^B_A} K_T(\mathbb{T}_C)'\arrow[r, "\bigoplus \widehat{m}_{CB}"]& K_T(\mathbb{T}_B)'.
    \end{tikzcd}
\end{equation*}

 \end{cor}


\begin{proof}
For simplicity of notation, we assume that $k=l=2$, that $A$ is the partition $a+b=n$, and that $B$ is the partition $c+e=n$. Then $S:=S^B_A$ is the set of partitions $C=(f_i)_{i=1}^4\in\mathbb{N}^4$ such that 
 \[ f_1+f_2=a,\,
      f_3+f_4=b,\,
     f_1+f_3=c,\,
      f_2+f_4=e.
\]
  Note that, for such a partition $C$, the partition $D$ is $(f_1, f_3, f_2, f_4)$.
  The swap morphism is:
 \begin{align*}
    \text{sw}_{ab}: K_T\left(\mathbb{T}(ad)_{av}\right)'\otimes K_T\left(\mathbb{T}(bd)_{bv}\right)' &\to K_T\left(\mathbb{T}(bd)_{bv}\right)'\otimes K_T\left(\mathbb{T}(ad)_{av}\right)',\\
    x\otimes y&\mapsto y\otimes x.
 \end{align*}
 We abuse notation and write $m_{ab}$ instead of $m_{ad, bd}$ etc. 
We then need to show that the following diagram commutes:
\begin{equation*}
    \begin{tikzcd}
    K_T\left(\mathbb{T}(ad)_{av}\right)'\otimes K_T\left(\mathbb{T}(bd)_{bv}\right)'\arrow[r, "m_{ab}"]\arrow[d, "\Delta"]& K_T\left(\mathbb{T}(nd)_{nv}\right)'\arrow[d, "\Delta_{ce}"]\\
    \bigoplus_{S} \bigotimes_{i=1}^4 K_T\left(\mathbb{T}(f_id)_{f_iv}\right)'\arrow[r, "m'"]& K_T\left(\mathbb{T}(cd)_{cv}\right)'\otimes K_T\left(\mathbb{T}(ed)_{ev}\right)',
    \end{tikzcd}
\end{equation*}
where $m':=\bigoplus_S \left(m_{f_1f_3}\otimes m_{f_2 f_4}\right)\left( 1\otimes \text{sw}_{f_2f_3}\otimes 1\right)$ and $\Delta:=\bigoplus_S \Delta_{f_1f_2}\otimes \Delta_{f_3f_4}$.

For $m\in \mathbb{N}$, recall that $\sigma_m:=m\tau_m=\sum_{i=1}^m \beta_i$. Then $\sigma_{ad}=\sigma_{f_1d}+\sigma_{f_2d}$ etc.
In this proof, we will use the notation $z:=q^{-\gamma}$ instead of $q^{-1}$ to reduce the use of the letter $q$. We use the notation $\Phi_n$ for the Koszul equivalence \eqref{koszuleqev} for $(nd, nv)$.
Then 
\[q^{\nu_{a, b}}=q^{-bd\sigma_{ad}}q^{ad\sigma_{bd}}z^{abd^2},\,\omega_{a,b}=q^{-bd\sigma_{ad}}q^{ad\sigma_{bd}}(-z)^{abd^2}.\]
Fix a tuplet $C=(f_i)_{i=1}^4$ as above and let $w=w_C\in \mathfrak{S}_{nd}$ be its corresponding Weyl element as in Subsection \ref{nu}.
Let $x_m\in K_T\left(\mathbb{T}(md)_{mv}\right)$ for $m\in \{f_i, a, b, n\mid1\leq i\leq 4\}$.
By the discussions in Subsections \ref{subsection211} and \ref{coproductdiml}, we have that 
\begin{align*}
    m_{ab}(x_a\boxtimes x_b)&=
    \Phi_n^{-1}
    \widetilde{m}_{ab}
    \left(
    \Phi_a(x_a)q^{-bd\sigma_{ad}}\boxtimes \Phi_b(x_b)q^{ad\sigma_{bd}}(-z)^{d^2ab}\right),\\
    \Delta_{ab}(x_n)&=
    \Phi_a^{-1}\boxtimes\Phi^{-1}_b
    \left(
    \left(\widetilde{\Delta}_{ab}\Phi_n(x_n)\right)q^{bd\sigma_{ad}}q^{-ad\sigma_{bd}}(-z)^{-d^2ab}
    \right),\\
    \widetilde{\text{sw}}_{C}\left(x_{f_1}\boxtimes x_{f_2}\boxtimes x_{f_3}\boxtimes x_{f_4}\right)&=
    x_{f_1}\boxtimes \left(x_{f_3}q^{2f_3\sigma_{f_2}}\right)\boxtimes \left(x_{f_2}q^{-2f_2\sigma_{f_3}}\right)\boxtimes x_{f_4}.
\end{align*}
We use the shorthand notations
$\text{sw}_{23}=1\boxtimes \text{sw}_{f_2 f_3}\boxtimes 1$,
$\widetilde{m}_{13}=\widetilde{m}_{f_1f_3}$, $\widetilde{m}_C:=\widetilde{m}_{13}\boxtimes \widetilde{m}_{24}$ etc. in what follows.
We compute, by ignoring the $z$ factor for simplicity of notation:
\begin{align}\label{deltam}
    &\Delta_{ce}m_{ab}(x_a\boxtimes x_b)=\\
    &\Phi_c^{-1}\boxtimes\Phi_e^{-1}\left(\widetilde{\Delta}_{cf}\left(\widetilde{m}_{ab}(\Phi_a(x_a)q^{-bd\sigma_{ad}}\boxtimes \Phi_b(x_b)q^{ad\sigma_{bd}})\right)q^{ed\sigma_{cd}}q^{-cd\sigma_{ed}}\right)=\nonumber\\
    &\Phi_c^{-1}\boxtimes \Phi_e^{-1}
    \sum_{S}
    \widetilde{m}_{C}\left(\widetilde{\text{sw}}_{C}\left(
    \widetilde{\Delta}_{12}(\Phi_a(x_a)q^{-bd\sigma_{ad}})\boxtimes \widetilde{\Delta}_{34}(\Phi_b(x_b)q^{ad\sigma_{bd}})
    \right)\right)q^{ed\sigma_{cd}}q^{-cd\sigma_{ed}}.\nonumber
    \end{align}
    We next compute 
    \begin{align}\label{mdelta}
    &\sum_S m'\Delta(x_a\boxtimes x_b)=\\
    &\Phi_c^{-1}\boxtimes\Phi_e^{-1} \sum_{S}
    \left(\widetilde{m}_{13}\boxtimes \widetilde{m}_{24}\right)\left(q^{-f_3d\sigma_{f_1d}}\boxtimes q^{f_1d\sigma_{f_3d}}\boxtimes q^{-f_4d\sigma_{f_2d}}\boxtimes q^{f_2d\sigma_{f_4d}}\right)\nonumber
    \\
    &\text{sw}_{23}
    \left(\left(\widetilde{\Delta}_{12}(\Phi_a(x_a))q^{f_2d\sigma_{f_1d}}q^{-f_1d\sigma_{f_2d}}\right)
    \boxtimes \left(\widetilde{\Delta}_{34}(\Phi_b(x_b))q^{f_4d\sigma_{f_3d}}q^{-f_3d\sigma_{f_4d}}\right)\right).\nonumber
\end{align}    
We claim that the expressions \eqref{deltam} and \eqref{mdelta} are equal. We only match the coefficients in $\mathbb{K}$ corresponding to $f_1$ and $f_2$ as the computations for $f_3$ and $f_4$ are similar:
\begin{align*}
    q^{(f_2+f_4)d\sigma_{f_1d}}q^{-(f_3+f_4)d\sigma_{f_1d}}&=q^{f_2d\sigma_{f_1d}}q^{-f_3d\sigma_{f_1d}},\\
    q^{-(f_1+f_3)d\sigma_{f_2d}}q^{-(f_3+f_4)d\sigma_{f_2d}}\cdot q^{2f_3d\sigma_{f_2d}}&=q^{-f_1d\sigma_{f_2d}}q^{-f_4d\sigma_{f_2d}}.
\end{align*}
Finally, the $z$ factors in $\Delta_{ce}m_{ab}$ and $m'\Delta$ are equal because 
\[(-z)^{d^2ab-d^2ce}=(-z)^{-d^2f_1f_2-d^2f_3f_4+d^2f_1f_3+d^2f_2f_4}.\]
\end{proof}

\subsection{The localized bialgebra}\label{s44}
Recall that for $V$ a $\mathbb{K}$-module, we let $V_\mathbb{F}:=V\otimes_{\mathbb{K}}\mathbb{F}$. 
Recall that $(d, v)\in \mathbb{N}\times\mathbb{Z}$ are coprime.
 Consider the $\mathbb{N}$-graded $\mathbb{F}$-vector space \begin{equation}\label{V}
     V:=K(\mathcal{D}_{d, v})_{\mathbb{F}}=
     \bigoplus_{n\geq 0} K_T\left(\mathbb{T}(nd)_{nv}\right)_{\mathbb{F}}.
 \end{equation}
 We next explain that the operations $m$ and $\Delta$ endow $V$ with the structure of a commutative and cocommutative $\mathbb{F}$-bialgebra. 
We show this by an explicit computation using the generators of these vector spaces, see \eqref{basis:S'} and \eqref{SS'}. Recall the complex $\mathcal{E}_{e, v}$ from Definition \ref{definition:Edw} and the shuffle elements
$A'_{e,v}$ and $A_{e, v}$ from \eqref{Ambullet} and \eqref{def:Ambullet2} for a pair $(e, v)\in\mathbb{N}\times\mathbb{Z}$.

\begin{prop}\label{commutative}
   Let $a, b\in \mathbb{N}$ with $a+b=n$.
   Then $[\mathcal{E}_{ad, av}]\cdot[\mathcal{E}_{bd, bv}]=[\mathcal{E}_{bd, bv}]\cdot[\mathcal{E}_{ad, av}]$ in $K_T\left(\mathbb{T}(nd)_{nv}\right)$.
\end{prop}

\begin{proof}
    By \eqref{basis:S'} and \eqref{SS'}, it suffices to check the statement for $A_{ad, aw}$ and $A_{bd, bw}$, or alternatively for $A'_{ad, aw}$ and $A'_{bd, bw}$ for $a, b\in\mathbb{N}$. 
    Any two such elements commute because they are in the subalgebra of $\mathcal{S}'_{\mathbb{F}}\cong \mathcal{S}_{\mathbb{F}}$ generated by elements $A'_{e, v}$ of fixed slope $\frac{v}{e}$, and such algebra are commutative, see for example \cite[Subsection 3.2]{Ne}. 
\end{proof}

For the following proposition, it is convenient to introduce the element 
\[\hat{A}_{nd, nv}:=\left(-q^{-1}\right)^{n-1}A_{nd, nv}=
\frac{\left(-q^{-1}\right)^{n-1}}{(1-q_1^{-1})^{nd-1}(1-q_2^{-1})^{nd-1}}[\mathcal{E}_{nd, nv}]\in K_T(\mathbb{T}(nd)_{nv})_\mathbb{F}.\]

\begin{prop}\label{1236bis}
Let $a, b, n\in\mathbb{N}$ be such that $a+b=n$. Then 
    \[\Delta_{ab}\left(\hat{A}_{nd, nv}\right)=
    \hat{A}_{ad, av}\boxtimes \hat{A}_{bd, bv}.\]
\end{prop}

For $d, v, n$ as above, let $O_n$ and $L_n$
be the following sets of $(T\times T(d))$-weights:
\begin{align*}
O_n&:=\{\beta_j-\beta_i+\gamma \mid \,i>j+1\}, \\
    L_n&:=\{\beta_i-\beta_j+\gamma_l \mid \,i>j,\, 1\leq l\leq 2\}\sqcup O_n.
\end{align*}
For a subset $I\subset L_n$, let $\sigma_I$ be the sum of the corresponding $T(d)$-weights in $I$, let $\gamma_I$ be the sum of the weights of the corresponding $T$-weights in $I$, and let $q_I:=q^{\gamma_I}$.
Let $\ell(I)$ be the length of the minimal Weyl element such that $w*(\chi-\sigma_I)$ is dominant or zero. 
Let 
 \[\chi_n:=\sum_{i=1}^{nd-1}\left(\frac{vi}{d}+1-\Big\lceil\frac{vi}{d}\Big\rceil\right)(\beta_{i+1}-\beta_i)+\frac{v}{d}\sum_{i=1}^{nd}\beta_i\in M(nd).\]
The element $A_{nd, nv}\in K_T\left(\X(nd)\right)$ from \eqref{def:Ambullet2} can be written as
\[A_{nd, nv}:=\sum_{I\subset L_n}(-1)^{|I|-\ell(I)}q_I^{-1}\left[\Gamma_{GL(nd)}\left((\chi_n-\sigma_I)^+\right)\otimes\mathcal{O}_{\X(nd)}\right].\]
Recall the framework on Subsection \ref{coproductdiml}.
Consider the composition of the Koszul equivalence \eqref{equiv:Phi} and the forget-the-potential map \eqref{definition:forgetpotential}:
\[\Psi: K_T(\mathbb{T}(nd)_{nv})_{\mathbb{F}}\xrightarrow{\sim} K_T(\mathbb{S}(nd)_{nv})_{\mathbb{F}}\xrightarrow{\Theta} K_T(\mathrm{MF}(\X(d),0))_{\mathbb{F}}\xrightarrow{\sim} K_T\left(\X(nd)\right)_{\mathbb{F}}.\] 
Using \eqref{elem:E} and the computation in \cite[Lemma 4.3]{PT0}, we have that \[\Psi\left((-q)^{n-1}\hat{A}_{nd, nv}\right)=A_{nd, nv}.\]
\begin{remark}
Note that $\Psi\left(\hat{A}_{nd, nv}\right)$ equals the shuffle element $E_{k, d}$ considered by Negu\c{t} in \cite[Equation 2.10]{Negut3}, where $q$, $k$, $d$ in loc. cit. correspond to $q^{-1}$, $nd$, and $nv$, respectively, in our paper.\end{remark}

\begin{proof}[Proof of Proposition \ref{1236bis}]

By the definition of $\Delta$ from Subsection \ref{coproductdiml}, 
it suffices to show that
\begin{equation}\label{deltaab}
\widetilde{\Delta}_{ab}\left(A_{nd, nv}\right)=\left(A_{ad, av}\boxtimes A_{bd, bv}\right)\otimes (-q)(-1)^{abd^2}q^{\nu_{a, b}},
\end{equation}
where we are using $\widetilde{\Delta}_{ab}$ instead of $\widetilde{\Delta}_{ad, bd}$.
Let $I\subset L_n$. Then 
\[\chi_n-\sigma_I+\rho=\sum_{i>j}d_{ij}(\beta_i-\beta_j)+\frac{v}{d}\sum_{i=1}^{nd}\beta_i,\] where 
\begin{align*}
    -\frac{3}{2}\leq d_{ij}\leq \frac{3}{2},\ 
    -\frac{3}{2}<d_{j+1, j}\leq \frac{3}{2}.
\end{align*}
Let $\lambda$ be the antidominant cocharacter associated to the partition $(ad, bd)$ of $nd$.
 Assume such a weight is on a wall $F\left(w\lambda\right)$ for some $w\in \mathfrak{S}_{nd}$. Then there exists a partition 
$E\sqcup C=\{1,\cdots, nd \}$
with $|C|=ad$ and 
\begin{equation}\label{dij32}
    d_{ij}=\begin{cases} \frac{3}{2}\text{ for }i\in E, j\in C\text{ and }i>j,\\
    -\frac{3}{2}\text{ for }i\in E, j\in C\text{ and }i<j,
    \end{cases}
\end{equation}
see \cite[Lemma 3.12]{hls}, \cite[Proposition 3.2]{P}.
We claim that $C=\{1,\ldots, ad\}$. Otherwise, there exists $1\leq j\leq ad$ with $j\in E$ and $j+1\in C$. Then $d_{j+1, j}>-\frac{3}{2}$ and this contradicts \eqref{dij32}. Then $d_{ij}=-\frac{3}{2}$ for $i>ad\geq j$. Further, 
we have 
\begin{align*}
    I \subset Q:=\{\beta_j-\beta_i+\gamma \mid i>ad\geq j, i-1>j\}
\end{align*}
and $I$ does not contain any weights $\beta_i-\beta_j-\gamma_l$ for $i>ad\geq j$ and $l\in\{1, 2\}$. Define $L_a$
and $L_b$\
similarly to $L_n$, 
using the weights $\beta_i$ with $1\leq i\leq ad$ for $L_a$ and $\beta_i$ with $ad<i\leq nd$ for $L_b$. 
We can thus write $I=Q\sqcup I_a\sqcup I_b$, where $I_a\subset L_a$ and $I_b\subset L_b$. Let $\mathcal{I}$ be the set of such sets $I\subset L_n$.
We have that, see \eqref{computationdeltab}: 
\begin{equation}\label{coproducttilde}
    \widetilde{\Delta}_{ab}
\left(A_{nd, nv}\right)=\widetilde{\Delta}_{ab}\left(\sum_{I\in\mathcal{I}}
(-1)^{|I|-\ell(I)}q_I^{-1}\left[\Gamma_{GL(nd)}\left((\chi_n-\sigma_I)^+\right)\otimes\mathcal{O}_{\X(nd)}\right]\right).
\end{equation}
Write $I=Q\sqcup I_a\sqcup I_b$ with $I_a\subset L_a$ and $I_b\subset L_b$. A direct computation shows that 
\[\chi_n-\sigma_I+\rho=\sum_{i>ad\geq j}\frac{3}{2}(\beta_i-\beta_j)-(d^2ab-1)\gamma
+
(\chi_a-\sigma_{I_a}+\rho_a)+(\chi_b-\sigma_{I_b}+\rho_b).\]
Let $w\in\mathfrak{S}_{nd}$ be the Weyl element of minimal length such that $w*(\chi_n-\sigma_I)$ is dominant, and let $w_a\in \mathfrak{S}_{na}$ and $w_b\in\mathfrak{S}_{nb}$ be the Weyl elements of minimal length such that
$w_a*(\chi_a-\sigma_{I_a})$ and $w_b*(\chi_b-\sigma_{I_b})$ are dominant. 
Then $w=w_a w_b$,
therefore
\[(\chi_n-\sigma_I)^+=\nu_{a, b}+\gamma+
(\chi_a-\sigma_{I_a})^++
(\chi_b-\sigma_{I_b})^+.\]
We have that $|Q|=d^2ab-1$, and 
\begin{align*}
    |I|=d^2ab-1+|I_a|+|I_b|,\
    \ell(I)=\ell(I_a)+\ell(I_b).
\end{align*}
For $m\in\{a, b, n\}$, let
\[E_m:=(-1)^{|I_m|-\ell(I_m)}\left[\Gamma_{GL(md)}\left((\chi_m-\sigma_{I_m})^+\right)\otimes
\mathcal{O}_{\X(md)}\right].\] 
The coproduct \eqref{coproducttilde} thus simplifies to 
\[\widetilde{\Delta}_{ab}(E_n)=(-q)(-1)^{d^2ab} q^{\nu_{a,b}}\left(E_a \boxtimes E_b\right).\] The conclusion thus follows. 
\end{proof}

\begin{cor}\label{corV}
    The operations $m$ and $\Delta$ endow $V$ with the structure of an $\mathbb{N}$-graded commutative and cocommutative $\mathbb{F}$-bialgebra.
\end{cor}

\subsection{Primitive elements}\label{subsection:primitive}

Let $R'_n\subset R_n$ be the complement of the trivial partition $n$. 
For $A=(n_i)_{i=1}^k$, let $\mathbb{T}_A:=\otimes_{i=1}^k \mathbb{T}(n_id)_{n_iv}$.
Define
\[\mathrm{P}(nd)_{nv}:=\text{ker}\left(\bigoplus_{A\in R'_n} \Delta_A: K_T\left(\mathbb{T}(nd)_{nv}\right)\to \bigoplus_{A\in R'_n} K_T\left(\mathbb{T}_A\right)\right).\]
Let $\mathrm{P}(nd)_{nv, \mathbb{F}}:=\mathrm{P}(nd)_{nv}\otimes_{\mathbb{K}}\mathbb{F}$. 
Corollary \ref{corV} can be rephrased as follows, see the isomorphism \eqref{PBWmacdonald}:
\begin{cor}\label{isolambda}
Recall the bialgebra $\Lambda_\mathbb{F}:=\lambda\otimes_\mathbb{Z}\mathbb{F}$ from Subsection \eqref{subsection:symmetric}.
    There exists an isomorphism of bialgebras
    \begin{align*}
    \Phi \colon \Lambda_{\mathbb{F}}&\cong V,\\
    e_n&\mapsto\widehat{A}_{nd, nv}.
    \end{align*}
    In particular, the $\mathbb{F}$-vector space $\mathrm{P}(nd)_{nv, \mathbb{F}}$ is one dimensional.
\end{cor}

\begin{proof}
    Both $\Lambda_{\mathbb{F}}$ and $V$ are commutative and cocommutative, and $\Phi$ is a morphism of algebras by construction. The coproduct is respected by Proposition \ref{1236bis}. Finally, $\Phi$ is an isomorphism of $\mathbb{N}$-graded vector spaces by \cite[Theorem 4.12]{PT0}, so the conclusion follows.
\end{proof}


\begin{remark}
The isomorphism $\Phi$
sends $e_n$ to $\widehat{A}_{nd, nv}$. 
Since $\widehat{A}_{nd, nv}$ is not contained in the integral part
$K_T(\mathbb{T}(nd)_{nv})$, it does not restrict to 
a morphism $\Lambda_{\mathbb{K}} \to K_T(\mathcal{D}_{d, v})$. 
On the other hand, we expect the existence of 
a McKay type functor 
$\mathrm{MF}^{\mathrm{gr}}([\mathbb{C}^{3d}/\mathfrak{S}_d], 0) \to \mathbb{T}(nd)_{nv}$
which may induce an isomorphism \[\Lambda_{\mathbb{K}} \stackrel{\cong}{\to} K_T(\mathcal{D}_{d, v}).\]
For a full understanding of $K_T(\mathbb{T}(nd)_{nv})$, we thus need to construct an isomorphism different from the one in Corollary~\ref{isolambda}. 
Such an isomorphism will be discussed in~\cite{PTsym}.
\end{remark}

Finally, we prove 
Corollary~\ref{intro:corDT}: 
\begin{cor}\label{corDT}
There is an isomorphism of $\mathbb{N}$-graded $\mathbb{F}$-vector spaces:
\begin{equation}\notag
\bigoplus_{d\geq 0}K_T(\mathcal{DT}(d))_{\mathbb{F}}\cong \bigotimes_{\substack{0\leq v<d \\ \gcd(d,v)=1}} \left(\bigotimes_{n\geq 1}\mathrm{Sym}\big(\mathrm{P}(nd)_{nv, \mathbb{F}}\big)\right).
\end{equation}
\end{cor}
\begin{proof}
For $\bullet\in\{\emptyset, \mathrm{gr}\}$, there are equivalences $\mathbb{S}^\bullet_T(d)_w\xrightarrow{\sim}\mathbb{S}^\bullet_T(d)_{d+w}$. By \cite[Theorem 1.1]{PT0}, there is an isomorphism:
\begin{equation}\label{5131}
K_T(\mathcal{DT}(d))_{\mathbb{F}}\cong \bigoplus_{\substack{0\leq v_1/d_1<\ldots<v_k/d_k<1\\d_1+\ldots+d_k=d}} \bigotimes_{i=1}^k K_T\left(\mathbb{S}(d_i)_{v_i}\right)_\mathbb{F}.
\end{equation}
Further, there are isomorphisms of $\mathbb{N}$-graded $\mathbb{F}$-vector spaces:
\begin{equation}\label{5132}
\bigoplus_{\substack{0\leq v_1/d_1<\ldots<v_k/d_k<1}} \bigotimes_{i=1}^k K_T\left(\mathbb{S}(d_i)_{v_i}\right)_\mathbb{F}\cong \bigotimes_{\substack{0\leq \mu<1 \\
\mu=a/b, \ \gcd(a, b)=1}}\left(\bigoplus_{n\ge 1}K_T\left(\mathbb{S}(nb)_{na}\right)_\mathbb{F}\right).
\end{equation}
For each coprime $(a, b) \in \mathbb{Z} \times \mathbb{N}$, 
by Corollary \ref{isolambda} and the isomorphism \eqref{PBWmacdonald}, we obtain the isomorphism: 
\begin{equation}\label{5133}
\bigoplus_{n\ge 1}K_T\left(\mathbb{S}(nb)_{na}\right)_\mathbb{F}
\cong  \bigotimes_{n\geq 1}\mathrm{Sym}\left(\mathrm{P}(nb)_{na, \mathbb{F}}\right).
\end{equation}
We obtain the conclusion by combining the isomorphisms \eqref{5131}, \eqref{5132}, \eqref{5133}. 
\end{proof}

We explain how the above isomorphism categorifies \eqref{wallcrossingDT} up to a sign.
We use the same computation as in \cite[Subsection 4.7]{PT0}. Let $a_d:=\dim_{\mathbb{F}} K_T(\mathcal{DT}(d))_{\mathbb{F}}$. We have that $\dim_\mathbb{F} \mathrm{P}(nb)_{na, \mathbb{F}}=1$ for any $(a,b)$ with $\gcd(a,b)=1$ and $n\in\mathbb{Z}_{\geq 1}$. 
For each $d \in \mathbb{Z}_{\geq 1}$, there is a bijection 
	\begin{align*}
		\{(n, a, b) \in \mathbb{Z}_{\geq 0}^3 : d=bn, \gcd(a, b)=1, 0\leq a<b\}
		\stackrel{\cong}{\to}
		\{0, 1, \ldots, d-1\}
		\end{align*}
	given by $(n, a, b) \mapsto na$. In particular, the number of 
	the elements of the left hand side equals to $d$. 
	We compute
	\[
	\sum_{d\geq 0} a_d q^d= 
\prod_{\begin{subarray}{c}
				0\leq \mu<1 \\
		\mu=a/b, \ \gcd(a, b)=1
		\end{subarray}	
		}
	\prod_{n\geq 1}\frac{1}{1-q^{bn}}=\prod_{d\geq 1}\frac{1}{\left(1-q^d\right)^d}.\]
Compare with the wall-crossing formula for DT invariants \eqref{wallcrossingDT}.

	\bibliographystyle{amsalpha}
\bibliography{math}

\providecommand{\bysame}{\leavevmode\hbox to3em{\hrulefill}\thinspace}
\providecommand{\MR}{\relax\ifhmode\unskip\space\fi MR }
\providecommand{\MRhref}[2]{%
  \href{http://www.ams.org/mathscinet-getitem?mr=#1}{#2}
}
\providecommand{\href}[2]{#2}
\begin{thebibliography}{BLdB10}

\bibitem[AG15]{AG}
D.~Arinkin and D.~Gaitsgory, \emph{Singular support of coherent sheaves and the
  geometric {L}anglands conjecture}, Selecta Math. (N.S.) \textbf{21} (2015),
  no.~1, 1--199.

\bibitem[BLdB10]{BLBergh}
R-O. Buchweitz, G.~J. Leuschke, and M.~Van den Bergh, \emph{Non-commutative
  desingularization of determinantal varieties {I}}, Invent. Math. \textbf{182}
  (2010), no.~1, 47--115.

\bibitem[Dava]{D}
B.~Davison, \emph{{BPS} {L}ie algebras and the less perverse filtration on the
  preprojective {C}o{HA}}, arXiv:2007.03289.

\bibitem[Davb]{Dav}
\bysame, \emph{The integrality conjecture and the cohomology of preprojective
  stacks}, arXiv:1602.02110.

\bibitem[Davc]{Davi}
\bysame, \emph{Purity and $2$-{C}alabi-{Y}au categories}, arXiv:2106.07692.

\bibitem[DM20]{DM}
B.~Davison and S.~Meinhardt, \emph{Cohomological {D}onaldson-{T}homas theory of
  a quiver with potential and quantum enveloping algebras}, Invent. Math.
  \textbf{221} (2020), no.~3, 777--871.

\bibitem[Eis95]{Ebud}
D.~Eisenbud, \emph{Commutative algebra: with a view toward algebraic geometry},
  Graduate Texts in Mathematics, vol. 150, Springer-Verlag, New York, 1995.

\bibitem[Hir17]{Hirano}
Y.~Hirano, \emph{Derived {K}n\"orrer periodicity and {O}rlov's theorem for
  gauged {L}andau-{G}inzburg models}, Compos. Math. \textbf{153} (2017), no.~5,
  973--1007.

\bibitem[HLS20]{hls}
D.~Halpern-Leistner and S.~V. Sam, \emph{Combinatorial constructions of derived
  equivalences}, J. Amer. Math. Soc. \textbf{33} (2020), no.~3, 735--773.

\bibitem[Isi13]{I}
M.~U. Isik, \emph{Equivalence of the derived category of a variety with a
  singularity category}, Int. Math. Res. Not. IMRN (2013), no.~12, 2787--2808.

\bibitem[JS12]{JS}
D.~Joyce and Y.~Song, \emph{A theory of generalized {D}onaldson-{T}homas
  invariants}, Mem. Amer. Math. Soc. \textbf{217} (2012), no.~1020, iv+199.

\bibitem[KT21]{KoTo}
N.~Koseki and Y.~Toda, \emph{Derived categories of {T}haddeus pair moduli
  spaces via d-critical flips}, Adv. Math. \textbf{391} (2021), Paper No.
  107965, 55.

\bibitem[KV]{KaVa2}
M.~Kapranov and E.~Vasserot, \emph{The cohomological {H}all algebra of a
  surface and factorization cohomology}, arXiv:1901.07641.

\bibitem[Mac79]{MacDonald}
I.G. MacDonald, \emph{Symmetric functions and {H}all polynomials}, Oxford
  Mathematical Monographs, The Clarendon Press, Oxford University Press, New
  York, 1979.

\bibitem[Neg14]{Neshuffle}
A.~Negu\c{t}, \emph{The shuffle algebra revisited}, Int. Math. Res. Not. IMRN
  (2014), no.~22, 6242--6275.

\bibitem[Neg18]{Negut3}
\bysame, \emph{The {q}-{AGT}-{W} relations via shuffle algebras}, Comm. Math.
  Phys. \textbf{358} (2018), 101–170, arXiv:1608.08613.

\bibitem[Neg19]{Ne}
\bysame, \emph{Shuffle algebras associated to surfaces}, Selecta Math. (N.S.)
  \textbf{25} (2019), no.~3, Art. 36, 57.

\bibitem[Neg22]{N}
\bysame, \emph{{$W$}-algebras associated to surfaces}, Proc. Lond. Math. Soc.
  (3) \textbf{124} (2022), no.~5, 601--679.

\bibitem[Neg23]{N2}
\bysame, \emph{Shuffle algebras for quivers and wheel conditions}, J. Reine
  Angew. Math. \textbf{795} (2023), 139--182.

\bibitem[P{\u{a}}da]{P}
T.~P{\u{a}}durariu, \emph{Generators for {K}-theoretic {H}all algebras of
  quivers with potential}, arXiv:2108.07919.

\bibitem[P{\u{a}}db]{P3}
\bysame, \emph{Noncommutative resolutions and intersection cohomology for
  quotient singularities}, arXiv:2103.06215.

\bibitem[P{\u{a}}d22]{P0}
\bysame, \emph{Categorical and {K}-theoretic {H}all algebras for quivers with
  potential}, J.~Inst.~Math.~Jussieu (2022), 1--31.

\bibitem[P{\u{a}}d23]{P2}
\bysame, \emph{Generators for {H}all algebras of surfaces}, Math. Z.
  \textbf{303} (2023), no.~2, Paper No. 40, 25.

\bibitem[PS23]{PoSa}
M.~Porta and F.~Sala, \emph{Two-dimensional categorified {H}all algebras}, J.
  Eur. Math. Soc. (JEMS) \textbf{25} (2023), no.~3, 1113--1205.

\bibitem[PTa]{PT0}
T.~P\u{a}durariu and Y.~Toda, \emph{Categorical and {K}-theoretic
  {D}onaldson-{T}homas theory of $\mathbb{C}^3$ (part {I})}, arXiv:2207.01899.

\bibitem[PTb]{PTsym}
\bysame, \emph{Symmetric powers and the {H}ilbert scheme of points on
  $\mathbb{C}^3$}, in preparation.

\bibitem[PT14]{MR3221298}
R.~Pandharipande and R.~P. Thomas, \emph{13/2 ways of counting curves}, Moduli
  spaces, London Math. Soc. Lecture Note Ser., vol. 411, Cambridge Univ. Press,
  Cambridge, 2014, pp.~282--333. \MR{3221298}

\bibitem[Sch12]{S}
O.~Schiffmann, \emph{Lectures on {H}all algebras}, Geometric methods in
  representation theory. {II}, S\'{e}min. Congr., vol.~24, Soc. Math. France,
  Paris, 2012, pp.~1--141.

\bibitem[{\v S}dB17]{SVdB}
{\v S}.~{\v S}penko and M.~Van den Bergh, \emph{Non-commutative resolutions of
  quotient singularities for reductive groups}, Invent. Math. \textbf{210}
  (2017), no.~1, 3--67.

\bibitem[SV13]{SV}
O.~Schiffmann and E.~Vasserot, \emph{The elliptic {H}all algebra and the
  {$K$}-theory of the {H}ilbert scheme of {$\mathbb{A}^2$}}, Duke Math. J.
  \textbf{162} (2013), no.~2, 279--366.

\bibitem[Toda]{T}
Y.~Toda, \emph{Categorical {D}onaldson-{T}homas theory for local surfaces},
  arXiv:1907.09076.

\bibitem[Todb]{T3}
\bysame, \emph{Categorical wall-crossing formula for {D}onaldson-{T}homas
  theory on the resolved conifold}, arXiv:2109.07064, to appear in Geometry and
  Topology.

\bibitem[Tod10]{T5}
\bysame, \emph{Curve counting theories via stable objects~{I}: {DT/PT}
  correspondence}, J.~Amer.~Math.~Soc.~ \textbf{23} (2010), 1119--1157.

\bibitem[Tod23]{T4}
\bysame, \emph{Categorical {D}onaldson--{T}homas theory for local surfaces:
  {$\mathbb{Z}$}/2-periodic version}, Int. Math. Res. Not. IMRN (2023), no.~13,
  11172--11216.

\bibitem[VV22]{VaVa}
M.~Varagnolo and E.~Vasserot, \emph{K-theoretic {H}all algebras, quantum groups
  and super quantum groups}, Selecta Math. (N.S.) \textbf{28} (2022), no.~7, 46
  pages.

\bibitem[Zha]{Zhao}
Y.~Zhao, \emph{The {F}eigin-{O}desskii wheel conditions and sheaves on
  surfaces}, arXiv:1909.07870.

\end{thebibliography}

\textsc{\small Tudor P\u adurariu: Columbia University, 
Mathematics Hall, 2990 Broadway, New York, NY 10027, USA.}\\
\textit{\small E-mail address:} \texttt{\small tgp2109@columbia.edu}\\

\textsc{\small Yukinobu Toda: Kavli Institute for the Physics and Mathematics of the Universe (WPI), University of Tokyo, 5-1-5 Kashiwanoha, Kashiwa, 277-8583, Japan.}\\
\textit{\small E-mail address:} \texttt{\small yukinobu.toda@ipmu.jp}\\

\end{document}